\documentclass[final]{siamltex}

\usepackage{latexsym,amsmath,amssymb,amsfonts, amscd}
\usepackage{multirow}
\usepackage{epsfig,verbatim,epstopdf,graphics}
\usepackage{color}
\usepackage{array}
\usepackage[ruled,linesnumbered]{algorithm2e}
\usepackage{subfigure}
\usepackage[bookmarks=true]{hyperref}
\usepackage{tikz-cd}
\usepackage{algorithmic}
\allowdisplaybreaks

\newcommand{\Omen}{{\Omega_{n}}}

\newcommand{\bi}{\mathbf{i}}
\newcommand{\bj}{\mathbf{j}}

\newcommand{\bl}{\mathbf{l}}
\newcommand{\bn}{\mathbf{n}}

\newcommand{\bx}{\mathbf{x}}

\newcommand{\bJ}{\mathbf{J}}
\newcommand{\bM}{\mathbf{M}}
\newcommand{\bP}{\mathbf{P}}
\newcommand{\bV}{\mathbf{V}}

\newcommand{\cI}{\mathcal{I}}

\newcommand{\YJ}[1]{{\color{blue}{\bf YJ:} #1}}

\newtheorem{thm}{Theorem}
\newtheorem{lem}{Lemma}
\newtheorem{rmk}{Remark}



\title
{An adaptive  high-order  piecewise polynomial based sparse grid collocation method   with  applications}

\author{
Zhanjing Tao
\thanks{Department of Mathematics, Jilin University,
Changchun, Jilin 130012, People’s Republic of China.
 {\tt zjtao@jlu.edu.cn}}.
\and
 Yan Jiang
\thanks{School of Mathematical Sciences, University of Science and Technology of China, Hefei, Anhui 230026, People’s Republic of China
{\tt jiangy@ustc.edu.cn}. Research is supported by NSFC grants 11901555 }
\and
Yingda Cheng
\thanks{Department of Mathematics, Department of  Computational Mathematics, Science and Engineering, Michigan State University,
East Lansing, MI 48824 U.S.A.
 {\tt ycheng@msu.edu}. Research is supported by NSF grants  DMS-1453661 and DMS-1720023.}
}

\date{\today}

\begin{document}
\maketitle

\begin{abstract}
This paper constructs adaptive sparse grid collocation method   onto  arbitrary order piecewise polynomial space.  The sparse grid 
method is a popular technique for high dimensional problems, and the associated collocation method has been well studied in the literature. The contribution of this work is
the introduction of a systematic framework for collocation onto high-order piecewise polynomial space that is allowed to be discontinuous. We consider both Lagrange and Hermite interpolation methods on nested collocation points. Our construction includes a wide range of function space, including those used in sparse grid continuous finite element method. 
Error estimates are provided, and the numerical results in function interpolation, integration and some benchmark problems in uncertainty quantification are used to  compare different collocation schemes.
\end{abstract}


\section{Introduction}
\label{sec:intro}

In this paper, we  consider the discretization of high dimensional problems using the sparse grid method \cite{smolyak1963quadrature, bungartz2004sparse, garcke2013sparse}, which  is a popular technique to reduce the number of degrees of freedom (DoF)  based on a tensor product hierarchical basis representation.  The specific objective of this work is to introduce a class of high order  hierarchical interpolating basis, and develop adaptive sparse grid collocation methods onto   piecewise polynomial space, as a continued effort in   developing sparse grid discontinuous Galerkin (DG) methods.
In our previous work \cite{sparsedgelliptic, guo2016sparse,guo2017adaptive}, we introduced the  sparse grid DG methods for  solving high dimensional PDEs, and show that the methods have significantly  reduced DoFs for the unknowns, and can maintain similar stability and conservation properties of DG methods using the Galerkin framework. Adaptivity can be  incorporated naturally treating solutions with less smoothness or local structures. However, the methods can not be efficiently applied for truly nonlinear problems. Examples include representation of a nonlinear function $f(u_h)$ onto the sparse grid finite element space on the computational domain. Even if $u_h$ belongs to the sparse grid space, $f(u_h)$ is not,  therefore a naive implemenation requires computational cost that is proportional to the number of elementary cells, i.e. $O(h^{-d})$ operations,   where $d$ is the number of dimensions, and $h$ is the mesh size in each dimension.

Collocation is a natural way to treat nonlinearity.  (Adaptive) sparse grid collocation  are well developed \cite{barthelmann2000high,xiu2005high, nobile2008sparse, ma2009adaptive}, and the most popular nested collocation methods include the ones based on Newton-Cotes and  Clenshaw-Curtis points.   Clenshaw-Curtis rule uses spectral (Chebyshev) approximations, while  Newton-Cotes rule is based on  continuous piecewise linear polynomial approximations and is more flexible with adaptivity.   We are interested in piecewise higher order polynomial approximations. In this context,    \cite{bungartz1996concepts} proposed a $P^2$ continuous sparse grid FEM bases. \cite{alpert2002adaptive} contains a non-nested collocation method on Gauss-Legendre quadrature points. 
Higher order continuous piecewise polynomial basis was constructed in \cite{bungartz1998finite}, see also  \cite{jakeman2012local} for its application in uncertainty quantification (UQ) and \cite{bungartz2004sparse} for a summary and extended reference. However, the high order Lagrangian interpolant constructed in \cite{bungartz1998finite} is not local, i.e. nodes that are outside of the current cell need to be used, thus the bases at the coarsest level are lower order accurate. 
It is also well known that the construction of sparse grid relies on  multiresolution analysis (MRA)  \cite{mallat1999wavelet} as building blocks. Relevant work in the literature  includes  interpolating wavelet transform 
\cite{donoho1992interpolating} and interpolatory MRA  which is considered by Harten \cite{harti1993discrete, harten1995multiresolution,harten1996multiresolution}  for PDE applications. The interpolation operator in \cite{harten1996multiresolution} includes a variety of discretizations, such as polynomial, spline, ENO and trigonometric interpolations. 

The focus of this work  is to construct adaptive sparse grid collocation methods based on interpolatory MRA onto  arbitrary order piecewise polynomial space with nested collocation points. Here, our interpolation will be local in contrast to \cite{bungartz1998finite, harti1993discrete}, meaning that we do not need to extend the stencils when we go higher order. Instead, more DoFs are placed inside the current cell like the DG method. This results in a more compact scheme. We require the collocation points to be nested to save computational cost. No continuity across the cell interface   is assumed. However, the methods we propose will include  the family of  sparse grid continuous finite element space. We consider  Lagrange and also Hermite interpolation as in \cite{stortkuhl1994numerisches, warming2000discrete}. Our construction is  systematic, and works for arbitrary high order accuracy. In particular, we follow the following steps: (1) locating nested interpolation points, (2) finding associated multiwavelet bases in 1D, (3) using Smolyak's idea to gain sparsity in high dimensions, (4) achieving adaptivity by measuring hierarchical surplus. Fast transforms between point values and coefficients are introduced with operation counts of $O(d\cdot \textrm{DoF} ).$ Theoretical justification will be provided, and applications in stochastic  differential equations are considered. We note that many recent work in UQ has considered more efficient  collocation schemes in higher dimensions and for functions with singularities \cite{wan2005adaptive, foo2008multi, foo2010multi, ma2010adaptive, bhaduri2018stochastic}. Those techniques are not explored in this work.  Rather, our main  motivation  is the design of adaptive multiresolution DG methods, which  is considered separately in  another work \cite{huang2019adaptive}.

The rest of the paper is organized as follows: in Section \ref{sec:method}, we introduce one-dimensional MRA on piecewise polynomial space. Section \ref{sec:multid} contains the discussion of multi-dimensional sparse grid and adaptive sparse grid collocation schemes. The methods are validated numerically in Section \ref{sec:numerical}. We conclude the paper in Section \ref{sec:conclusion}.

\section{One-dimensional interpolatory MRA}
\label{sec:method}


In this section, we  introduce the nested grids and collocation points, and the corresponding hierarchical bases in one dimension. Without loss of generality, we consider the interval $I=[0,1]$. A multi-resolution interpolation method  will be introduced. 

\subsection{Nested collocation points}
To begin, we define a set of {\em nested grids}, where the $n$-th level grid $\Omega_{n}$ consists of $2^n$ uniform cells $I_{n}^{j}= \left( 2^{-n}j, 2^{-n}(j+1) \right)$, $j=0, \ldots, 2^n-1$, for any $n\geq 0$
with cell size $h_{n}=2^{-n}$.
We can define $P+1$ distinct points within each cell of $\Omen$ with the same relative locations, 
\begin{align}
\label{eq:point}
	x^{j}_{i,n} = 2^{-n} (j + x^{0}_{i,0}) 
\end{align}
where $x^{0}_{i,0}\in[0,1]$, $i=0, \ldots, P,$ and $i$ numbers the relative location of the points within the cell. In this paper, we consider  general functions that are supported on the grid $\Omega_N$ and are allowed to be discontinuous at the interface of $\Omega_N,$ where $N$ is a prescribed integer. This is particularly needed for the implementation of multiresolution DG scheme \cite{huang2019adaptive}. In definition \eqref{eq:point}, if the point $x^{j}_{i,n}$ lands on the interface of $\Omega_N,$ it should be defined either as the left or right limit point.
In particular, the collection of those points 
\begin{align}
	X^{P}_n=\{x^{j}_{i,n}, \, i=0,\ldots, P, \, j=0, \ldots, 2^n-1 \}
\end{align}
is called {\em nested points}, if 
\begin{align}
\label{eq:nested_grids2}
	X^{P}_0 \subset X^{P}_{1} \subset X^{P}_{2} \subset \cdots.
\end{align}
This means, for any point $x^{j}_{i,n-1}\in X^{P}_{n-1}$, we can always find an integer $r\in\{0,\ldots,P\}$, such that
\begin{align}
\label{eq:nested_relation1}
	x_{i,n-1}^{j} = x_{r,n}^{2j} \quad \text{or} \quad x_{i,n-1}^{j} = x_{r,n}^{2j+1}.
\end{align} 
The choice of $\{x^{0}_{i,0}\}$ always exists so that \eqref{eq:nested_grids2} can be satisfied. We have the following lemma to quantify the total number of possible choices for each $P.$

\begin{lem} \label{lem1}
	For any $P \geq 0$, we have $ \begin{pmatrix} 2P+2 \\ P+1\\	\end{pmatrix} - 2  \begin{pmatrix} 2P \\ P-1\\	\end{pmatrix} + \begin{pmatrix} 2P-2 \\ P-3\\	\end{pmatrix} $ types of   nested points. Here,
	$$ \begin{pmatrix} N \\ n\\	\end{pmatrix} = \left\{ \begin{array}{ll} 
	\frac{N!}{n! (N-n)!}, &  0<n<N,\\
	1, & n=0, N,\\
	0, & n<0.\\
	\end{array}\right. $$
\end{lem}
\begin{proof}
The proof is given in Appendix \ref{sec:append2}.
\end{proof}

We should note that when $\{X^{P}_{n}\}$ are nested, the points can be rearranged in such a way that
\begin{align}
\label{eq:nested_grids}
	X^{P}_{n} = X^{P}_{0} \cup \widetilde{X}^{P}_{1} \cup \cdots \cup \widetilde{X}^{P}_n, \quad\text{with} \ \widetilde{X}^{P}_n= X^{P}_n / X^{P}_{n-1} .
\end{align}
We can denote the points in $\widetilde{X}^{P}_{1} = X^{P}_1 / X^{P}_{0} = \{ \tilde{x}^{0}_{0,1}, \ldots, \tilde{x}^{0}_{P,1} \}$, then the points in $\widetilde{X}^{P}_n$ for $n\geq 1$ can be represented by
\begin{align}
	\widetilde{X}^{P}_n = \{ \tilde{x}^{j}_{i,n}:= 2^{-(n-1)} (j+\tilde{x}^{0}_{i,1}),  \, i=0,\ldots, P, \, j=0, \ldots, 2^{n-1}-1 \}.
\end{align}
Note that $\tilde{x}^{j}_{i,n}\in I^{j}_{n-1}$. For notational convenience, we extend the definition of $\widetilde{X}^P_n$ and $\tilde{x}^{j}_{i,n}$ to all $n\geq0$ by defining
\begin{align}
	\widetilde{X}^{P}_{n}=\left\{ \begin{array}{ll} 
	X^{P}_{0}, & n=0,\\
	\widetilde{X}^{P}_{n}, & n\geq1,\\
	\end{array}\right. \quad 
	\tilde{x}^{j}_{i,n}=\left\{ \begin{array}{ll} 
	x^{j}_{i,0}, & n=0,\\
	\tilde{x}^{j}_{i,n}, & n\geq1,\\
	\end{array}\right. \, j=0,\ldots,\max(2^{n-1}-1,0).
\end{align}
In Appendix \ref{sec:append1}, we provide some examples of nested collocation point sets.

Finally, we would like to introduce the special level ``-1". This technique has been used in \cite{garcke2006dimension} to further reduce DoFs for high dimensional problems. We define $\widetilde{X}^{P}_{-1}=\{ \tilde{x}^{0}_{0,-1} = x^{0}_{i*,0} \}$ consists of a single point. Here, $x^{0}_{i*,0} $ is a point chosen from $X^{P}_0$ arbitrarily. It would be specified in this work later.

\subsection{MRA}
In this subsection, we will introduce MRA   of piecewise polynomial spaces associated with the nested collocation points. We define  the space of piecewise polynomial functions of degree at most $K$ on $\Omen$ by
\begin{align}
\label{eq:polynomial}
	V^{K}_{n} = \{v: v \in P^{K}( I^{j}_{n}),\, j=0,\ldots,2^n-1\}.
\end{align}

We consider using the Lagrange $(M=0)$ or Hermite $(M\geq1)$ interpolating polynomials on point set $X^{P}_{n}$  with respect to the first $M$ derivatives as basis functions, denoted as $\phi^{j}_{i,l,n}(x)$:
\begin{align}
	\partial_{x}^{l'}\phi^{j}_{i,l,n}(x^{j}_{i',n}) = \delta_{ll'}  \delta_{ii'}, \quad i,i'=0,\ldots,P, \, \text{and} \,\, l,l'=0,\ldots,M, 
\end{align}
where $\partial_{x}^{l'}$ denotes the $l'$-th derivative operator and the Kronecker delta is defined by $\delta_{ii'}=\left\{ \begin{array}{ll} 1, & i=i', \\ 0, & i\neq i'.\\ \end{array} \right.$ 
We can see that the degree of $\phi^{j}_{i,l,n}$ is  $K=(P+1)(M+1)-1$. 
Those bases can be obtained by a rescaling of  $\{\phi_{i,l}, i=0,\ldots,P, \, l=0,\ldots, M\},$ which are the Lagrange or Hermite interpolating polynomials defined on $I=[0,1]$, satisfying 
$$\partial_{x}^{l'}\phi_{i,l}(x_{i'})= \delta_{ll'} \,\delta_{ii'}.$$
Then we have the relation
\begin{align}
	\phi^{j}_{i,l,n}(x) = 2^{-nl} \phi_{i,l}(2^n x-j),
\end{align}
and
\begin{align}
\label{eq:polynomial2}
	V^{K}_{n} = span\{\phi^{j}_{i,l,n},\, i=0, \ldots P, \, l=0, \ldots M, \,j=0,\ldots2^n-1\}.
\end{align}


Moreover, we can now define the subspace $W^K_n$, $n\geq1$, as the complement of $V^K_{n-1}$ in $V^K_n$, in which the piecewise polynomials and their derivatives vanish at all points in $X^{P}_{n-1}$,
\begin{align}
\label{eq:w}
	V^{K}_{n} = V^{K}_{n-1}\oplus W^{K}_{n}, \quad W^{K}_{n}=span\{\phi^{j}_{i,l,n}: x^{j}_{i,n} \in \widetilde{X}^{P}_{n}, l=0,\ldots,M\}.
\end{align} 

\noindent
We   now provide the details of the multiwavelet bases of $W^K_n.$ This can be achieved by specifying the bases of $W^K_1,$ which are defined by $\{\varphi_{i,l},  \,i=0,\ldots,P, l=0,\ldots,M\}$ as
\begin{align*}
	\varphi_{i,l}=\left\{ \begin{array}{ll}
	\phi^{0}_{r,l,1}, & x\in(0,1/2),\\
	0, & x\in(1/2,1),\\
	\end{array}\right. \quad \text{if $\tilde{x}^{0}_{i,1}=x^{0}_{r,1}$},
\end{align*} 
or
\begin{align*}
	\varphi_{i,l}=\left\{ \begin{array}{ll}
	0, & x\in(0,1/2),\\
	\phi^{1}_{r,l,1}, & x\in(1/2,1),\\
	\end{array}\right.\quad \text{if $\tilde{x}^{0}_{i,1}=x^{1}_{r,1}$,}
\end{align*}
where $r$ is an integer from $\{0,\ldots,P\}$.
Clearly, $\varphi_{i,l}$ satisfies
\begin{align*}
		\partial_{x}^{l'}\varphi_{i,l}(x^0_{i',0})= 0, \qquad	
		\partial_{x}^{l'}\varphi_{i,l}(\tilde{x}^0_{i',1})= \delta_{ll'} \,\delta_{ii'}.
\end{align*} 

\noindent 
As a result, for $n\geq1$,  we have that
\begin{align}
\label{eq:basis_mapping}
 	W^{K}_{n} = span \{ \varphi^{j}_{i,l,n}: \varphi^{j}_{i,l,n}=2^{-l(n-1)}\varphi_{i,l}(2^{(n-1)} x-j), \, i=0,\ldots,P, \notag\\ \, l=0,\ldots,M, \, j=0,\ldots,2^{n-1}-1\}.
\end{align}



We can now extend the definition to the $0$-th level by defining
\begin{align}
W^{K}_{n}=\left\{ \begin{array}{ll} 
V^{K}_{0}, & n=0\\
W^{K}_{n}, & n\geq1,\\
\end{array}\right. \quad
\varphi^{j}_{i,l,n}=\left\{ \begin{array}{ll} 
\phi^{0}_{i,l,0}, & n=0,\\
\varphi^{j}_{i,l,n}, & n\geq1.\\
\end{array}\right. 
\end{align} 

\noindent
Thus, we have
\begin{align*}
V^{K}_{N} =  \bigoplus_{0\leq n\leq N} W^{K}_{n},
\end{align*} 
which shows a hierarchical representation of the standard piecewise polynomial space $V^K_N$.  

Finally, we incorporate the ``-1" level by further decomposing the space on the $0$-th level.  We set $W^K_{c,-1}$ as the space of constant function  on $[0,1]$, and the basis function is denoted by $\varphi^{0}_{c,0,0,-1}=1.$ Consequently, the ``corrected" space at level 0 is denoted by $W^{K}_{c,0}=W^{K}_{0}/W^K_{-1}$, with the basis functions $\{\varphi^{0}_{i,l,0}, i\neq i^* \ \text{or} \ l\neq0 \}.$ For convenience, we define
\begin{align}
	\varphi^{0}_{c,i,l,0} = \left\{ \begin{array}{ll}
	0, & \text{if} \, i=i^* \, \text{and} \, l=0, \\
	\varphi^{0}_{i,l,0}, & \text{otherwise}, \\
	\end{array} \right. \quad i=0, \cdots, P, \, l = 0, \cdots, M .
\end{align}
and for $n\geq1,$ we let $\varphi^{j}_{c,i,l,n} =\varphi^{j}_{i,l,n}, W^{K}_{c,n}=W^{K}_{n}.$ 
Therefore, now we have
$$V^{k}_{N} = W^{K}_{c,-1} \oplus W^{K}_{c,0} \oplus \cdots \oplus W^{K}_{c,N}. $$

\subsection{Multiresolution interpolation}
In this subsection, we define the multiresolution interpolation operator and the fast transform between function and derivative values with the hierarchical coefficients.
For a given function $f(x)\in C^{M}(\Omega_N)$, which is the piecewise $C^M$ function space on grid $\Omega_N,$ we can define $\mathcal{I}^{P,M}_{n}[f]$ as the standard Lagrange or Hermite interpolation on $V^{K}_{n}$:
\begin{align}
	\label{eq:1D_interpolation1}
	\mathcal{I}^{P,M}_n[f](x) 
	= \sum_{j=0}^{2^n-1} \sum_{l=0}^{M} \sum_{i=0}^{P} f^{(l)}(x^{j}_{i,n}) \phi^{j}_{i,l,n}(x),
\end{align}
such that 
\begin{align}
	\partial_{x}^l \left(\mathcal{I}^{P,M}_{n}[f] \right) (x_{i,n}^{j}) = f^{(l)}(x_{i,n}^{j}), \quad l=0,\ldots,M, \quad x_{i,n}^{j}\in X_{n}^P.
\end{align} 
Here, $f^{(l)}$ is the $l$-th derivative $\partial_{x}^{l} f$. Note that when the points $x_{i,n}^{j}$ contains the left or right limit, the function and derivative values should be read accordingly.

We can define  the increment interpolation operator  
\begin{align}
	\label{eq:increment_operator}
	\widetilde{\mathcal{I}}^{P,M}_{n}:=\left\{ \begin{array}{ll}
	\mathcal{I}^{P,M}_{0}, & n=0, \\
	\mathcal{I}^{P,M}_{n}-\mathcal{I}^{P,M}_{n-1}, & n\geq 1.\\
	\end{array}\right. 
\end{align}
Hence,
$
	\mathcal{I}^{P,M}_{N}[f](x) 
	= \sum_{n=0}^{N} \widetilde{\mathcal{I}}^{P,M}_{n}[f](x).
$
Furthermore, the interpolation operator $\mathcal{I}^{P,M}_{N}$ can be represented by multiwavelet bases
\begin{align}
	\label{eq:1D_interpolation2}
	\mathcal{I}^{P,M}_{N}[f](x) 
	= \sum_{n=0}^{N} \widetilde{\mathcal{I}}^{P,M}_{n}[f](x)
	= \sum_{n=0}^{N} \sum _{j=0}^{\max(2^{n-1}-1,0)} \sum_{l=0}^{M} \sum_{i=0}^{P} b^{j}_{i,l,n} \varphi^{j}_{i,l,n}(x).
\end{align}

%

The transform between the function and derivative values and the hierarchical coefficients can be computed fast using the pyramid scheme, which is illustrated in detail below.
First, we  compute the hierarchical coefficients given the function and derivative values. 
Note that an important property of the hierarchical bases is that  for any basis function $\varphi^{j}_{i,l,n}$,   its derivatives and itself will vanish on all point in $\widetilde{X}^P_{m}$, $m\leq n$ except for $\tilde{x}^{j}_{i,n}$. Hence, it is straightforward that
\begin{align*}
	f^{(l)}(\tilde{x}^{0}_{i,0}) = \partial_x^l \mathcal{I}^{P,M}_{N}[f] (\tilde{x}^{0}_{i,0}) 
	= b_{i,l,0}^{0} \,  \partial_{x}^{l}\varphi^{0}_{i,l,0} (\tilde{x}^{0}_{i,0})
	= b^{0}_{i,l,0}  .
\end{align*}
While for $n\geq1$ and $\tilde{x}^{j}_{i,n}\in\widetilde{X}^P_{n}$, we have
\begin{align*}
	f^{(l)}(\tilde{x}^{j}_{i,n}) 
	=& \partial_x^l \mathcal{I}^{P,M}_{N}[f] (\tilde{x}^{j}_{i,n}) \\
	=&  \sum_{n'=0}^{n-1} \sum_{j'=0}^{\max(2^{n'-1}-1,0)} \sum_{l'=0}^{M} \sum_{i'=0}^{P}  b_{i',l',n'}^{j'} \, \partial_{x}^{l}\varphi^{j'}_{i',l',n'}(\tilde{x}^{j}_{i,n})
   + b^{j}_{i,l,n} \, \partial_{x}^{l} \varphi^{j}_{i,l,n}(\tilde{x}^{j}_{i,n}) \notag \\
=& \partial_x^l \mathcal{I}^{P,M}_{n-1}[f](\tilde{x}^{j}_{i,n}) + b^{j}_{i,l,n} \\
=& \sum_{l'=0}^{M} \sum_{i'=0}^{P}  f^{(l')}(x^{j}_{i',n-1}) \, \partial_{x}^{l}\phi^{j}_{i',l',n-1} (\tilde{x}^{j}_{i,n}) + b^{j}_{i,l,n}  \\
=& \sum_{l'=0}^{M} \sum_{i'=0}^{P}  2^{n(l-l')} f^{(l')}(x^{j}_{i',n-1})  \partial_{x}^{l}\phi_{i',l'}(\tilde{x}^{0}_{i,1})
+ b^{j}_{i,l,n}.
\end{align*}
In summary, we can define an operator $\mathcal{F}^{-1}$ mapping from  $f^{(l)}(x^{j}_{i,n})$ to hierarchical coefficients $b^{j}_{i,l,n},$ for $\forall i, l, j,$
\begin{align}
\label{eq:1D_coeff}
	b^{j}_{i,l,n} 
	= \mathcal{F}^{-1}[f]  
	=
	\left\{ \begin{array}{ll}
	f^{(l)}(\tilde{x}^{j}_{i,0}) , & n=0,\\ 
	\displaystyle	f^{(l)}(\tilde{x}^{j}_{i,n}) -   \sum_{l'=0}^{M} \sum_{i'=0}^{P}  2^{(n-1)(l-l')} f^{(l')}(x^{j}_{i',n-1})  \partial_{x}^{l}\phi_{i',l'}(\tilde{x}^{0}_{i,1}), & n\geq1.\\
	\end{array} \right. 
\end{align}
For any given function $f(x)$, the interpolation coefficients $b^{j}_{i,l,n}$ in \eqref{eq:1D_interpolation2} can be obtained by \eqref{eq:1D_coeff}, 
with the pyramid scheme
$$\begin{tikzcd}
	 f^{(l)}(x^{j}_{i,0}) \arrow[rd] \arrow[d] & f^{(l)}(x^{j}_{i,1})  \arrow[rd] \arrow[d] & 
	 f^{(l)}(x^{j}_{i,2})  \arrow[rd] \arrow[d]  & \cdots  \\  
	b^{j}_{i,l,0}   & b^{j}_{i,l,1}  & b^{j}_{i,l,2}  & \cdots 
\end{tikzcd}$$ 
where we only need to pre-store the coefficients $ \partial_{x}^{l}\phi_{i',l'}(\tilde{x}^{0}_{i,1}),$ and the computational cost scales as $\mathcal{O}(2^N(K+1)^2)$.


On the other hand, suppose we have the piecewise polynomial 
\begin{align}
\label{eq:function}
	f(x)=\sum_{n=0}^{N} \sum _{j=0}^{\max(2^{n-1}-1,0)} \sum_{l=0}^{M} \sum_{i=0}^{P} b_{i,l,n}^{j} \varphi^{j}_{i,l,n}(x) \in V^{K}_{N},  
\end{align}
with given coefficients $b^{j}_{i,l,n},$ 
we can obtain the point values with incrementing level $n$ by using similar argument, so that for $\forall i, l, j,$
\begin{align}
\label{eq:1d_pointvalues}
	f^{(l)}(\tilde{x}^{j}_{i,n}) 
	=& \sum_{n'=0}^{n} \sum _{j'=0}^{\max(2^{n'-1}-1,0)} \sum_{l'=0}^{M} \sum_{i'=0}^{P} b_{i',l',n'}^{j'}  \,  \partial_{x}^{l}\varphi^{j'}_{i',l',n'}(\tilde{x}^{j}_{i,n})   \notag\\
	=& \mathcal{F}[b] = \left\{ \begin{array}{ll}
	b^{0}_{i,l,0} , & n=0, \\ 
	\displaystyle	b^{j}_{i,l,n} + \sum_{l'=0}^{M} \sum_{i'=0}^{P}  2^{(n-1)(l-l')} f^{(l')}(x^{j}_{i',n-1}) \partial_{x}^{l}\phi_{i',l'}(\tilde{x}^{0}_{i,1}) , & n\geq1.\\ 
	\end{array}\right.
\end{align}
This procedure can be carried out by the following pyramid scheme with total computational cost  $\mathcal{O}(2^N(K+1)^2)$.
$$\begin{tikzcd}
	b^{j}_{i,l,0} \arrow[d]  & b^{j}_{i,l,1} \arrow[d]  & b^{j}_{i,l,2} \arrow[d]  & \cdots \\
 	f^{(l)}(\tilde{x}^{j}_{i,0}) \arrow[r] & f^{(l)}(\tilde{x}^{j}_{i,1})  \arrow[r] & 
	f^{(l)}(\tilde{x}^{j}_{i,2})  \arrow[r]   & \cdots  \\  
\end{tikzcd}$$


In addition, if the level ``-1" is taken into account, we will split $\mathcal{I}^{P,M}_{0}$ into $\widehat{\mathcal{I}}^{P,M}_{c,-1}$ and $\widehat{\mathcal{I}}^{P,M}_{c,0}$, with
\begin{align}
\widehat{\mathcal{I}}^{P,M}_{c,-1}[f](x)=f(\tilde{x}^{0}_{0,-1}),  \quad \text{and} \quad
\widehat{\mathcal{I}}^{P,M}_{c,0}=\mathcal{I}^{P,M}_{0} - \widehat{\mathcal{I}}^{P,M}_{c,-1}.
\end{align}
And we define $\widehat{\mathcal{I}}^{P,M}_{c,n}=\widehat{\mathcal{I}}^{P,M}_{n}$, $n\geq1$, for notation convenience.
Then the interpolation $\mathcal{I}^{P,M}_{N}$ can be rewritten as 
\begin{align}
\label{eq:1D_interpolation3}
		\mathcal{I}^{P,M}_{N}[f](x) 
	=& \widehat{\mathcal{I}}^{P,M}_{c,-1}[f](x) + \widehat{\mathcal{I}}^{P,M}_{c,0}[f](x) + \sum_{n=1}^{N} \widehat{\mathcal{I}}^{P,M}_{n}[f](x) \notag \\
	=& b^{0}_{c,0,0,-1} \varphi^{0}_{c,0,0,-1}(x) + \sum_{l=0}^{M} \sum_{i=0}^{P} b^{0}_{c,i,l,0} \varphi^{0}_{c,i,l,0}(x) 
	+ \sum_{n=1}^{N} \sum _{j=0}^{2^{n-1}} \sum_{l=0}^{M} \sum_{i=0}^{P} b^{j}_{c,i,l,n} \varphi^{j}_{c,i,l,n}(x).
\end{align}

\noindent
Using the definition, we know that $b^{0}_{c,i*,0,0}=0$ and $b^{j}_{c,i,l,n}=b^{j}_{i,l,n}$ for $n\geq1$ always holds. We denote the corrected mapping between the hierarchical coefficients and point values as $\mathcal{F}_c / \mathcal{F}_c^{-1}$. Moreover, they are the same as $\mathcal{F}$ \eqref{eq:1d_pointvalues} or $\mathcal{F}^{-1}$ \eqref{eq:1D_coeff} for $n\geq1$, and the corresponding parts for $n=-1, 0$ are replaced by
\begin{align}
\label{eq:forward_1D2}
	f^{(l)}(\tilde{x}^{0}_{i,n}) =& \mathcal{F}_{c}[b_c] 
	=\left\{ \begin{array}{ll}
	b^{0}_{c,0,0,-1}, & n=-1, \, l=0, \, \text{and} \, i=0, \\
	b^{0}_{c,0,0,-1}+b^{0}_{c,i,0,0} , & n=0, \, l=0, \, \text{and} \, i\neq i^*,\\
	b^{0}_{c,i,l,0}, & n=0, l\neq 0, \\
	\end{array}\right.
\end{align}
and
\begin{align}
\label{eq:backward_1D2}
	b^{0}_{c,i,l,n} =& \mathcal{F}_{c}^{-1}[f] 
	= \left\{ \begin{array}{ll}
	f(\tilde{x}^{0}_{0,-1}), & n=-1, \, l=0, \, \text{and} \, i=0, \\
	f(\tilde{x}^{0}_{i,0}) - f(\tilde{x}^{0}_{0,-1}), & n=0, \, l=0, \, \text{and} \, i\neq i^*,\\
	f^{(l)}(\tilde{x}^{0}_{i,0}), & n=0,\, l\neq0. \\
	\end{array} \right. 
\end{align}

\begin{rmk}
	We want to remark that for simplicity, we have defined the Hermite interpolation \eqref{eq:1D_interpolation1} based on a fixed value of $M$ for all points.  In general, some points may have more known derivatives than others. In this case, in order for the nested structure to hold, we require 
	\begin{align*}
		M_{i} \leq M_{r}, \quad \text{if} \quad x^{j}_{i,n-1} = x^{2j}_{r,n} \, \,\text{or} \,\, x^{j}_{i,n-1} = x^{2j+1}_{r,n}.
	\end{align*}
	Then, the hierarchical basis and fast transforms    can be obtained similarly. 
	\end{rmk}

Finally, numerical quadrature of the function $f(x)$ on $[0,1]$ can be obtained by computing the integral of the interpolation function \eqref{eq:1D_interpolation2} or \eqref{eq:1D_interpolation3}  
\begin{align}
\label{eq:integral}
	\int_{0}^{1} \mathcal{I}^{P,M}_N [f](x) dx 
	=&  \sum_{n=0}^{N} \sum _{j=0}^{\max(2^{n-1}-1,0)} \sum_{l=0}^{M} \sum_{i=0}^{P} \omega^{j}_{i,l,n} b^{j}_{i,l,n}  \notag \\
	=&   \omega^{0}_{c,0,0,-1} b^{0}_{c,0,0,-1} 
	+ \sum_{n=0}^{N} \sum _{j=0}^{\max(2^{n-1},0)} \sum_{l=0}^{M} \sum_{i=0}^{P} \omega^{j}_{c,i,l,n} b^{j}_{c,i,l,n}.
\end{align} 
Here, the quadrature weights $\omega^{j}_{i,l,n} = \int_{0}^{1}\varphi^{j}_{i,l,n}(x)dx$ and $\omega^{j}_{c,i,l,n} = \int_{0}^{1}\varphi^{j}_{c,i,l,n}(x)dx$ can be pre-computed and stored.

\subsection{Some special cases for interpolation of continuous functions}\label{sec:cont}

We have discussed interpolation for functions in piecewise polynomial space $C^M(\Omega_N).$ For $f(x)$ with continuity on the whole domain, some simplification can be made, which results in fewer DoFs of $W_n^K.$ We will give a few such examples.

If $f(x) \in C^M[0,1]$ and $X^{1}_{0}=\{ 0^+, 1^-\},$ then we can merge the basis and reduce the DoFs from $2(M+1)$ to $M+1.$ For example, this is illustrated in Figure \ref{fig:cbasis} for the case when $P=1, M=0, 1.$ The case of $P=1, M=0$ actually corresponds to the most widely used hierarchical bases for piecewise linear continuous finite element.

Another example is when $f(x) \in C^0[0,1],$ $P=2, M=0$ as shown in Appendix \ref{sec:append2.2}. It is easy to check that the hierarchical coefficient for the basis $\varphi_{1,0}$ will always be zero because of continuity, so we can simply remove this basis from $W_1^K$ and hence reduce the DoFs from $3$ to $2.$

\begin{figure}[htp]
	\begin{center}
		\subfigure[$P=1, M=0$] {\includegraphics[width=.8\textwidth]{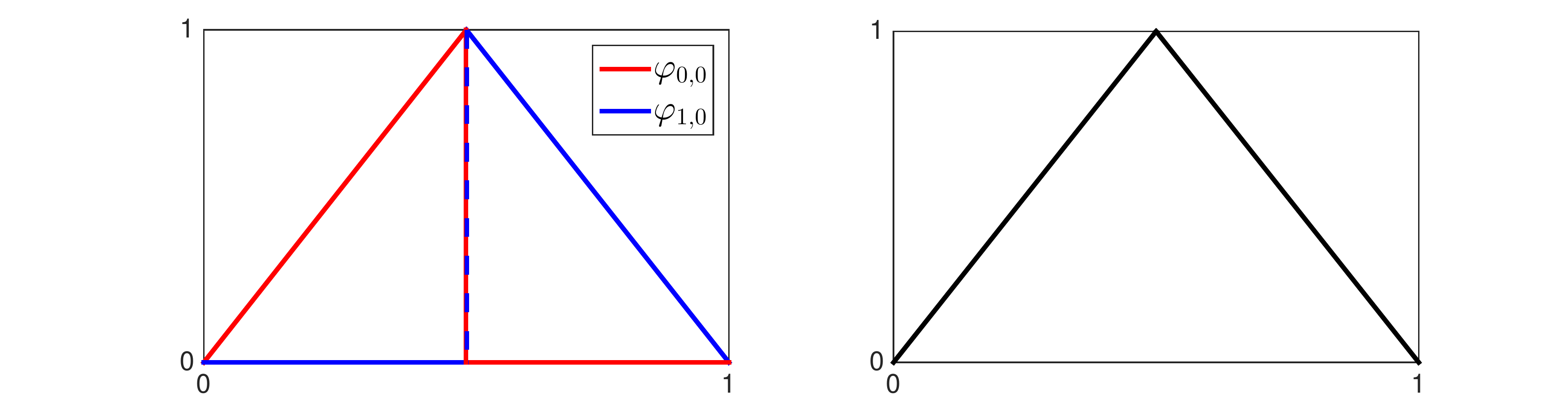}}
		\subfigure[$P=1, M=1$] {\includegraphics[width=.8\textwidth]{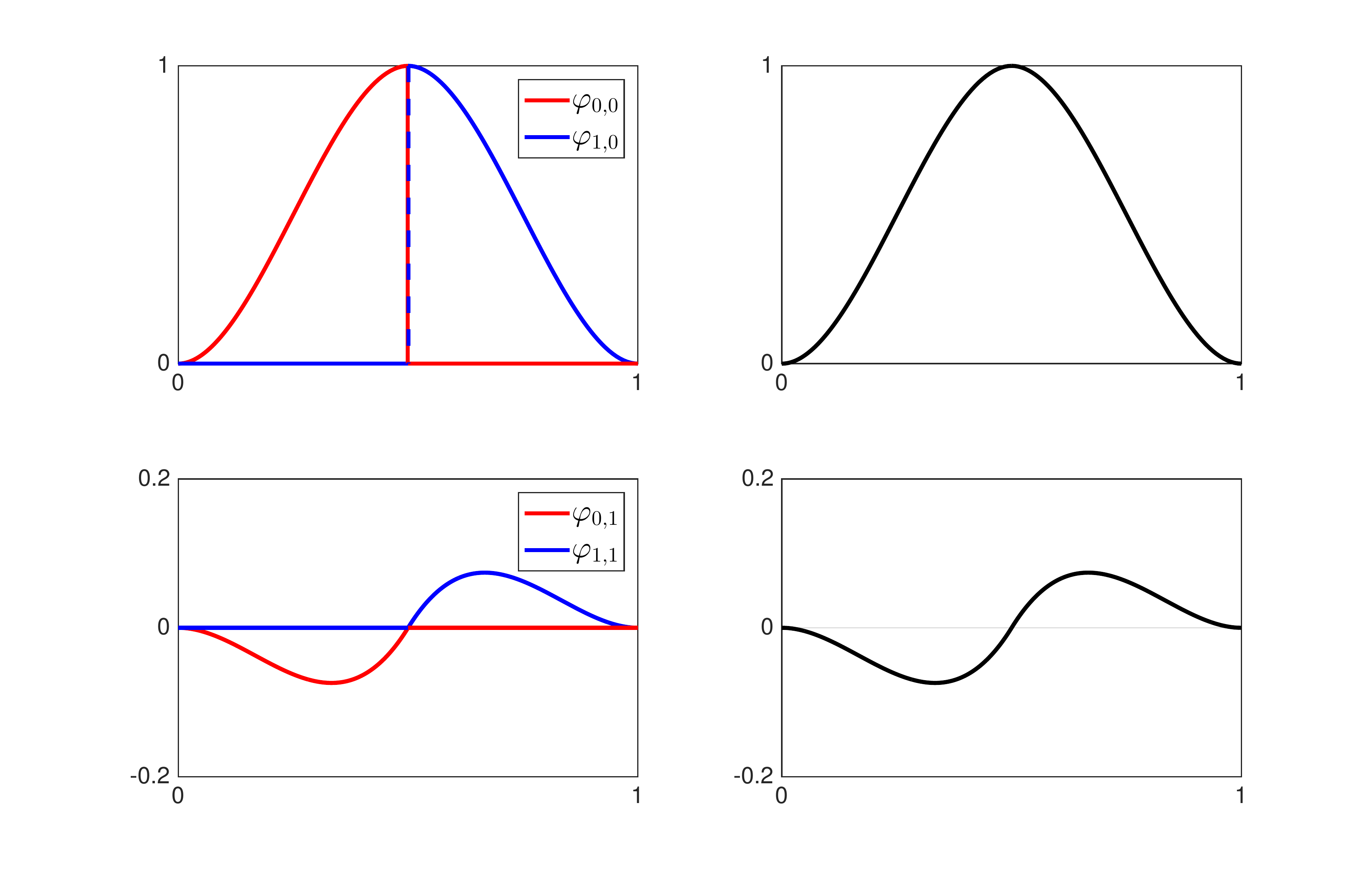}}
	\end{center}
	\caption{Illustration of merging of hierarchical basis functions. Left: discontinuous hierarchical bases in $W_1^K,$ Right: merged bases in $W_1^K$ when the function $f$ is $C^M$ continuous. The expressions of bases can be found in Appendix \ref{sec:append1.2} (Type 1) and
\ref{sec:append2.3}, resp.
	}
	\label{fig:cbasis}
\end{figure}

\section{Multidimensional case}
\label{sec:multid}


\subsection{Multiresolution interpolation method}

Now, we consider   multi-dimensional case based on the tensor product of the one-dimensional bases introduced previously.
For a $d$-dimensional problem, we consider the domain $I^d=[0,1]^d$. 

First, we recall some   notations in $\mathbb{N}^{d}$ and $\mathbb{N}^{d}_0$, where $\mathbb{N}$ ($\mathbb{N}_0$) denotes the set of (nonnegative) integers. 
For a multi-index $\mathbf{\alpha}=(\alpha_{1},\ldots, \alpha_{d})\in \mathbb{N}^{d}$,  the $l^{1}$ and $l^{\infty}$ indices are defined as
\begin{align*}
| \mathbf{\alpha} |_{1} = \sum_{m=1}^{d}\alpha_{m}, \quad
| \mathbf{\alpha} |_{\infty} = \max_{1\leq m\leq d}\alpha_{m}.
\end{align*}
The component-wise arithmetic operators and the relational operators are defined as
\begin{align*}
&\mathbf{\alpha}\pm \mathbf{\beta}=( \alpha_{1}\pm\beta_{1}, \ldots, \alpha_{d}\pm\beta_{d} ), \quad
2^{\mathbf{\alpha}}=(2^{\alpha_{1}},\ldots, 2^{\alpha_{d}}),\\
	&\mathbb{\alpha}\leq \mathbb{\beta} \Leftrightarrow \alpha_{m}\leq \beta_{m}, \, \forall m, \qquad
	\mathbb{\alpha}< \mathbb{\beta} \Leftrightarrow \mathbf{\alpha}\leq \mathbf{\beta} \, \text{and} \, \mathbf{\alpha}\neq\mathbf{\beta},\\
	&\mathbf{0}=(0,\ldots, 0), \quad \mathbf{-1}=(-1,\ldots, -1), \quad \mathbf{P}=(P,\ldots, P), \quad \mathbf{M}=(M,\ldots, M).
\end{align*}

In multi-dimensional sense, we define the tensor-product mesh grid $\Omega_{\mathbf{n}}=\Omega_{n_{1}} \otimes\cdots \otimes \Omega_{n_{d}}$ and mesh size $h_{\mathbf{n}}=(h_{n_1}, \ldots, h_{n_d})$ with mesh level $\mathbf{n}=(n_{1}, \ldots, n_{d})\in \mathbb{N}^{d}_{0}$. 
Then the tensor product piecewise polynomial space can be obtained as
\begin{align}
	\mathbf{V}^{K}_{\mathbf{n}} =\{ v: v\in P^{K}(I^{\mathbf{j}}_{\mathbf{n}}), \mathbf{0}\leq \mathbf{j} \leq 2^{\mathbf{n}} - \mathbf{1} \} = V^{K}_{n_1,x_1} \times \cdots \times V^{K}_{n_{d},x_d},
\end{align}
where, $I^{\mathbf{j}}_{\mathbf{n}}=\{\mathbf{x}: \mathbf{x}=(x_1,\ldots,x_d)\in\mathbb{R}^d, x_{m}\in I^{j_m}_{n_m}, m=1,\ldots,d\}$ denotes the elementary cell on $\Omega_{\mathbf{n}}$, $P^{K}(I^{\mathbf{j}}_{\mathbf{n}})$ denotes the collection of the polynomial of degree up to $K$ in each dimension on cell $I^{\mathbf{j}}_{\mathbf{n}}$ and $V^{K}_{n_m,x_m}$ corresponds to the space $V^{K}_{n_m}$ defined in the $m$-th dimension. If we use equal mesh refinement of size $h_N=2^{-N}$ in each coordinate direction, the grid and space will be denoted by $\Omega_{N}$ and $\mathbf{V}^{K}_{N}$, respectively. 

Basis functions of $\mathbf{V}^{K}_{\mathbf{n}}$ can be obtained by tensor product,
\begin{align*}
	 \phi^{\mathbf{j}}_{\mathbf{i},\mathbf{l}, \mathbf{n}}(\mathbf{x}) = \prod_{m=1}^{d}\phi^{j_m}_{i_m, l_m,n_m}(x_m), \quad  j_m=0, \ldots, 2^{n_m}-1, \, i_m=0,\ldots,P, \, l_m=0,\ldots, M, 
\end{align*} 
which are the Lagrange interpolation ($M=0$) or Hermite interpolation ($M\geq1$) polynomials corresponding to the point
	$$\mathbf{x}^{\mathbf{j}}_{\mathbf{i},\mathbf{n}} = \left( x^{j_1}_{i_1,n_1}, \ldots, x^{j_d}_{i_d,n_d} \right) \in \mathbf{X}^{P}_{\mathbf{n}} = X^{P}_{n_1} \times \cdots  \times X^{P}_{n_d},$$
and $K=(P+1)(M+1)-1$. We want to remark that the interpolation can use different $P$ and $M$ in each direction. However, this aspect is not explored in this paper.

Similar to the 1D case, the space $\mathbf{V}^{K}_{\mathbf{n}}$ are be represented hierarchically. 
\begin{align}
	\mathbf{V}^{K}_{\mathbf{n}} = \bigoplus_{ 0 \leq n'_1 \leq n_1,\ldots, 0 \leq n'_d \leq  n_d} \mathbf{W}^{K}_{\mathbf{n'}},  \quad \text{and} \quad 
	\mathbf{V}^{K}_{N} = \bigoplus_{|\mathbf{n'}|_{\infty}\leq N,  \, \mathbf{n'}\in \mathbb{N}_{0}^{d}} \mathbf{W}^{K}_{\mathbf{n'}}, 
\end{align}
with 
\begin{align}
	\mathbf{W}^{K}_{\mathbf{n}} = W^{K}_{n_1,x_1}\times \cdots \times W^{K}_{n_d,x_d}.
\end{align} 
For the multi-dimensional increment space $\mathbf{W}^{K}_{\mathbf{n}}$, basis functions are defined by tensor products as well,
\begin{align}
	\varphi^{\mathbf{j}}_{\mathbf{i},\mathbf{l},\mathbf{n}} (\mathbf{x}) = \prod_{m=1}^{d} \varphi^{j_m}_{i_m,l_m,n_m}(x_m).
\end{align} 

Then, we introduce the interpolation operator in multi-dimension $\mathcal{I}^{P,M}_{\mathbf{n}}: C^{M}(\Omega_{N}) \rightarrow \mathbf{V}^{K}_{\mathbf{n}}, \bn\in\mathbb{N}_{0}^d,$ 
\begin{align}
\label{eq:multiD_interpolation}
	\mathcal{I}^{P,M}_{\mathbf{n}}[f](\mathbf{x}) 
	=& \mathcal{I}^{P,M}_{n_1} \circ \cdots \circ \mathcal{I}^{P,M}_{n_d} [f](\mathbf{x}) 
	= \sum_{ \substack{ \mathbf{0}\leq \mathbf{j} \leq 2^{\mathbf{n}}-1 \\ \mathbf{0}\leq \mathbf{i}\leq \mathbf{P} \\ \mathbf{0}\leq\mathbf{l}\leq\mathbf{M} }}  f^{(\mathbf{l})}(\mathbf{x}^{\mathbf{j}}_{\mathbf{i},\mathbf{n}} ) \phi^{\mathbf{j}}_{\mathbf{i}, \mathbf{l}, \mathbf{n}} (\mathbf{x}) \notag\\
	=& \sum_{\mathbf{0} \leq \mathbf{n'}\leq \mathbf{n}} \widetilde{\mathcal{I}}^{P,M}_{n'_1} \circ \cdots \circ \widetilde{\mathcal{I}}^{P,M}_{n'_d} [f](\mathbf{x}) 
	= \sum_{ \substack{ \mathbf{0} \leq \mathbf{n'}\leq \mathbf{n}, \\ \mathbf{0}\leq \mathbf{j} \leq \max(2^{\mathbf{n'}-\mathbf{1}}-\mathbf{1},\mathbf{0}) \\ \mathbf{0}\leq\mathbf{i}\leq\mathbf{P}, \, \mathbf{0}\leq\mathbf{l}\leq\mathbf{M} }} b^{\mathbf{j}}_{\mathbf{i}, \mathbf{l}, \mathbf{n'}} \varphi^{\mathbf{j}}_{\mathbf{i}, \mathbf{l}, \mathbf{n'}} (\mathbf{x}),
\end{align}
with the set of interpolation points as $\mathbf{X}^{P}_{\mathbf{n}}$. Here, $f^{(\mathbf{l})}$ denotes the mixed derivative $\partial_{x_1}^{l_1}\cdots\partial_{x_d}^{l_d}f$. 
Consequently, for each point $\widetilde{\mathbf{x}}^{\mathbf{j}}_{\mathbf{i}, \mathbf{n}} \in \widetilde{\mathbf{X}}^{P}_{\mathbf{n}} = \widetilde{X}^{P}_{n_1} \times \cdots \times \widetilde{X}^{P}_{n_d} \subset \mathbf{X}^{P}_\mathbf{n}$, we have
\begin{align}
\label{eq:interpolate_multiD}
	f^{(\mathbf{l})}(\widetilde{\mathbf{x}}^{\mathbf{j}}_{\mathbf{i}, \mathbf{n}} )  
	=&  \partial_{\mathbf{x}}^{\mathbf{l}} \, \mathcal{I}^{P,M}_{\mathbf{n}} [f] (\widetilde{\mathbf{x}}^{\mathbf{j}}_{\mathbf{i}, \mathbf{n}}) \notag \\
	=& \sum_{ \substack{ 0\leq n'_d \leq n_d , \\ 0\leq j'_d \leq \max(2^{n'_d-1}-1,0) \\ 0\leq i'_d\leq P \\ 0\leq l'_d\leq M}} \partial_{x_d}^{l_d}\varphi^{j'_d}_{i'_d, l'_d, n'_d} (\tilde{x}^{j_d}_{i_d,n_d})
	\left( \cdots \left( \sum_{ \substack{ 0 \leq n'_1 \leq n_1 , \\ 0\leq j'_1 \leq \max(2^{n'_1-1}-1,0) \\ 0\leq i'_1\leq P \\ 0\leq l'_1\leq M}} \partial_{x_1}^{l_1}\varphi^{j'_1}_{i'_1, l'_1, n'_1}(\tilde{x}^{j_1}_{i_1,n_1}) \cdot b^{\mathbf{j'}}_{\mathbf{i'}, \mathbf{l'}, \mathbf{n'}} \right) 
	\right).
\end{align}

\noindent
Note that in each direction, the formula is the same as one-dimensional case. Thus, the multi-dimension problem can be splitted into $d$ one-dimensional problems. For example, when $d=2$, \eqref{eq:interpolate_multiD} can be computed by
\begin{align*}
	& f^{(\mathbf{l})} (\widetilde{\mathbf{x}}^{\mathbf{j}}_{\mathbf{i}, \mathbf{n}} )  
	=  \sum_{ \substack{ 0 \leq n'_2 \leq n_2 , \\ 0\leq j'_2 \leq \max(2^{n'_2-1}-1,0) \\ 0\leq i'_2\leq P, \,\, 0\leq l'_2\leq M}} \partial_{x_2}^{l_2}\varphi^{j'_2}_{i'_2, l'_2, n'_2} (\tilde{x}^{j_2}_{i_2,n_2}) \cdot \tilde{b}^{(j_1,j'_2)}_{(i_1,i'_2), (l_1,l'_2), (n_1, n'_2)}, \\ 
	& \tilde{b}^{(j_1,j'_2)}_{(i_1,i'_2),(l_1,l'_2),(n_1,n'_2)} = \sum_{ \substack{ 0\leq n'_1\leq n_1 \\ 0 \leq j'_1 \leq \max(2^{n'_1-1}-1,0) \\ 0 \leq i'_1\leq P, \,\,  0\leq l'_1 \leq M }} \partial_{x_1}^{l_1}\varphi^{j'_1}_{i'_1, l'_1, n'_1} (\tilde{x}^{j_1}_{i_1,n_1}) \cdot b^{(j'_1,j'_2)}_{ (i'_1,i'_2), (l'_1, l'_2), (n'_1, n'_2) }. 
\end{align*}
In a compact notation in arbitrary $d$ dimensions, this means
\begin{align}
\label{eq:forward_2D}
	f^{(\mathbf{l})} (\widetilde{\mathbf{x}}^{\mathbf{j}}_{\mathbf{i}, \mathbf{n}}) =  \mathcal{F}_{x_d} \circ \cdots \circ \mathcal{F}_{x_1} [b], 
\end{align}
where $\mathcal{F}_{x_m}$ denotes  the operator \eqref{eq:1d_pointvalues} working on $x_m$-direction.   Conversely,
\begin{align}
\label{eq:backward_2D}
	b^{\mathbf{j}}_{\mathbf{i}, \mathbf{l}, \mathbf{n}}  = \mathcal{F}^{-1}_{x_1} \circ \cdots \circ \mathcal{F}^{-1}_{x_d}[f].
\end{align}
 Here, the directions $x_1, \cdots, x_d$ in \eqref{eq:forward_2D} and \eqref{eq:backward_2D} can be reordered arbitrarily.
The computational complexity of \eqref{eq:forward_2D} and \eqref{eq:backward_2D} 
scales as $\mathcal{O}(d2^{Nd}(K+1)^{d+1}).$

Now we consider the multiresolution interpolation starting from level ``-1",  that is 
\begin{align}
\label{eq:multiD_interpolation2}
	\mathcal{I}^{P,M}_{\mathbf{n}}[f](\mathbf{x}) 
	= \sum_{ \substack{ \mathbf{-1} \leq \mathbf{n'}\leq \mathbf{n}, \, \mathbf{0}\leq \mathbf{j} \leq \mathbf{J_{n'}} \\ \mathbf{0}\leq\mathbf{i}\leq\mathbf{P_{n'}}, \, \mathbf{0}\leq\mathbf{l}\leq\mathbf{M_{n'}} }} b^{\mathbf{j}}_{c,\mathbf{i}, \mathbf{l}, \mathbf{n'}} \, \varphi^{\mathbf{j}}_{c,\mathbf{i}, \mathbf{l}, \mathbf{n'}} (\mathbf{x}),
\end{align}
with 
$$\mathbf{J_n}=(J_{n_1},\cdots,J_{n_d}), \quad \mathbf{M_n}=(M_{n_1},\cdots,M_{n_d}),\quad \mathbf{P_n}=(P_{n_1},\cdots,P_{n_d}),$$ 
and 
\begin{align*}
	J_n =  \left\{\begin{array}{ll} 0, & n=-1, 0, \\ 2^{n-1}-1, & n\geq1, \\ \end{array}\right.  \quad 
	M_n = \left\{\begin{array}{ll} 0, & n=-1, \\ M, & n\geq0, \\ \end{array}\right.  \quad 
	P_n =  \left\{\begin{array}{ll} 0, & n=-1, \\ P, & n\geq0. \\ \end{array}\right. 
\end{align*}
Then, similarly
\begin{align}
\label{eq:forward_2D2}
	f^{(\mathbf{l})} (\widetilde{\mathbf{x}}^{\mathbf{j}}_{\mathbf{i}, \mathbf{n}}) =  \mathcal{F}_{c,x_d} \circ \cdots \circ \mathcal{F}_{c,x_1} [b_{c}] , \quad 
	b^{\mathbf{j}}_{c,\mathbf{i}, \mathbf{l}, \mathbf{n}} = \mathcal{F}^{-1}_{c,x_1}  \circ \cdots \circ \mathcal{F}^{-1}_{c,x_d} [f],
\end{align}
where $\mathcal{F}_c$ and $\mathcal{F}_c^{-1}$ are given in \eqref{eq:forward_1D2} and \eqref{eq:backward_1D2}. The computational complexity of $\mathcal{I}^{P,M}_{N}$ still scales as $\mathcal{O}(d2^{Nd}(K+1)^{d+1}).$

\subsection{Multiresolution interpolation method on sparse grid}

In this subsection, we introduce  sparse grid interpolation method. Again, we begin the discussion with all levels starting from 0. 
We define the sparse grid space  as 
\begin{align}
	\widehat{\mathbf{V}}^{K}_{N} = \bigoplus_{|\mathbf{n}|_{1}\leq N,\, \mathbf{n}\in \mathbb{N}^{d}_{0}} \mathbf{W}^{K}_{\mathbf{n}}.
\end{align}
This is a subspace of $\mathbf{V}^{K}_{N} = \bigoplus_{|\mathbf{n}|_{\infty}\leq N, \,\mathbf{n}\in \mathbb{N}^{d}_{0}} \mathbf{W}^{K}_{\mathbf{n}}$. Following Lemma 2.3 in \cite{sparsedgelliptic}, we can prove that the dimension of $\widehat{\mathbf{V}}^{K}_{N}$ is given by
\begin{align*}
dim \left( \widehat{\mathbf{V}}^{K}_{N} \right) =& (K+1)^d \left\{ \sum_{m=0}^{d-1} \begin{pmatrix} d \\ m\\ \end{pmatrix}  \left( (-1)^{d+m} + 2^{N+m-d+1} \sum_{n=0}^{d-m-1} \begin{pmatrix} N \\ n\\	\end{pmatrix} \cdot (-2)^{d-m-1-n} \right) +1 \right\} \\
=&  \mathcal{O}\left( (K+1)^d 2^N N^{d-1} \right),
\end{align*} 
which is significantly less than that of $\mathbf{V}^{K}_{N}$ with $\mathcal{O}((K+1)^d 2^{Nd})$ when $d$ is large. 


We can now introduce the interpolation operator $\widehat{\mathcal{I}}^{P,M}_{N}: C^{M}(\Omega_{N})\rightarrow \widehat{\mathbf{V}}^{K}_{N}$
\begin{align}
\label{eq:multiD_sparse}
\widehat{\mathcal{I}}^{P,M}_{N}[f](\mathbf{x}) 
= \sum_{|\mathbf{n}|_1\leq N} \widetilde{\mathcal{I}}^{P,M}_{n_1} \circ \cdots \circ \widetilde{\mathcal{I}}^{P,M}_{n_d} [f](\mathbf{x}) 
= \sum_{ \substack{ |\mathbf{n}|_{1}\leq N, \\ \mathbf{0}\leq \mathbf{j} \leq \max(2^{\mathbf{n}-\mathbf{1}}-\mathbf{1},\mathbf{0}) \\ \mathbf{0}\leq \mathbf{i}\leq \mathbf{P} \\ \mathbf{0}\leq \mathbf{l}\leq \mathbf{M} }} b^{\mathbf{j}}_{\mathbf{i}, \mathbf{l}, \mathbf{n}} \varphi^{\mathbf{j}}_{\mathbf{i}, \mathbf{l}, \mathbf{n}} (\mathbf{x}) 
\end{align}
with the set of interpolation points as
\begin{align}
\widehat{\mathbf{X}}^{P}_{N} =& \{ \widetilde{\mathbf{x}}^{\mathbf{j}}_{\mathbf{i},\mathbf{n}} : |\mathbf{n}|_{1}\leq N, \mathbf{0}\leq \mathbf{j} \leq \max(2^{\mathbf{n}-\mathbf{1}}-\mathbf{1},\mathbf{0}), \mathbf{0}\leq \mathbf{i}\leq \mathbf{P} \} \notag \\
=& \bigcup_{|\mathbf{n}|_{1}\leq N} \widetilde{X}^{P}_{n_1}\times \cdots \times \widetilde{X}^{P}_{n_d},
\end{align}
and $b^{\mathbf{j}}_{\mathbf{i}, \mathbf{l}, \mathbf{n}}$ are the corresponding hierarchical coefficients.

To transform between the function (and derivative) values at collocation points and the hierarchical coefficients, we will perform the fast methods similar to \cite{shen2010efficient}.
In particular, for 2D cases, \eqref{eq:multiD_sparse} implies that for any $\widetilde{\mathbf{x}}^{\mathbf{j}}_{\mathbf{i}, \mathbf{n}} \in \widehat{\mathbf{X}}^{P}_{N}$ with $n_1+n_2\leq N$,
\begin{align*}
\label{eq:coeff_multiD}
f^{(\mathbf{l})}(\widetilde{\mathbf{x}}^{\mathbf{j}}_{\mathbf{i}, \mathbf{n}} )  
=& \partial_{\mathbf{x}}^{\mathbf{l}} \widehat{\mathcal{I}}^{P,M}_{N}[f] (\widetilde{\mathbf{x}}^{\mathbf{j}}_{\mathbf{i}, \mathbf{n}}) \notag\\
=& \sum_{ \substack{ 0\leq n'_2 \leq N , \\ 0\leq j'_2 \leq \max(2^{n'_2-1}-1,0) \\ 0\leq i'_2\leq P \\ 0\leq l'_2\leq M}} \partial_{x_2}^{l_2}\varphi^{j'_2}_{i'_2, l'_2, n'_2} (\tilde{x}^{j_2}_{i_2,n_2})
\left( \sum_{ \substack{ 0 \leq n'_1 \leq N-n'_2 , \\ 0\leq j'_1 \leq \max(2^{n'_1-1}-1,0) \\ 0\leq i'_1\leq P \\ 0\leq l'_1\leq M}} \partial_{x_1}^{l_1}\varphi^{j'_1}_{i'_1, l'_1, n'_1}(\tilde{x}^{j_1}_{i_1,n_1}) \cdot b^{\mathbf{j'}}_{\mathbf{i'}, \mathbf{l'}, \mathbf{n'}} \right) \\
=& \sum_{ \substack{ 0\leq n'_2 \leq n_2 , \\ 0\leq j'_2 \leq \max(2^{n'_2-1}-1,0) \\ 0\leq i'_2\leq P \\ 0\leq l'_2\leq M}} \partial_{x_2}^{l_2}\varphi^{j'_2}_{i'_2, l'_2, n'_2} (\tilde{x}^{j_2}_{i_2,n_2})
\left( \sum_{ \substack{ 0 \leq n'_1 \leq n_1, \\ 0\leq j'_1 \leq \max(2^{n'_1-1}-1,0) \\ 0\leq i'_1\leq P \\ 0\leq l'_1\leq M}} \partial_{x_1}^{l_1}\varphi^{j'_1}_{i'_1, l'_1, n'_1}(\tilde{x}^{j_1}_{i_1,n_1}) \cdot b^{\mathbf{j'}}_{\mathbf{i'}, \mathbf{l'}, \mathbf{n'}} \right),
\end{align*}
where we have used the property of the hierarchical bases in the second equality. 
This formulation is now the  same as \eqref{eq:interpolate_multiD}.
Hence, the algorithm \eqref{eq:forward_2D} can be applied for sparse approximation space as well, and we can split the problem to two 1D problems.
\begin{align*}
	\widehat{X}^{P}_{N} 
	= \bigcup_{0\leq n_1\leq  N} \widetilde{X}^{P}_{n_1}\times X^{P}_{N-n_1} 
	= \bigcup_{0\leq n_2\leq  N} X^{P}_{N-n_2}\times \widetilde{X}^{P}_{n_2},
\end{align*}
Then, the total computational complexity of the transforms for a $d$-dimension problem would be $\mathcal{O}(dN_d(K+1))$ with $N_d$ is the DoF in $\widehat{\mathbf{V}}^K_N$. 
The exact procedures are described in Algorithms \ref{algorithm1} and \ref{algorithm2} below.


\begin{algorithm}
	\caption{Fast transform from hierarchical coefficients $ b^{\bj}_{\bi,\bl,\bn}$ to point (and derivative) values $ f^{(\bl)}(\widetilde{\bx}^{\bj}_{\bi,\bn})$ on sparse approximation space in d dimensions} 
	\label{algorithm1}
	\begin{algorithmic}
		\STATE \textbf{Input:} $N$, $d$, $P, M, \widehat{\mathbf{X}}^P_N$ and coefficients $\{ b^{\bj}_{\bi,\bl,\bn}\}$ 
		
		\STATE \textbf{Output:}	the point values $\{ f^{(\bl)}(\widetilde{\mathbf{x}}^{\bj}_{\bi, \bl}) \}$
		
		\STATE
		
		\STATE set $\tilde{b}^{\bj}_{\bi, \bl, \bn}=b^{\bj}_{\bi, \bl, \bn}$
		
		\FOR  {$d'=1$ to $d$} 
		
		 \FOR {all $\bn'\in \{ \bn'=(n_1, \ldots, n_{d'-1}, n_{d'+1},\ldots, n_{d}), |\mathbf{n}'|_1\leq N \}$}
		
		\STATE one-dimensional transforms $\tilde{f}^{(\bl)}(\widetilde{\mathbf{x}}^{\bj}_{\bi, \bn}) =\mathcal{F}_{x_d}[\tilde{b}]$ on 
		$\{ \widetilde{\bx}^{\bj}_{\bi, \bn} \in \widetilde{X}^{P}_{n_1}\times \ldots \times \widetilde{X}^{P}_{n_{d'-1}}\times X^P_{N-|\bn'|_{1}} \times \widetilde{X}^{P}_{n_{d'+1}} \times \ldots \times \widetilde{X}^{P}_{n_d} \}$  
		along the $d'$-direction
		
		\ENDFOR
		
		\STATE Set $ \tilde{b}^{\bj}_{\bi, \bl, \bn} = \tilde{f}^{(\bl)}(\widetilde{\mathbf{x}}^{\bj}_{\bi, \bn}) $
		
		\ENDFOR
		
		\STATE set $f^{(\bl)} (\widetilde{\mathbf{x}}^{\bj}_{\bi, \bn}) =\tilde{f}^{(\bl)} (\widetilde{\mathbf{x}}^{\bj}_{\bi, \bn}) $
		
	\end{algorithmic}
\end{algorithm}


\begin{algorithm}
	\caption{Fast transform from point (and derivative) values $f^{(\bl)}(\widetilde{\mathbf{x}}^{\bj}_{\bi, \bn})$ to hierarchical coefficients $ b^{\bj}_{\bi, \bl, \bn}$ on sparse approximation space in d dimensions} 
	\label{algorithm2}
	\begin{algorithmic}
		\STATE \textbf{Input:} $N$, $d, P, M$, $\widehat{\mathbf{X}}^P_N$ and the point values $\{f^{(\bl)}(\widetilde{\mathbf{x}}^{\bj}_{\bi, \bn}): \widetilde{\mathbf{x}}^{\bj}_{\bi, \bn}\in\widehat{\mathbf{X}}^P_N , \mathbf{0} \leq \bl \leq \bM \}$
		
		\STATE \textbf{Output:}	$\{ b^{\bj}_{\bi, \bl, \bn}: \bn\in\mathbb{N}^{d}_{0}, |\bn|_{1}\leq N, \mathbf{0} \leq \bj \leq \max(2^{\bn-\mathbf{1}}, \mathbf{1}), \mathbf{0}\leq \bi \leq \bP, \mathbf{0}\leq \bl\leq \bM\}$
		
		\STATE  
		\STATE set $\tilde{f}^{(\bl)} (\widetilde{\mathbf{x}}^{\bj}_{\bi, \bn}) =f^{(\bl)}(\widetilde{\mathbf{x}}^{\bj}_{\bi, \bn}) $
		
		\FOR  {$d'=1$ to $d$} 
		
		\FOR {all $\bn'\in \{ \bn'=(n_1, \ldots, n_{d'-1}, n_{d'+1},\ldots, n_{d}), |\mathbf{n}'|_1\leq N \}$}
		
		\STATE one-dimensional transforms $\tilde{b}^{\bj}_{\bi, \bl, \bn} = \mathcal{F}^{-1}_{x_d}[\tilde{f}]$ on 
		$\{ \widetilde{\bx}^{\bj}_{\bi,\bn} \in \widetilde{X}^{P}_{n_1}\times \ldots \times \widetilde{X}^{P}_{n_{d'-1}}\times X^P_{N-|\bn'|_{1}} \times \widetilde{X}^{P}_{n_{d'+1}} \times \ldots \times \widetilde{X}^{P}_{n_d}) \}$  along the $d'$-direction
		
		\ENDFOR
		
		\STATE Set $ \tilde{f}^{(\bl)} (\widetilde{\mathbf{x}}^{\bj}_{\bi, \bn}) = \tilde{b}^{\bj}_{\bi, \bl, \bn} $
		
		\ENDFOR
		
		\STATE set $b^{\bj}_{\bi, \bl, \bn} =\tilde{b}^{\bj}_{\bi, \bl, \bn}$
		
	\end{algorithmic}
\end{algorithm}

\bigskip

Next, we consider the sparse grid space starting from level ``-1", which is given as
\begin{align}
	\widehat{\mathbf{V}}^{K}_{c,N} = \bigoplus_{\substack{ \mathbf{-1}\leq \mathbf{n} \leq \mathbf{N}  \\ |\mathbf{n}|_{1}\leq N-d+1 }} \mathbf{W}^{K}_{c,\mathbf{n}},
\end{align}
with 
	\begin{align}
	\mathbf{W}^{K}_{c,\mathbf{n}} = W^{K}_{c,n_1,x_1}\times \cdots \times W^{K}_{c,n_d,x_d}.
	\end{align} 
	 Following the proof in \cite{sparsedgelliptic}, we can obtain the dimension
\begin{align}
	dim\left(\widehat{\mathbf{V}}^{K}_{c,N} \right)
	=&  1 + \sum_{q=0}^{d-1} \begin{pmatrix} d \\ q \\ \end{pmatrix} K^{d-q} + \sum_{q=0}^{d-1} \sum_{m=0}^{d-q-1}  \begin{pmatrix} d \\ q \\ \end{pmatrix} \begin{pmatrix} d-q \\ m \\ \end{pmatrix} K^m (K+1)^{d-q-m}  \notag \\
	&  \left\{ (-1)^{d-q-m} + \sum_{n=0}^{d-q-1-m} \begin{pmatrix} N-d+q+1 \\ n \\ \end{pmatrix} (-1)^{d-q-1-m-n} 2^{N-d+q+1-n} \right\}.
\end{align} 
Suppose there is an upper bound on the dimension $d\leq d_0$, then there exist constants $c_{d_0}$ and $C_{d_0}$ depending only on $d_0$, such that
\begin{align}
	c_{d_0} \, 2^N (N-d+1)^{d-1} (K+1)^d
	\leq dim\left(\widehat{\mathbf{V}}^{K}_{c,N} \right) 
	\leq C_{d_0} \, 2^N N^{d-1} (K+1)^d.
\end{align} 
Then, we can construct the interpolation operator $\widehat{I}^{P,M}_{c,N}: C^{M}(\Omega_{N})\rightarrowtail \widehat{\mathbf{V}}^{K}_{c,N}$ as
\begin{align}
\label{eq:multiD_sparse2}
	\widehat{\mathcal{I}}^{P,M}_{c,N}[f](\mathbf{x}) 
	= \sum_{\substack{ \mathbf{-1}\leq \bn \leq \mathbf{N}  \\ |\bn|_{1}\leq N-d+1 }} \widetilde{\mathcal{I}}^{P,M}_{c,n_1} \circ \cdots \circ \widetilde{\mathcal{I}}^{P,M}_{c,n_d} [f](\mathbf{x}) 
	= \sum_{ \substack{  \mathbf{-1}\leq \bn \leq \mathbf{N},  \, |\bn|_{1}\leq N-d+1 , \\ \mathbf{0}\leq \bj \leq \bJ_\bn \\ \mathbf{0}\leq \bi \leq \bP_\bn \\ \mathbf{0}\leq \bl \leq \bM_\bn }} b^{\bj}_{c,\bi, \bl, \bn} \varphi^{\bj}_{c,\bi, \bl, \bn} (\mathbf{x}) 
\end{align}
with the set of interpolation points as
\begin{align}
	\widehat{\mathbf{X}}^{P}_{c,N} 
	=& \bigcup_{\mathbf{-1}\leq \bn \leq \mathbf{N},  \, |\mathbf{n}|_{1}\leq N-1} \widetilde{X}^{P}_{c,n_1}\times \cdots \times \widetilde{X}^{P}_{c,n_d}.
\end{align}
Similarly, we can   use the fast algorithm 1 and 2 by replacing  $\mathcal{F}_{x_d} / \mathcal{F}_{x_d}^{-1}$ with $\mathcal{F}_{c,x_d} / \mathcal{F}_{c,x_d}^{-1}$, and the space to $\{ \widetilde{\bx}^{\bj}_{c,\bi,\bn} \in \widetilde{X}^{P}_{c,n_1}\times \ldots \times \widetilde{X}^{P}_{c,n_{d'-1}}\times X^P_{c,N-d+1-|\bn'|_{1}} \times \widetilde{X}^{P}_{c,n_{d'+1}} \times \ldots \times \widetilde{X}^{P}_{c,n_d}) \}.$

\subsection{Error estimates}

In this subsection, we   study the approximation error of the interpolation operator on the sparse approximation space. Similar analysis has been performed in \cite{bungartz2004sparse, schwab2008sparse,guo2016sparse}, and the key is to gather one-dimensional approximation results and then work on multi-dimensional case.

We first review some approximation results in 1D. Suppose $f\in W^{r+1,p}(0,1)$. Then based on the Bramble-Hilbert lemma, there exists a constant $C_1$ which is independent of $n$, such that for any integer $1\leq q \leq \min(r,K)$  and $p\in[1,\infty]$, $s\in [0,q]$ 
\begin{align*}
	\|f-\mathcal{I}^{P,M}_{n}[f]\|_{W^{s,p}(I^{j}_{n})} \leq C_{1} \, h_{n}^{q+1-s} |f|_{W^{q+1,p}(I^{j}_{n})},  
\end{align*}
with $h_{n}=2^{-n}.$ Therefore, we can obtain the bound of the increment interpolation operator   $\widetilde{\mathcal{I}}^{P,M}_{n}, n\geq1$, 
\begin{align}
\label{eq:error_interpolation_1D1}
	\|\widetilde{\mathcal{I}}^{P,M}_{n}[f]\|_{W^{s,p}(I^{j}_{n})} 
	&\leq \|f-\mathcal{I}^{P,M}_{n-1}[f]\|_{W^{s,p}(I^{j}_{n})} + \|f-\mathcal{I}^{P,M}_{n}[f]\|_{W^{s,p}(I^{j}_{n})} \notag\\
	&\leq
	C_1 \left(1+2^{q+1-s}\right)\,h_{n}^{q+1-s}  \, |f|_{W^{q+1,p}(I^{\lfloor j/2\rfloor}_{n-1})}
\end{align} 
When $n=0$, using the definition of $\widetilde{\mathcal{I}}^{P,M}_{0}$, we can obtain that
\begin{align}
\label{eq:error_interpolation_1D2}
	\|\widetilde{\mathcal{I}}^{P,M}_{n}[f]\|_{W^{s,p}(0,1)} \leq C_{2} \|f\|_{W^{M,\infty}(0,1)},
\end{align}
where the constant $C_2$ that depends on $K$. We denote $\bar{C}=\max(C_1,C_2),$
and similarly
\begin{align}
\label{eq:error_interpolation_c_1D}
	\| \widehat{\mathcal{I}}^{P,M}_{c,n}[f] \|_{W^{s,p}(I^{j}_{n})} \leq 
	\left\{ \begin{array}{ll}
	\bar{C} \|f\|_{W^{0,\infty}(0,1)} \leq \bar{C} \|f\|_{W^{M,\infty}(0,1)}, & n=-1,\\
	\bar{C} \|f\|_{W^{M,\infty}(0,1)}, & n=0, \\
	\bar{C} \left(1+2^{q+1-s}\right)\,h_{n}^{q+1-s}  \, |f|_{W^{q+1,p}(I^{\lfloor j/2\rfloor}_{n-1})}, & n\geq1.\\
	\end{array}\right.
\end{align}

Next, we introduce some notations for error estimates in multi-dimensions. Consider a set $L=\{ d_1, \cdots, d_r\} \subset \{1,\ldots,d\}$ with $|L|=r$, we define $L^c$ to be the complement set of $L$ in $\{1, \ldots, d\}$, denoted as $L^c= \{d'_1,\ldots, d'_{d-r}\}$. Then, we define a semi-norm on domain $\Omega=\Omega_{x_1}\times\cdots \times\Omega_{x_d}$ as
\begin{align*}
|f|_{W^{q+1,M,p,L}(\Omega)}= \max_{\substack{0\leq m_{s} \leq M \\ x_{d'_s} \in \Omega_{d_s'} \\ 1\leq s \leq d-r}} \left\{ \int_{\Omega_{d_r}} \cdots \int_{\Omega_{d_1}} \left| \partial^{m_{1}}_{x_{d'_1}}\cdots \partial^{m_{d-r}}_{x_{d'_{d-r}}} \partial^{q+1}_{x_{d_1}} \cdots \partial^{q+1}_{x_{d_r}} f \right|^p dx_{d_1} \cdots d x_{d_r}\right\}^{1/p} ,
\end{align*}
and 
\begin{align*}
|f|_{W^{q+1,M,p}(\Omega)} = \max_{1\leq r\leq d} \left( \max_{\substack{L\subset\{1,\cdots,d\} \\ |L|=r}} |f|_{W^{q+1,M,p,L}(\Omega)} \right).
\end{align*}

 	
Then, we have the following theorem, quantifying the interpolation error in multi-dimensions. Many discussions and notations below are similar to Lemma 3.2 in \cite{guo2016sparse}.

\begin{thm}
\label{thm:1}
	Suppose $f\in W^{r+1,p}(\Omega),$ for any integer $1\leq q \leq \min(r,K)$  and $p\in[1,\infty]$,  we have 
	\begin{align}
	\label{eqn:relation}
	& \|f - \widehat{\cI}^{P,M}_{N}[f]\|_{L^{p}(\Omega_N)} \leq    \left\{ \tilde{C} + \bar{C}^d (N+1)^d \, \left(2+2^{q+2}\right)^{d} 2^{-(q+1)} \right\} \,  2^{-N (q+1)} |f|_{W^{q+1,M,p}(\Omega)} ,\\
	\label{eqn:relation3}
	& |f - \widehat{\cI}^{P,M}_{N}[f]|_{H^{1}(\Omega_N)} \leq   \left\{ \tilde{C} + \bar{C}^d  d^{3/2}\, \left(2+2^{q+2}\right)^{d} 2^{d-1} \right\} \,  2^{-N q} |f|_{W^{q+1,M,2}(\Omega)} 
	\end{align}
	and
	\begin{align}
	\label{eqn:relation2}
	 \|f - \widehat{\cI}^{P,M}_{c,N}[f]\|_{L^{p}(\Omega_N)}
	&\leq   \left\{\tilde{C} + \bar{C}^d \left( N+1+2^{-(q+1)} \right)^d (2+2^{q+2})^{d} \, 2^{(d-2)(q+1)} \right\}\,  2^{-N(q+1)}   |f|_{W^{q+1,M,p}(\Omega)} , \\
	\label{eqn:relation4}
	 |f - \widehat{\cI}^{P,M}_{c,N}[f] |_{H^{1}(\Omega_N)}
	&\leq   \left\{\tilde{C} + d^{3/2}  \, \bar{C}^{d}  \left( 1+2^{-(q+1)}\right)^d  \,  (2+2^{q+2})^{d} \, 2^{(d-1)(q+1)}   \right\}\,  2^{-qN}  |f|_{W^{q+1,M,2}(\Omega)} 
	\end{align}
	when $N\geq1, d\geq2$, where $\tilde{C}$ is a constant that depends on $p, q, s$, but not on $N$. 
\end{thm} 
\begin{proof}
The proof can be found in Appendix \ref{sec:proof}.
\end{proof}

\subsection{The adaptive sparse grid collocation scheme}

%



In this subsection, we  describe the adaptive sparse grid collocation method.
For   solutions with less smoothness, adaptive methods are necessary to capture the fine local structures. 
The main idea of the algorithm is not to use $ \widehat{\bV}_{c,N}^K$ in a pre-determined fashion, but rather to choose a subspace of $ \bV_N^K$ adaptively. 
The setup of the algorithm is very similar to \cite{guo2017adaptive} for adaptive projection method, including the data structures.
Given a maximum mesh level $N$ and an accuracy threshold $\varepsilon>0$,  we use the adaptive multiresolution interpolation algorithm  to get the numerical solution $f_h(\bx)$ for the exact function $f(\bx)$.  The details of the method are listed below in Algorithm \ref{algorithm3}.

\begin{algorithm}[htb]
	\caption{ Adaptive multiresolution interpolation} \label{algorithm3}
	\begin{algorithmic}
		\STATE {\bf Input:} Function $f(\bx)$.

		\STATE {\bf Parameters:} Maximum level $N,$ polynomial degree $K,$  error threshold   $\varepsilon.$

		\STATE {\bf Output:} Hash table $H$, leaf table $L$ and interpolating  solution $f_h(\bx) \in \bV_{N,H}^K.$

\STATE
\begin{enumerate}
	\item Interpolate $f(\bx)$ onto the coarsest level of mesh, e.g., level -1, and compute the  hierarchical coefficients (surplus) $ b^{\mathbf{0}}_{c,\mathbf{0},\mathbf{0},\textbf{-1}}$
	based on point value at $\widetilde{\bx}^{\textbf{0}}_{c,\textbf{0},\textbf{0},\textbf{-1}}$.
	Add all elements to the hash table $H$ (active list). Define an element without children as a leaf element, and add all the leaf elements to the leaf table $L$ (a smaller hash table). 
	
	\item For each leaf element $\bV_\bn^\bj=span \{\varphi^\bj_{c,\bi,\bl, \bn}, \mathbf{0}\leq\bi\leq\bP_{\bn}, \mathbf{0}\leq\bl\leq\bM_{\bn} \}$ in the leaf table, if the refinement criteria holds
	\begin{align}
	\label{eq:l1}
	& \sum_{\substack{\mathbf{0}\leq\bi\leq\bP_{\bn} \\ \mathbf{0}\leq\bl\leq\bM_{\bn}}} |b^\bj_{c,\bi,\bl,\bn}| \cdot \|\varphi^\bj_{c,\bi,\bl,\bn}(\bx)\|_{L^1(\Omega)} > \varepsilon, \\
	or \quad 	& \left(\sum_{\substack{\mathbf{0}\leq\bi\leq\bP_{\bn} \\ \mathbf{0}\leq\bl\leq\bM_{\bn}}} |b^\bj_{c,\bi,\bl,\bn}|^2 \cdot \|\varphi^\bj_{c,\bi,\bl,\bn}(\bx)\|^2_{L^2(\Omega)} \right)^{\frac{1}{2}} > \varepsilon,  \label{eq:l2} \\
	or \quad & \sum_{\substack{\mathbf{0}\leq\bi\leq\bP_{\bn} \\ \mathbf{0}\leq\bl\leq\bM_{\bn}}} |b^\bj_{c,\bi,\bl,\bn}| \cdot \|\varphi^\bj_{c,\bi,\bl,\bn}(\bx)\|_{L^{\infty}(\Omega)} > \varepsilon, \label{eq:l8}
	\end{align}
	then we consider its child  elements: for a child element $\bV_{\bn'}^{\bj'}$, if it has not been added to the table $H$, then  compute the detail coefficients $ \{ b^{\bj'}_{c,\bi,\bl,\bn'}, \mathbf{0}\leq\bi\leq\bP_{\bn}, \mathbf{0}\leq\bl\leq\bM_{\bn} \}$ with fast transform  algorithm 2 based on point values $f^{(\bl)}(\widetilde{\mathbf{x}}^{\mathbf{j'}}_{\mathbf{i},\mathbf{n'}})$  and add $\bV_{\bn'}^{\bj'}$ to both table $H$ and table $L$. The cost of adding each element is $\mathcal{O}(d(K+1)^{d+1})$.  For its parent elements in $H$, we increase the number of children by one.
	
	\item Remove the parent elements from table $L$ for all the newly added elements. 
	\item Repeat step 2 - step 3, until no element can be further added.
	\item (Optional) Coarsen the points which are insignificant. The hash table $H$ is traversed, if the coarsening criterion for each point
	\begin{align}
	&|b^\bj_{c,\bi,\bl, \bn}| \cdot \|\varphi^\bj_{c,\bi,\bl, \bn}(\bx)\|_{L^s(\Omega)}<\eta, \,\, s=1, 2, or \,  \infty
	\label{eq:l1_c_pt}
	\end{align}
	is satisfied, where $\eta$ is a prescribed error constant, then we remove the point from the element in $H$, and set the associated coefficients $b^{\bj}_{c,\bi,\bl, \bn}=0$. If this element is in the leaf table $L$, we also remove the point from the element in $L$.

\end{enumerate}
\end{algorithmic}
\end{algorithm}

In the algorithm, the formulas in \eqref{eq:l1},  \eqref{eq:l2}, and  \eqref{eq:l8} correspond to $L^1, L^2$ and $L^{\infty}$ norm based refinement criteria  in  \cite{guo2017adaptive}, respectively.
When the adaptive interpolation algorithm completes, it will generate a   hash table $H$, leaf table $L$ and   $f_h(\bx)=\sum_{\varphi^\bj_{c,\bi,\bl,\bn} \in H} b^\bj_{c,\bi,\bl, \bn} \varphi^\bj_{c,\bi,\bl, \bn}(\bx)$.  We denote the approximation space  $\bV^K_{N,H}=\textrm{span}\{\varphi^\bj_{c,\bi,\bl, \bn} \in H\}$ and it is a subspace of $\bV^K_N.$  In practice, $\eta$ is chosen to be smaller than $\varepsilon$ for safety. In the simulations presented in this paper, we use $\eta = \varepsilon/10$. 

If the solution are evolved in time, refinement step will be done according to $f_h^{n}$ at each time level $t^n$. We traverse the hash table $H$ and if an element $\bV^\bj_{\bn}=\{\varphi^\bj_{c,\bi,\bl, \bn}, \mathbf{0}\leq\bi\leq\bP_{\bn}, \mathbf{0}\leq\bl\leq\bM_{\bn} \}$ satisfies the refinement criteria \eqref{eq:l1},  \eqref{eq:l2}, or \eqref{eq:l8}, 
indicating that such an element becomes significant,
then we need to add its children elements
to $H$ and $L$  provided they are not added yet, and compute the numerical solutions of the associated points. We also need to make sure that all the parent elements of the newly added element are in $H$ (i.e., no ``hole" is allowed in the hash table) and increase the number of children for all its parent elements by one. This step  generates the updated hash table $H$ and leaf table $L$.

In this adaptive approach, algorithm 1 or 2 is also applied to the fast transform between point values and hierarchical coefficients. The  computational cost for transformation is $\mathcal{O}(dn(K+1))$, where $n$ is the total degree of freedom.  


\begin{rmk}
	Step 5  is optional and employed for function interpolations.  In particular, since step 5 will coarsen the points within some elements, this may result in some ``holes" in these elements, i.e. the space is no longer downward closed, so it should not be used with time evolution problems. Instead, an element-wise coarsening procedure as in  \cite{guo2017adaptive} should be used.
\end{rmk}

\begin{rmk}
	The reason we introduce level -1 is based on the simulation to high dimensional case. Suppose we start with level 0 with $d=10, K=3$, there will be $(K+1)^d = 1048576$  degrees of freedom in the coarsest element  at  level 0, and the number will increase dramatically along with the dimension $d$. To reduce the number of degrees of freedom, we therefore consider -1 level and we just put one degree of freedom at level -1 in each direction. Then, the degree of freedom in the coarsest element at level -1 is 1. This technique has been used before in \cite{garcke2006dimension}. 
\end{rmk}

\section{Numerical examples}
\label{sec:numerical}


In this section, we demonstrate the performance of our algorithms through several examples, including function interpolation and integration, and some benchmark test problems in uncertainty quantification (UQ). For all numerical examples, we have used the method starting from mesh level ``-1". Various types of interpolation orders have been tested, and the details of the corresponding collocation points and bases are described in the Appendix \ref{sec:append1}. In particular, for Lagrange interpolation ($M=0$), we test $P=1, 2, 3$ (see Appendix \ref{sec:append1.2} type 1, \ref{sec:append2.2} and \ref{sec:append2.4}), which implies that the polynomial order $K=1, 2, 3.$ For Hermite interpolation, we test $P=1, M=1,$ (see Appendix \ref{sec:append2.3}) which implies that the polynomial order $K=3.$

%
%
%
%
%
%
%
%
%
%
%

\subsection{Function interpolation and integration}

In this subsection, we demonstrate the performance of the (adaptive) sparse grid   method in function interpolation and integration. The error has been calculated using randomly generated  $100000$ sample points.
We will use the following five functions which have been considered in previous work \cite{jakeman2012local, ma2009adaptive, bhaduri2018stochastic}.

$$f_0({\bf x}) = \exp \left(  \sin(2 \pi (x_1+x_2))   \right), \quad  {\bf x} \in [0,1]^2,$$

$$f_1({\bf x}) = \frac{1}{|0.3-x_1^2-x_2^2|+0.1}, \quad  {\bf x} \in [0,1]^2,$$

$$f_2( {\bf x} ) = \exp \left(  -\sum_{i=1}^{d} c_i^2 (x_i - 0.51)^2 \right), \quad  {\bf x} \in [0,1]^d,$$

$$f_3( {\bf x} ) = \exp \left(  -\sum_{i=1}^{d} c_i |x_i - 0.51| \right), \quad  {\bf x} \in [0,1]^d,$$

$$f_4( {\bf x} ) = 	\left\{\begin{array}{ll}
0, &  x_1>0.5 \,\, \text{or} \,\, x_2>0.5\\
\exp \left(  \sum_{i=1}^{d} c_i x_i \right), & \text{otherwise}\\
\end{array}\right., \quad  {\bf x} \in [0,1]^d, $$
where $d$ is the dimension.
Functions $f_0( {\bf x} ) $ and $f_2( {\bf x} ) $ are smooth. Functions $f_1( {\bf x} ) $ and $f_3( {\bf x} ) $ have   jump  discontinuities in the derivatives while $f_4( {\bf x} )$ is discontinuous.

We first verify the accuracy of the sparse grid collocation method by interpolating function $f_0( {\bf x} ).$  The $L^1, L^2, L^\infty$ and $H^1$ errors and orders of various interpolations are presented in Table \ref{table:interp_2D}. The orders are calculated with respect to  $h_N$. We can see all interpolations achieve $L^1, L^2$ and $L^\infty$ accuracy order slightly less than $K+1,$ and  $H^1$ accuracy of $K-$th order  as predicted by Theorem \ref{thm:1}.

\begin{table}[htp]
	
	\caption{$L^1, L^2, L^{\infty}$ and $H^1$ errors and orders of accuracy for Lagrange and Hermite interpolation  of function $f_0({\bf x})$. 
		$N$ is the number of mesh levels,   $K$ is the polynomial order.  The orders are calculated with respect to  $h_N$. 
	}
	\centering
	\begin{tabular}{|c|c|c c|c c|c c|c c|}
		\hline
		& &       \multicolumn{6}{c|}{Lagrange}     &       \multicolumn{2}{c|}{Hermite}\\
	
		\hline
		$N$& $h_N$&   \multicolumn{2}{c|}{$ K=1$}    &     \multicolumn{2}{c|}{$ K=2$}  &     \multicolumn{2}{c|}{$ K=3$}  
		&     \multicolumn{2}{c|}{$ K=3$}\\
	   \hline
        & & $L^1$ error & order& $L^1$ error & order & $L^1$  error & order & $L^1$  error & order\\
		\hline
		6 &$1/64$  &  1.35E-01	&   --       &   7.15E-03	&	--     &	5.19E-04	&	--     & 7.54E-03	& --\\
		7 &$1/128$ &  4.74E-02	&	1.51     &	 1.22E-03	&	2.55   &	4.03E-05	&	3.69   & 7.55E-04	& 3.32\\
		8 &$1/256$ &  1.60E-02	&	1.57     &	 2.17E-04	&	2.49   &	4.74E-06	&	3.09   & 1.21E-04	& 2.64\\
		9 &$1/512$ &  5.33E-03	&	1.59     &	 3.13E-05	&	2.79   &	3.36E-07	&	3.82   & 1.01E-05	& 3.58\\
		10 &$1/1024$&  1.68E-03   &	1.67     &	 4.45E-06	&	2.81   &	2.36E-08	&	3.83   & 8.33E-07	& 3.60\\
		\hline
		& & $L^2$ error & order& $L^2$ error & order & $L^2$  error & order & $L^2$  error & order\\
		\hline
		6 &$1/64$  &  1.89E-01	&   --       &   1.15E-02	&	--     &	9.02E-04	&	--     & 1.55E-02	& --\\
		7 &$1/128$ &  6.45E-02	&	1.55     &	 1.85E-03	&	2.64   &	6.56E-05	&	3.78   & 1.16E-03	& 3.74\\
		8 &$1/256$ &  2.35E-02	&	1.46     &	 3.50E-04	&	2.40   &	8.05E-06	&	3.03   & 2.07E-04	& 2.49\\
		9 &$1/512$ &  7.76E-03	&	1.60     &	 4.98E-05	&	2.81   &	5.47E-07	&	3.88   & 1.65E-05	& 3.65\\
		10 &$1/1024$&  2.43E-03   &	1.68     &	 7.24E-06	&	2.78   &	3.96E-08	&	3.79   & 1.41E-06	& 3.55\\
		\hline
		& & $L^{\infty}$ error & order& $L^{\infty}$ error & order & $L^{\infty}$  error & order & $L^{\infty}$  error & order\\
		\hline
		6 &$1/64$  &  5.69E-01	&   --       &   4.69E-02	&	--     &	4.37E-03	&	--     & 7.42E-02	& --\\
		7 &$1/128$ &  1.98E-01	&	1.52     &	 7.58E-03	&	2.63   &	3.53E-04	&	3.63   & 4.09E-03	& 4.18\\
		8 &$1/256$ &  9.37E-02	&	1.08     &	 2.01E-03	&	1.92   &	4.61E-05	&	2.94   & 9.51E-04	& 2.10\\
		9 &$1/512$ &  2.98E-02	&	1.65     &	 3.14E-04	&	2.68   &	3.17E-06	&	3.86   & 7.34E-05	& 3.70\\
		10 &$1/1024$&  1.02E-02   &	1.55     &	 5.53E-05	&	2.51   &	2.74E-07	&	3.53   & 7.65E-06	& 3.26\\
		\hline
		& & $H^1$ error & order& $H^1$ error & order & $H^1$  error & order & $H^1$  error & order\\
		\hline
		6 &$1/64$  &  3.69E+00	&   --       &   4.91E-01	&	--     &	5.97E-02	&	--     & 3.45E-01	& --\\
		7 &$1/128$ &  1.84E+00	&	1.00     &	 1.29E-01	&	1.93   &	7.12E-03	& 3.07   & 4.05E-02	& 3.09\\
		8 &$1/256$ &  9.76E-01	&	0.91     &	 3.73E-02	&	1.79   &	1.23E-03	&	2.53   & 9.14E-03	& 2.15\\
		9 &$1/512$ &  4.99E-01	&	0.97     &	 9.35E-03	&	2.00   &	1.49E-04	&	3.05   & 1.07E-03	& 3.09\\
		10 &$1/1024$&  2.51E-01   &	0.99     &	 2.35E-03	&	1.99   &	1.89E-05	&	2.98   & 1.38E-04	& 2.95\\
		\hline

	\end{tabular}
	\label{table:interp_2D}
\end{table}

Then we consider the adaptive sparse grid method. Function $f_1( {\bf x} ) $ has a 1D singularity that is not along the grid directions. It is well known that the standard sparse grid method without adaptivity cannot resolve such singular or discontinuous profiles. Here, we fix $N=11,\varepsilon=10^{-4}$ and compare the performance of the scheme with $L^1, L^2$, and $L^{\infty}$ norm based criteria. We present the surface and adaptive grids based on different criteria with $P=2$ for function  $f_1( {\bf x} ) $ in Figure \ref{fig:surface_point_2D_Ma_N11}.
The grid with $L^1$ norm criteria is the most sparse, but the surface profile is slightly worse than the other two criteria. The $L^{\infty}$ norm based criteria use the  most DoFs but do not provide the significantly better resolution than $L^2$ norm based criteria. Based on the performance and cost of the three criteria, we will use $L^{2}$ norm based criteria in all simulations.

In Figure \ref{fig:surface_point_2D_Ma_N11}(d) and Figure \ref{fig:interp_fun_2D_grid}(a)-(c), we plot the adaptive grids by  Lagrange $P^1, P^2, P^3$ and Hermite $P^3$ bases for function $f_1( {\bf x} ) $.  In all cases, we can observe that adaptive methods can   capture the right positions of the singularities.
Compared with linear interpolation, quadratic and cubic interpolations are more concentrated  near the singularity. Comparing Hermite $P^3$ with Lagrange $P^3$ interpolation, the Hermite methods have more compact representation near the singularity, although with larger DoFs.  
We also compare the efficiency of the methods by  plotting the $L^2$ error  vs DoFs for all four methods. Here the error parameter $\varepsilon$ varies from $1.0E-2$ to $1.0E-6$ for Lagrange $P^1$ bases and $1.0E-1$ to $1.0E-5$ for Lagrange $P^2, P^3$ and Hermite $P^3$ bases. We consider the $L^2$ error on the whole domain and the $L^2$ error in smooth region excluding the part of the domain that is within $0.1$ distance to the singularity.
The results are shown in  Figure \ref{subfigb}. For regular $L^2$ error, Lagrange $P^2, P^3$ seem more efficient than Lagrange $P^1$ and Hermite $P^3$ interpolations. When we remove the singular regions and when $\varepsilon$ is small, the performance of Hermite $P^3$ method is similar to Lagrange $P^2, P^3$ methods. Based on the least square linear curve fitting function ``lsqcurvefit'' in Matlab, we obtain the slopes of the curves in Figure \ref{subfigb}. The slopes are -1.33, -1.59, -1.51, -1.85 for Lagrange $P^1, P^2, P^3$ and Hermite $P^3$ bases, respectively, with the whole domain, while the slopes are -1.58, -1.82, -1.83, -2.26 for Lagrange $P^1, P^2, P^3$ and Hermite $P^3$ bases, respectively, without singular regions. This example shows that for low dimensional functions in $C^0$ space, the higher order methods are preferred compared to $P^1$ collocation methods.

\begin{figure}[htp]
	\begin{center}
		\subfigure[$L^1$ criteria] {\includegraphics[width=.46\textwidth]{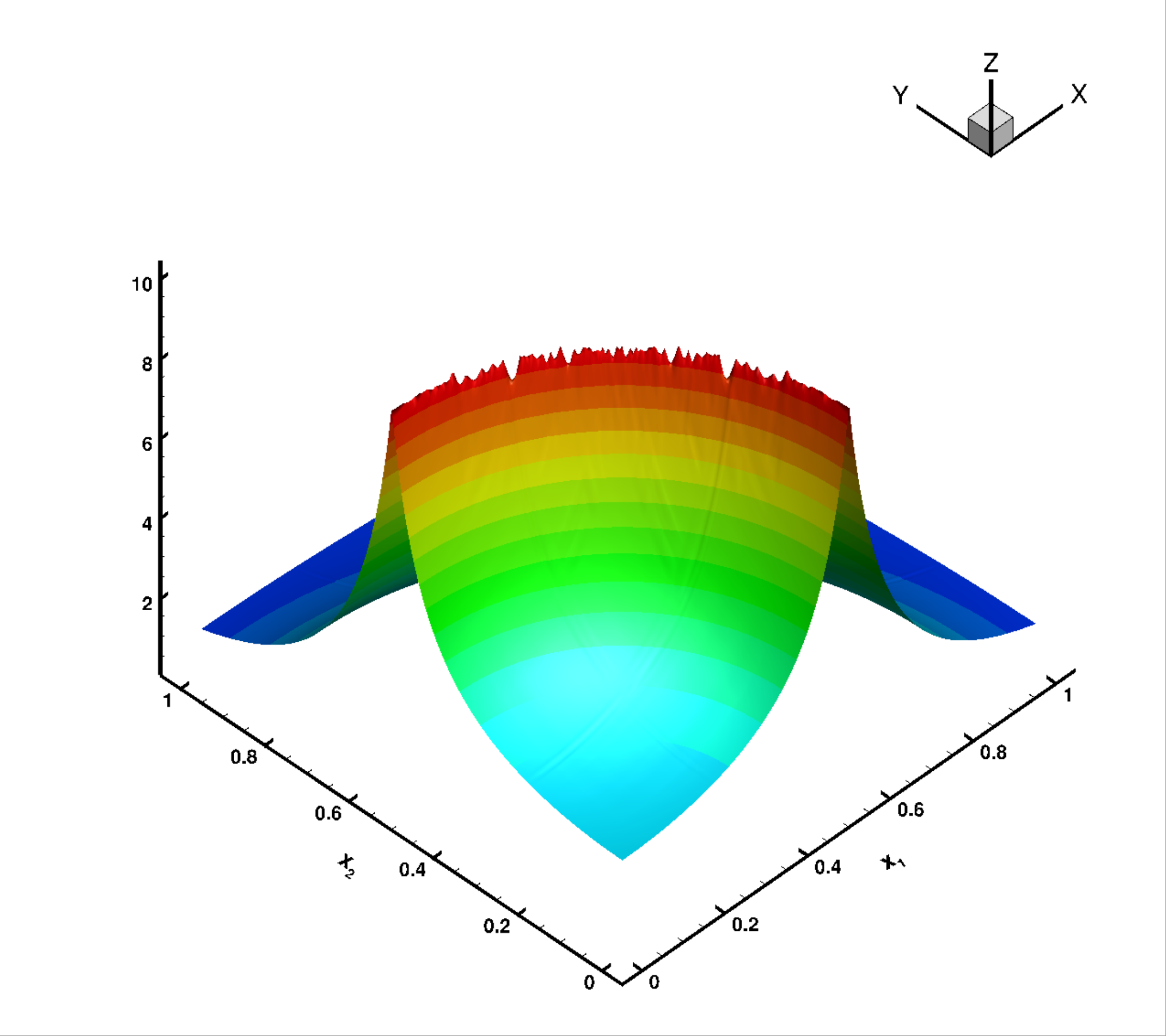}}
		\subfigure[$L^1$ criteria (1545 DoF) ] {\includegraphics[width=.46\textwidth]{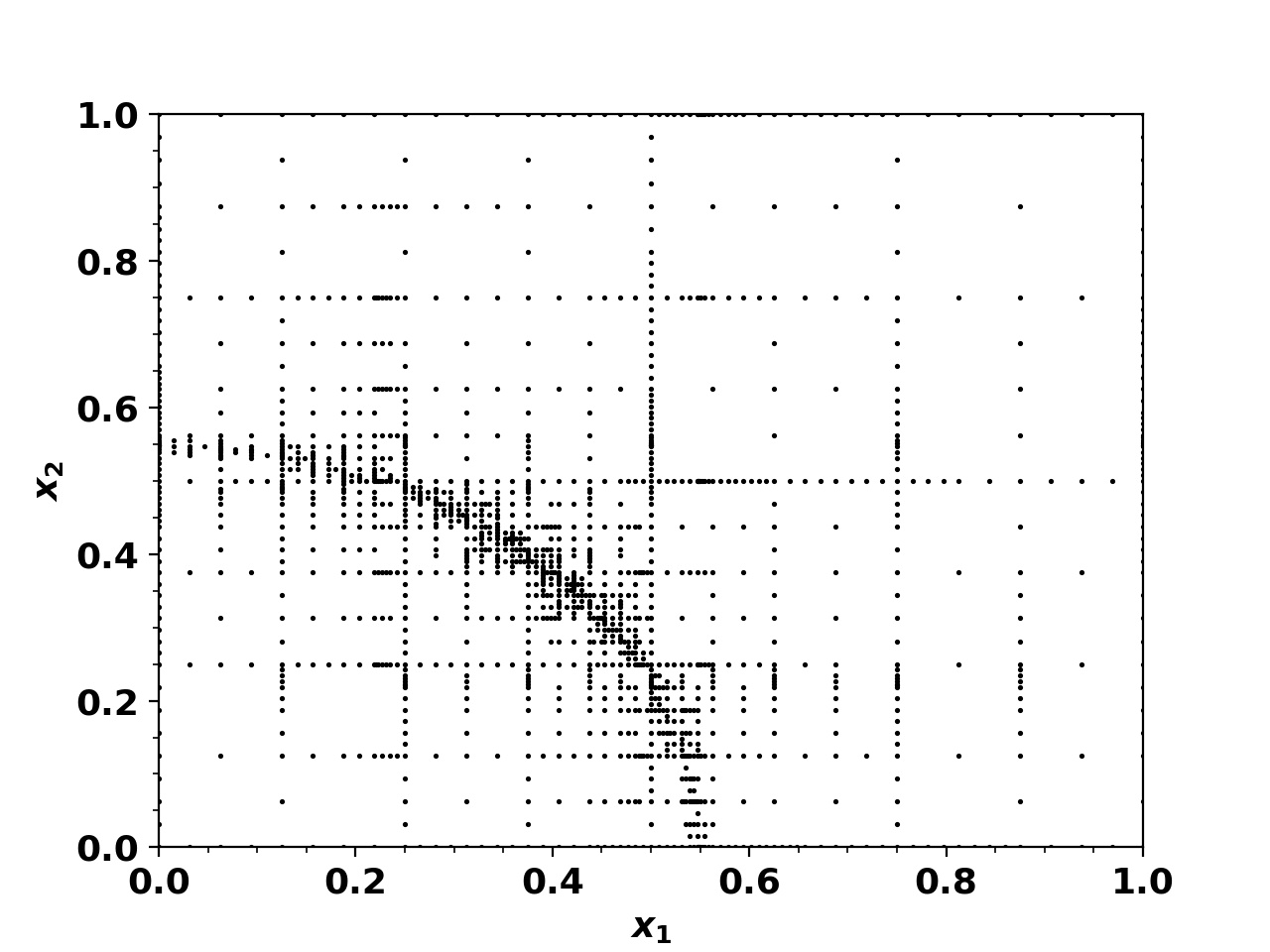}}\\
		\subfigure[$L^2$ criteria] {\includegraphics[width=.46\textwidth]{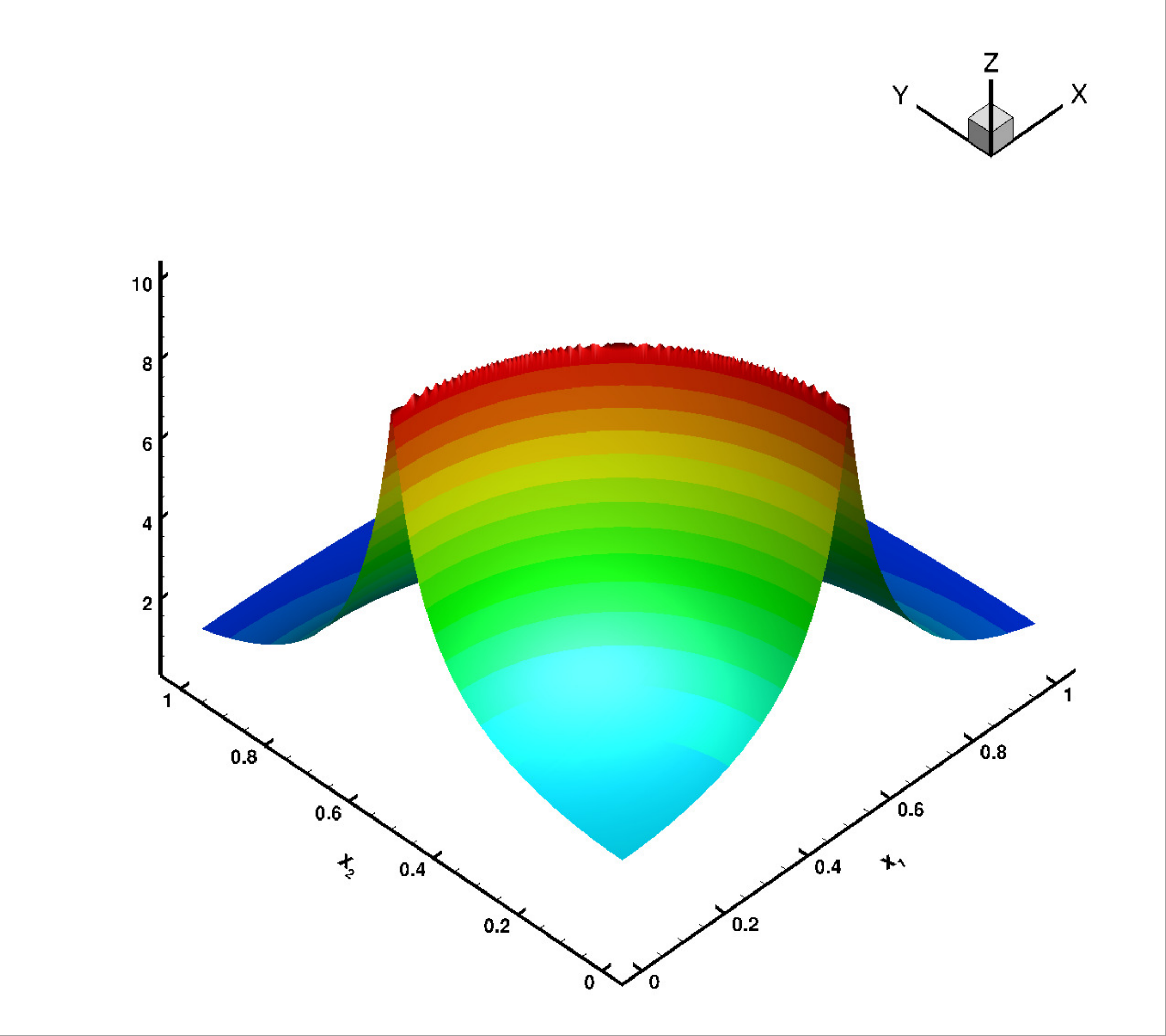}}
		\subfigure[$L^2$ criteria (8000 DoF) ] {\includegraphics[width=.46\textwidth]{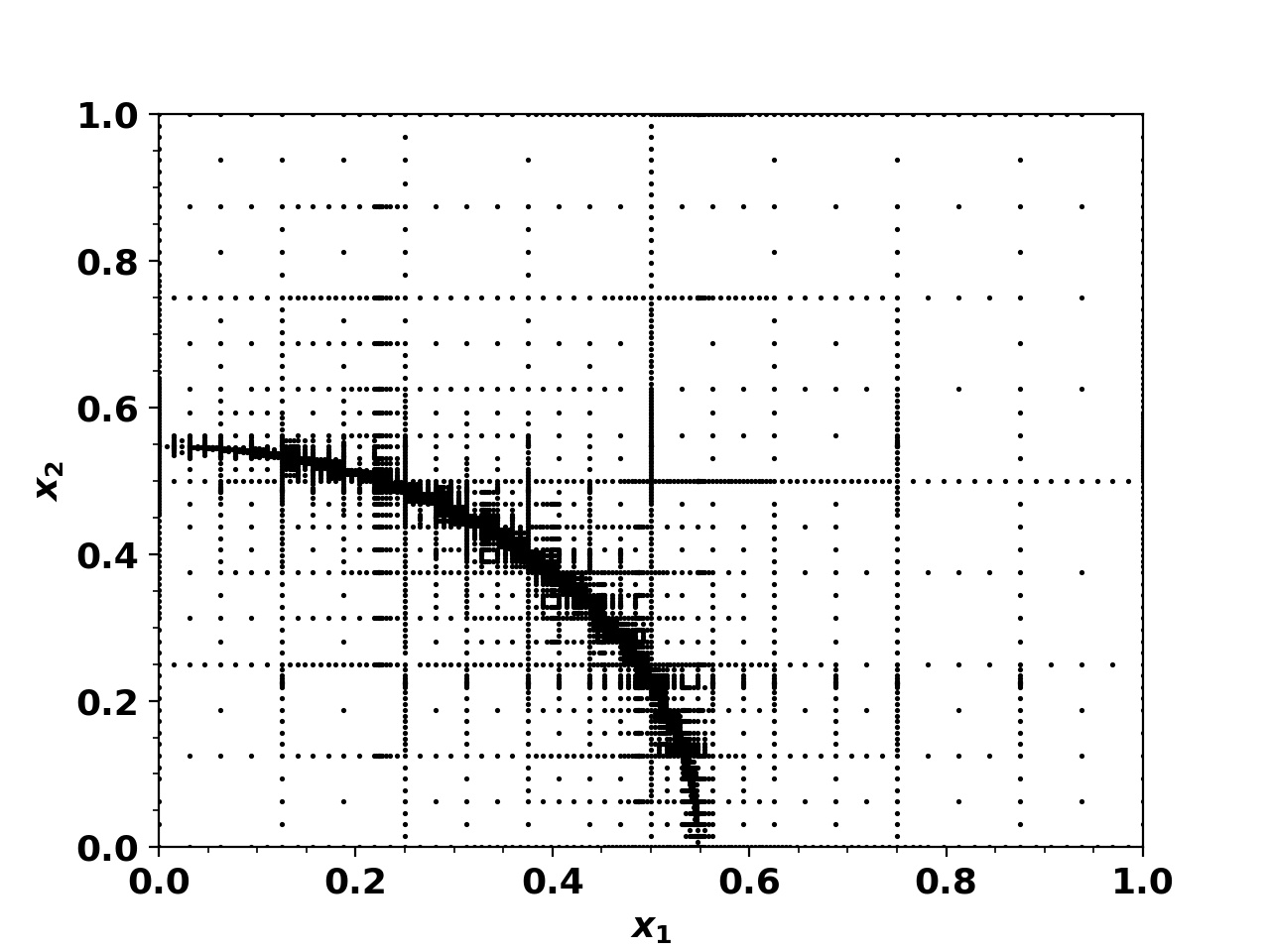}}\\
		\subfigure[$L^{\infty}$ criteria] {\includegraphics[width=.46\textwidth]{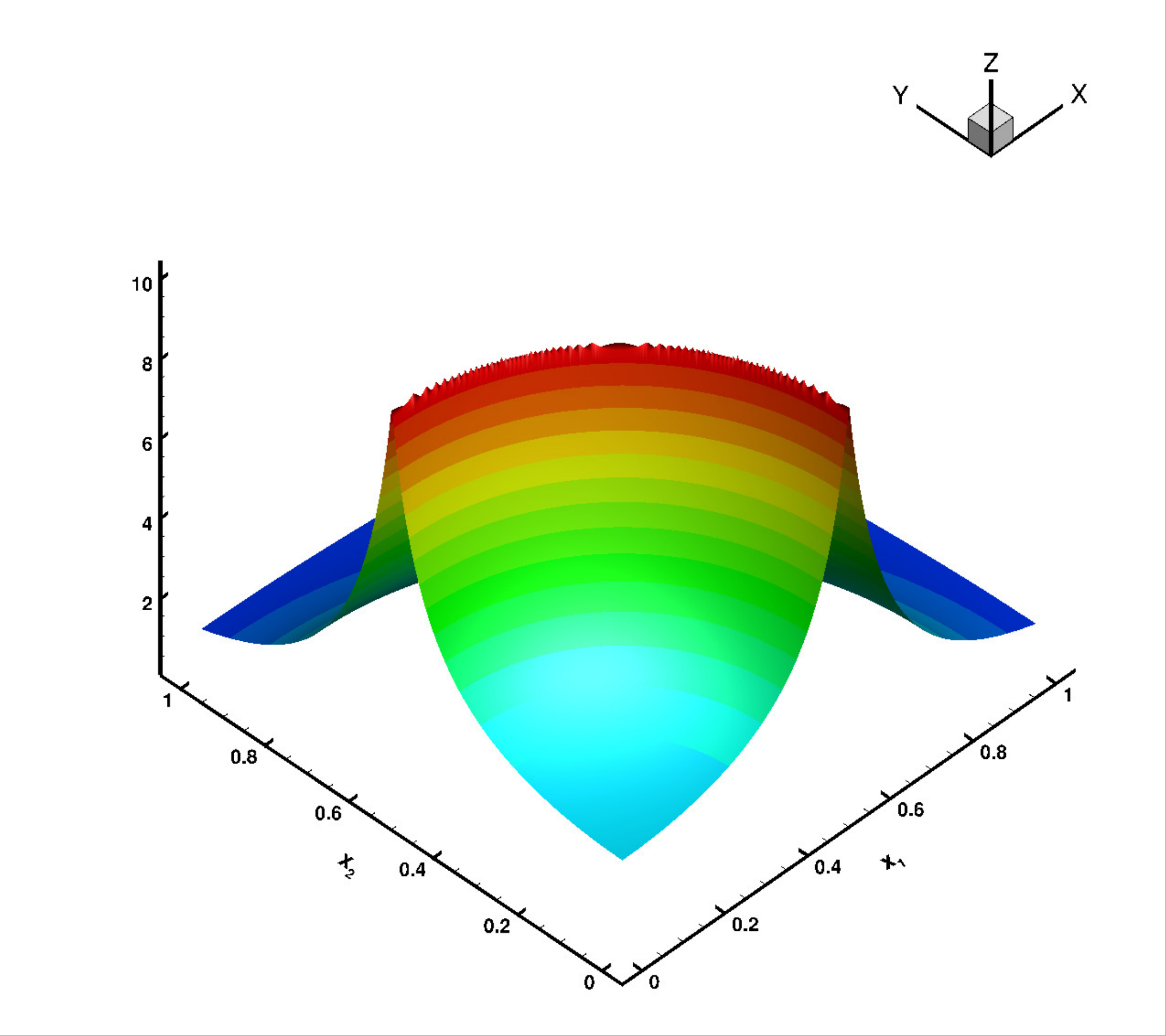}}
		\subfigure[$L^{\infty}$ criteria  (39407 DoF) ] {\includegraphics[width=.46\textwidth]{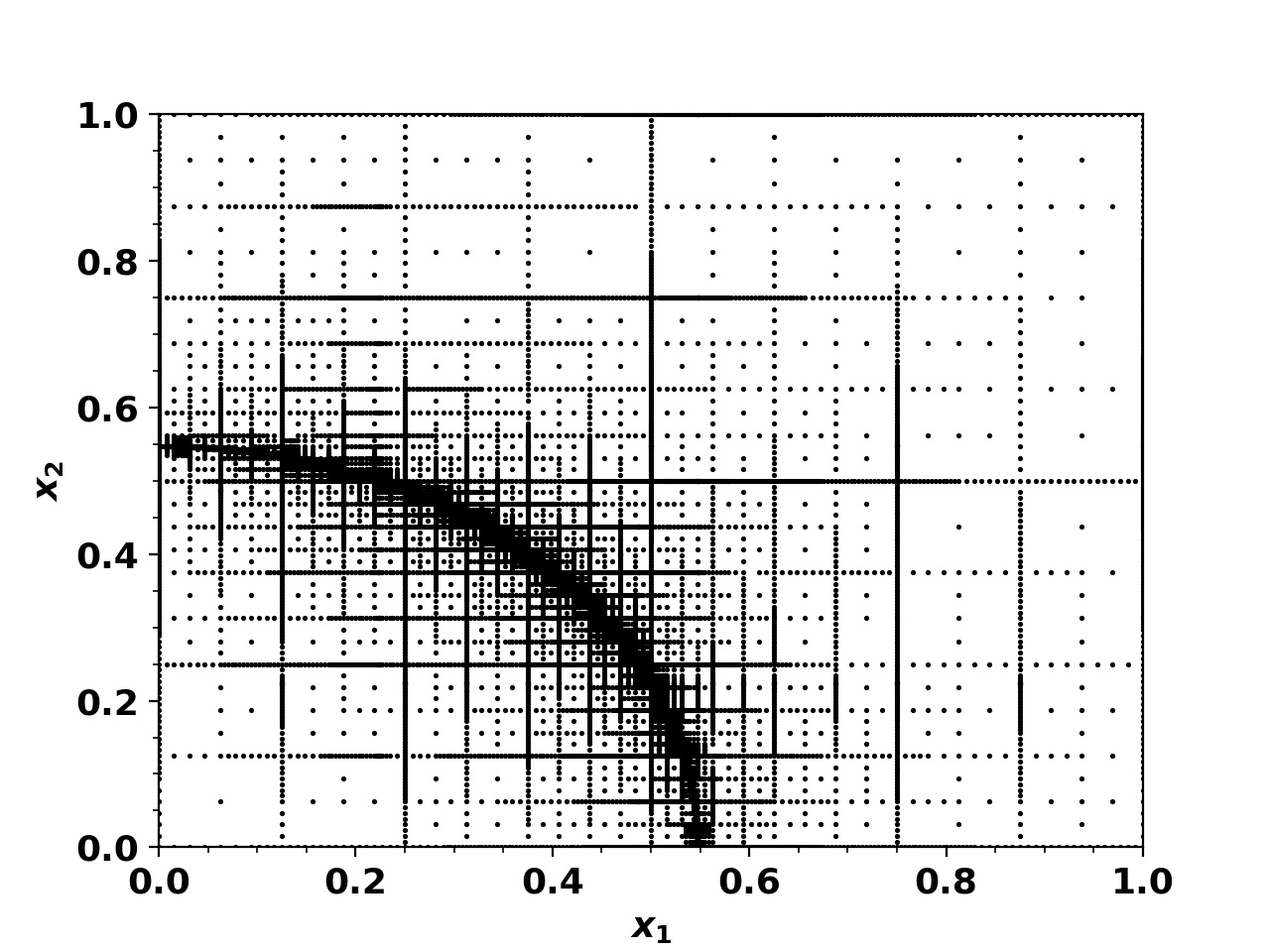}}\\
	\end{center}
	\caption{Surface and adaptive grids for function $f_1( {\bf x} )$ based on different criteria. $\varepsilon = 10^{-4}, N=11$, Lagrange $P^2$. }
	\label{fig:surface_point_2D_Ma_N11}
\end{figure}

\begin{figure}[]
	\begin{center}
		\subfigure[Lagrange $P^1$ (5456 DoF)]{\includegraphics[width=.49\textwidth]{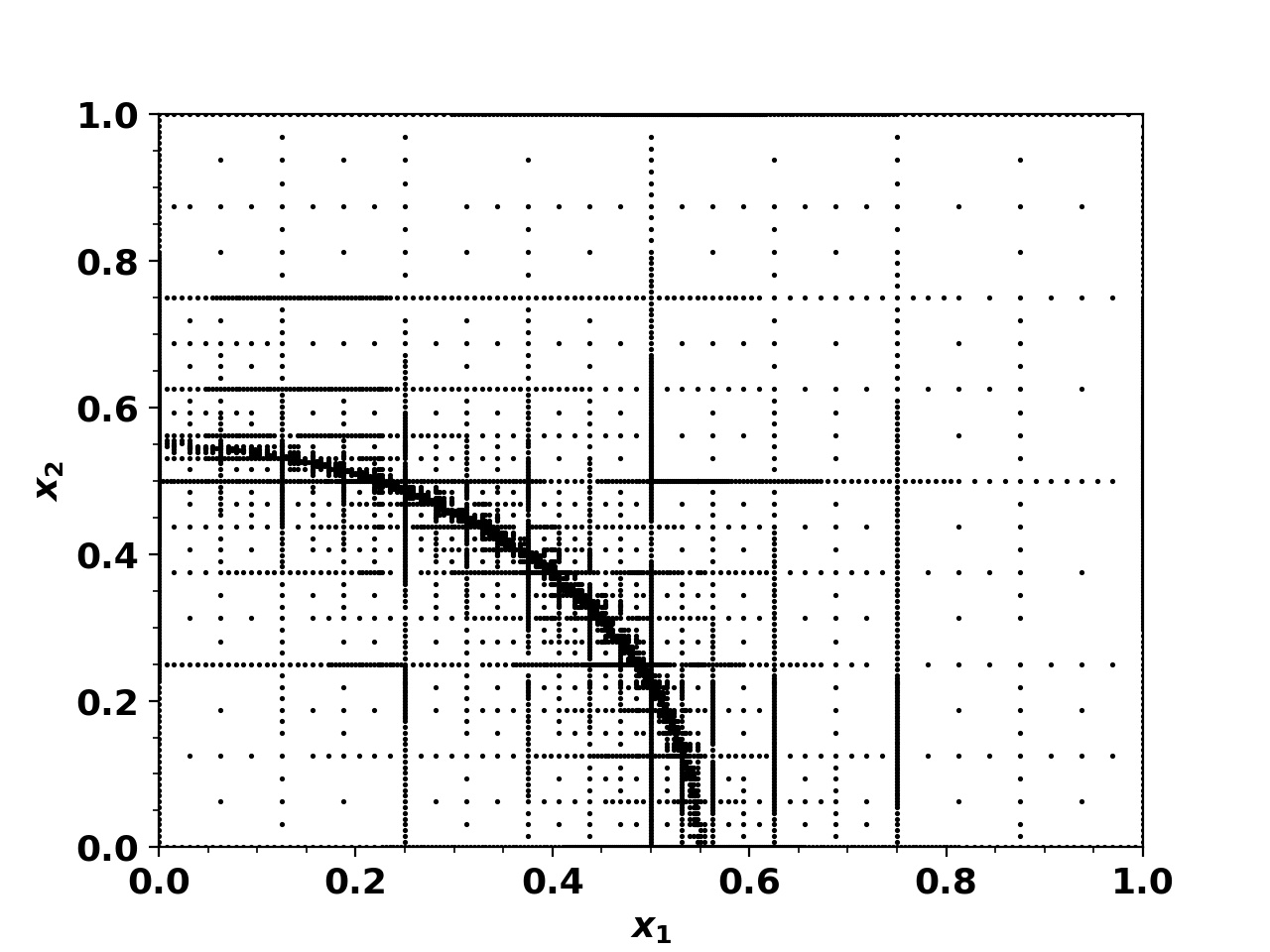}}
		\subfigure[Lagrange $P^3$ (9364 DoF)]{\includegraphics[width=.49\textwidth]{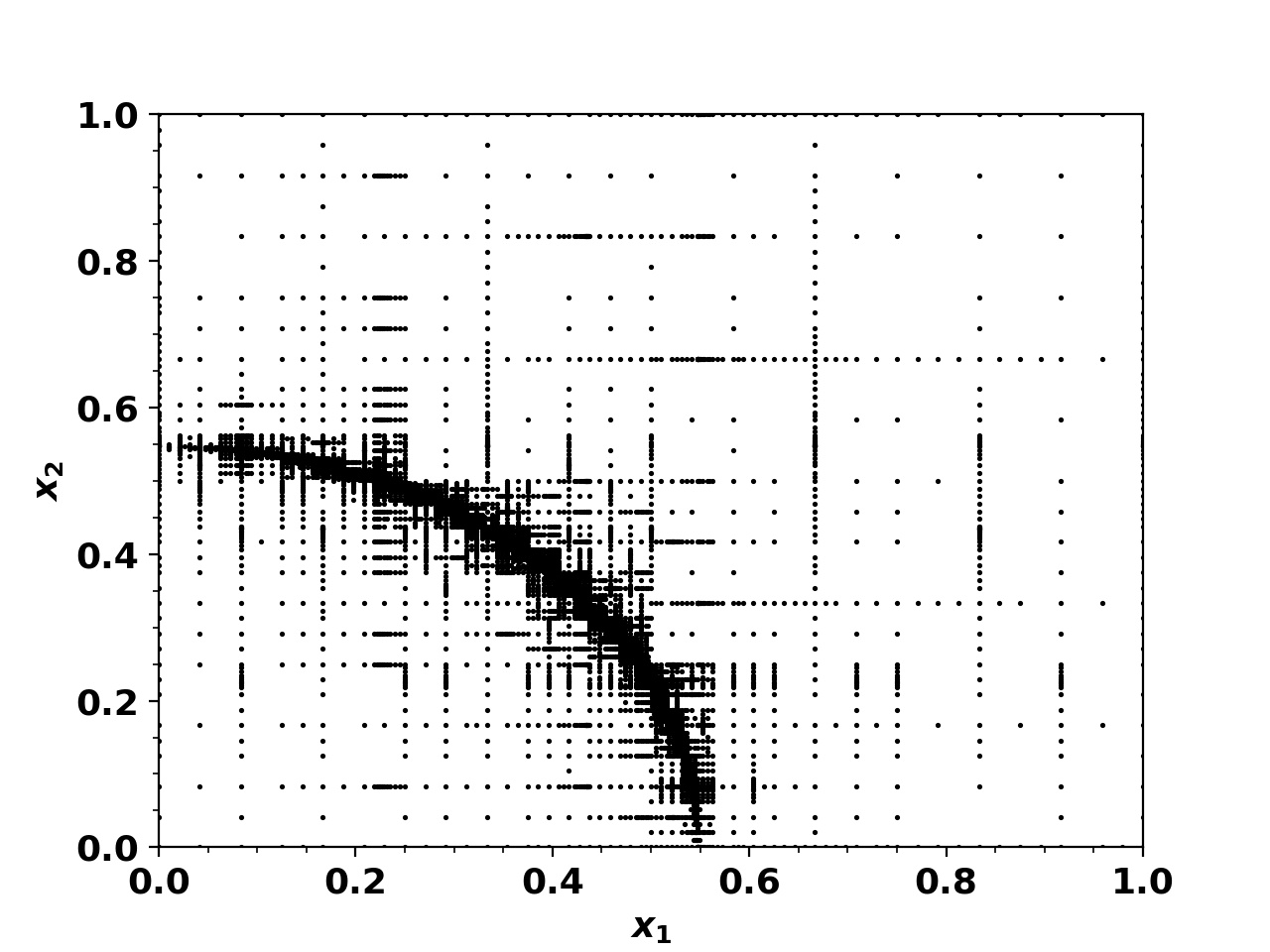}}
		\subfigure[Hermite $P^3$ (11503 DoF)]{\includegraphics[width=.49\textwidth]{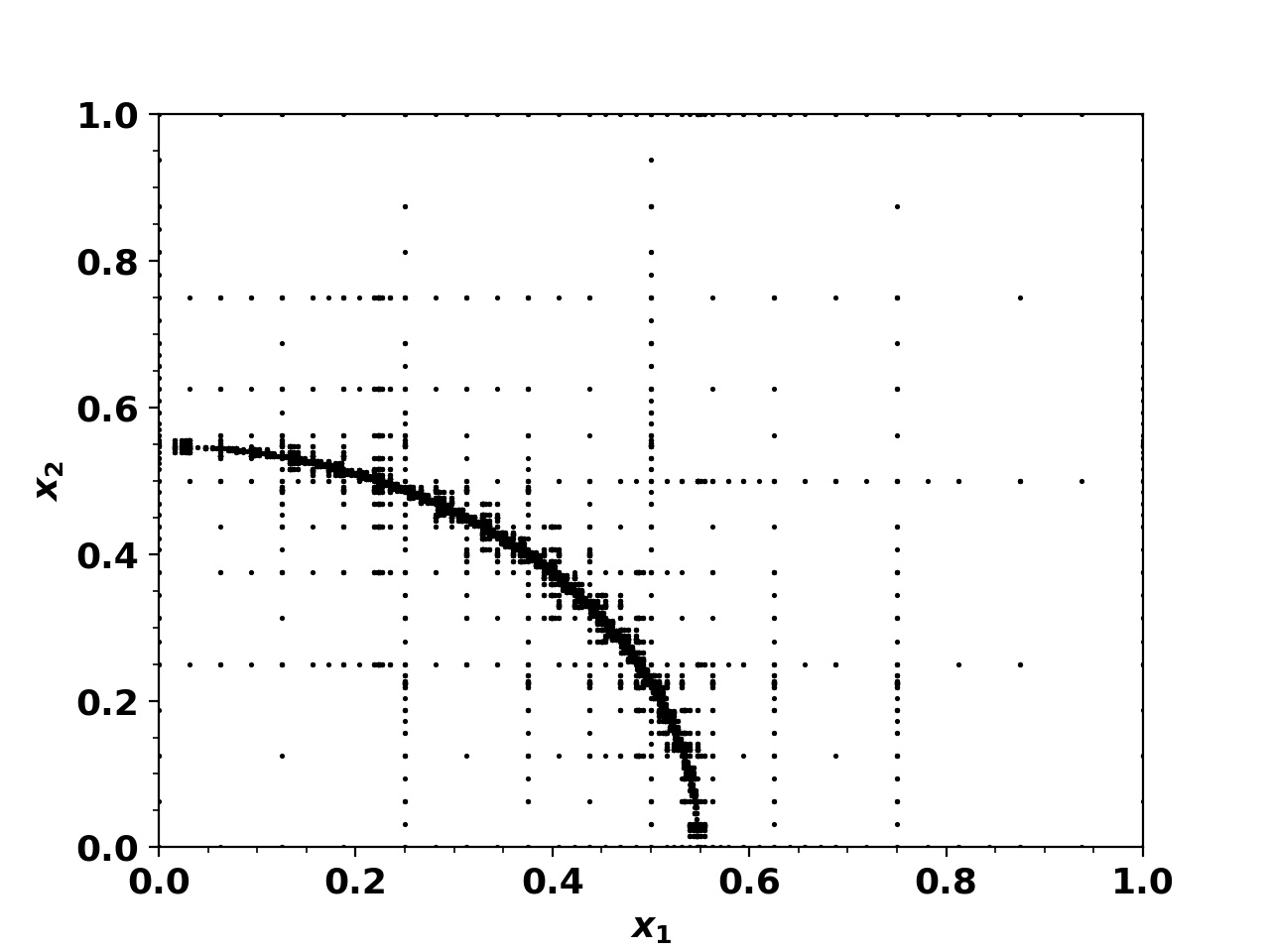}}
		\subfigure[$L^2$ error vs DoFs]{\includegraphics[width=.49\textwidth]{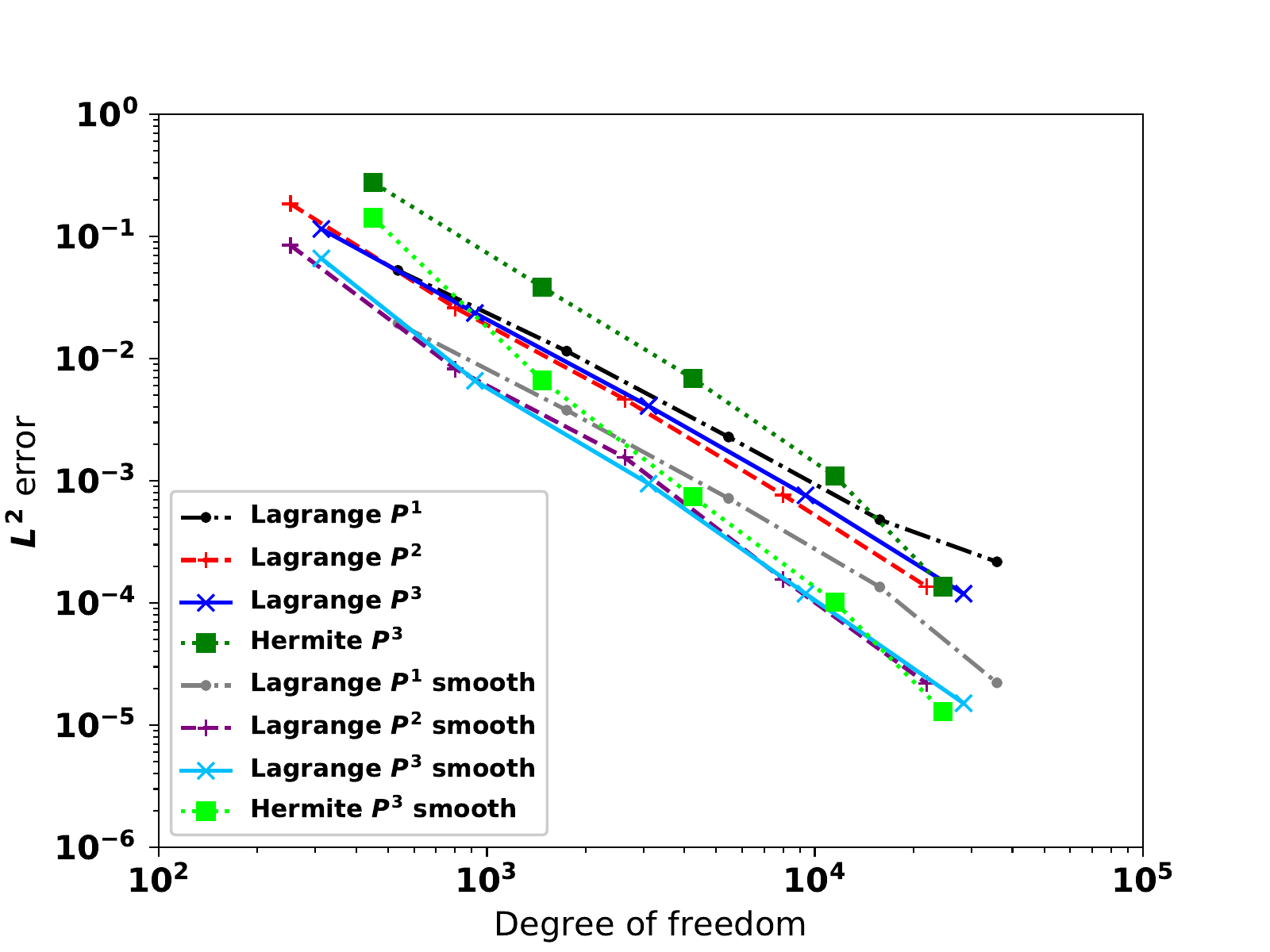}\label{subfigb}}
	\end{center}
	\caption{The adaptive grids and  error comparisons for function $f_1( {\bf x} )$.    
	}
	\label{fig:interp_fun_2D_grid}
\end{figure}

Now we consider functions $f_2({\bf x}), f_3({\bf x}), f_4({\bf x})$ in higher dimensions $d=10$.
 Here $c_i$ is taken as $ \frac{1}{2^{i+2}}$ and the maximum level is set to be very large $N=30$ for all  functions, which implies that the adaptive space is thresholded by the error parameter $\varepsilon.$
  For function $f_2({\bf x})$, the error parameter $\varepsilon$ varies from $1.0E-3$ to $1.0E-11$ for  Lagrange  $P^1$ bases and $1.0E-3$ to $1.0E-13$ for Lagrange  $P^2, P^3$ and Hermite $P^3$ bases.
For function $f_3({\bf x})$, the error parameter $\varepsilon$ varies from $1.0E-3$ to $1.0E-9$ for  Lagrange  $P^1$ bases and $1.0E-2$ to $1.0E-8$ for Lagrange  $P^2, P^3$ and Hermite $P^3$ bases.
For function $f_4({\bf x})$, the error parameter $\varepsilon$ varies from $1.0E-2$ to $1.0E-8$ for  Lagrange  $P^1, P^2, P^3$ and Hermite $P^3$ bases. 
In Figure \ref{fig:integral_error_10D_ele_pt_coa}, we show the $L^2$ errors and the quadrature errors of various interpolations vs DoFs.  For  continuous function $f_2({\bf x})$, higher order interpolations outperform lower order ones. $P^2$ and $P^3$ interpolations provide drastic improvement over $P^1$ interpolation, though the difference between $P^2$ and $P^3$ interpolations is rather small.  When the mesh is more refined, the $P^3$ interpolations (both Lagrange and Hermite) are slightly better than $P^2$ interpolation. 
For $C^0$ function $f_3({\bf x})$ or discontinuous function $f_4({\bf x})$, the performance of   all methods are qualitatively similar. 



%
%
%

\begin{figure}[htp]
	\begin{center}
		\subfigure[$f_2({\bf x})$]{\includegraphics[width=.49\textwidth]{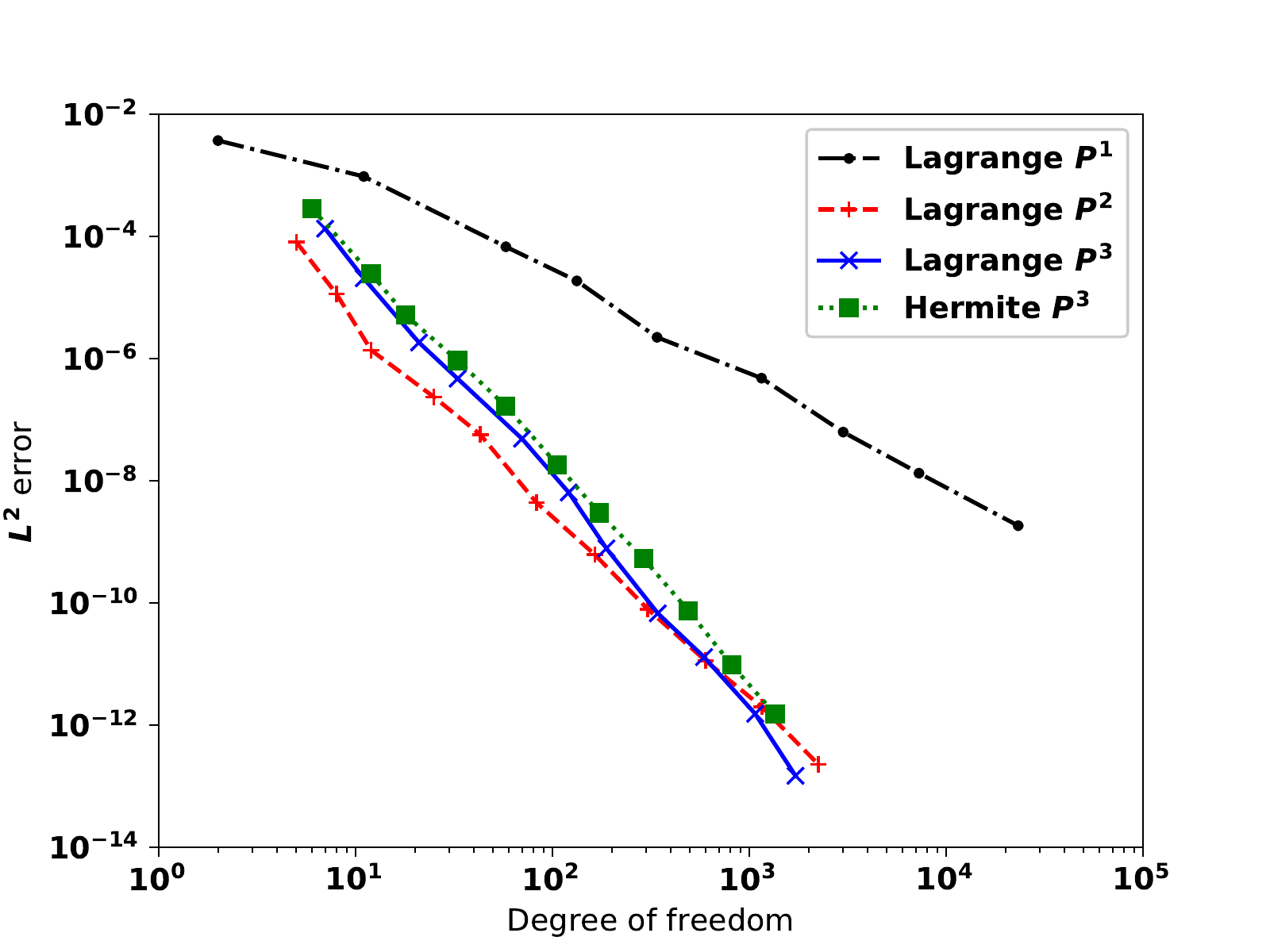}}
		\subfigure[$f_2({\bf x})$]{\includegraphics[width=.49\textwidth]{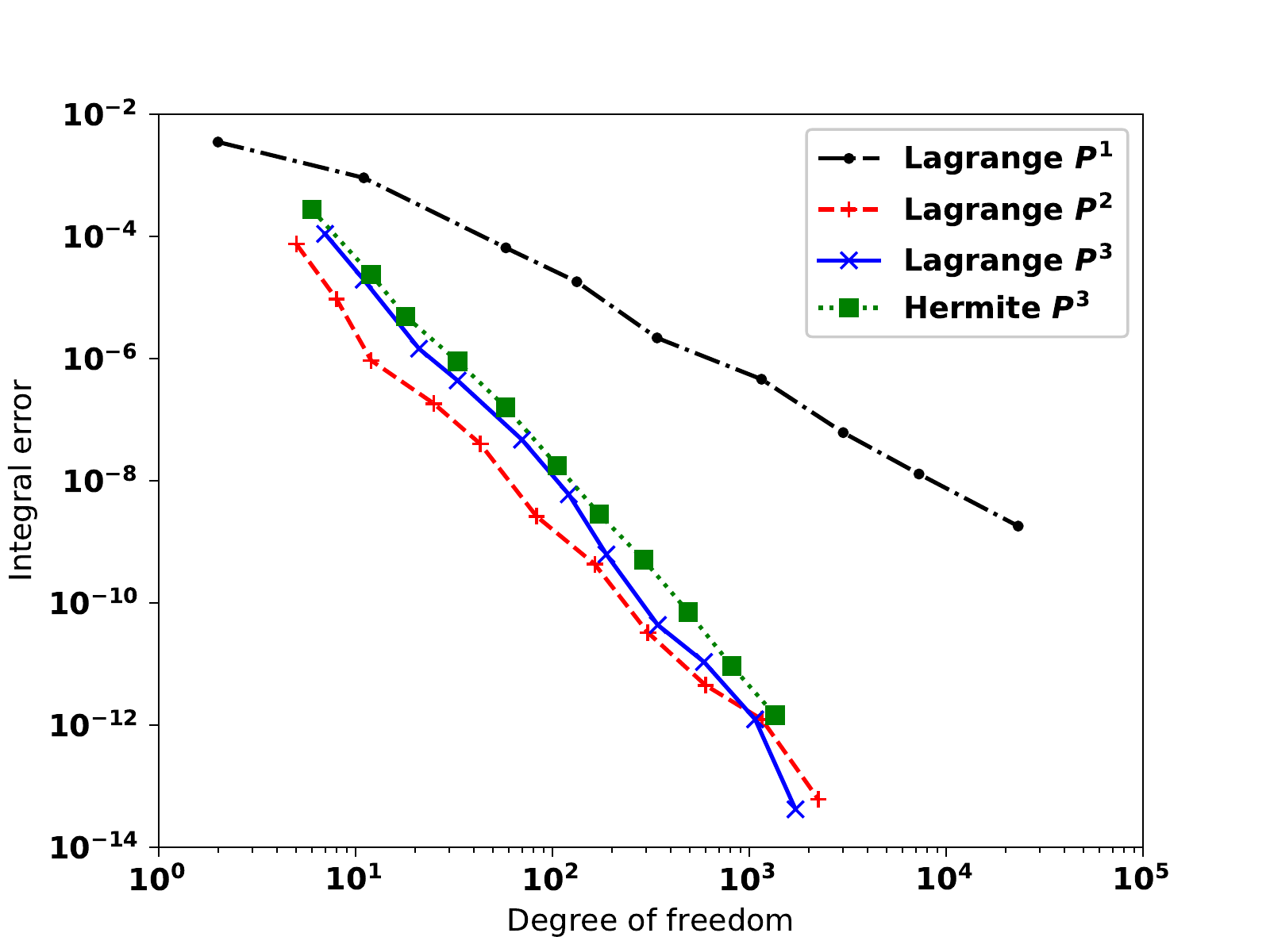}}\\
		\subfigure[$f_3({\bf x})$]{\includegraphics[width=.49\textwidth]{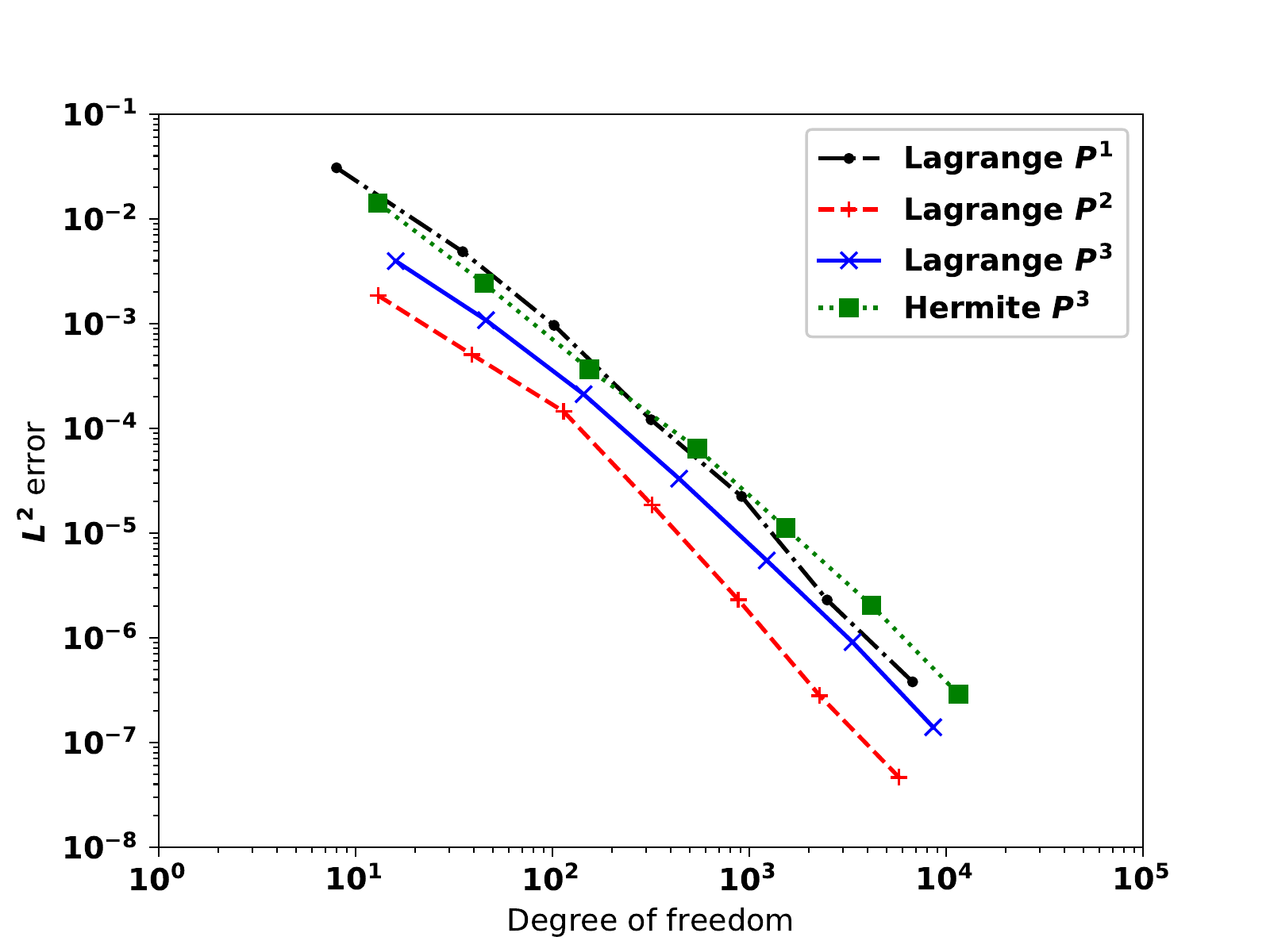}}
		\subfigure[$f_3({\bf x})$]{\includegraphics[width=.49\textwidth]{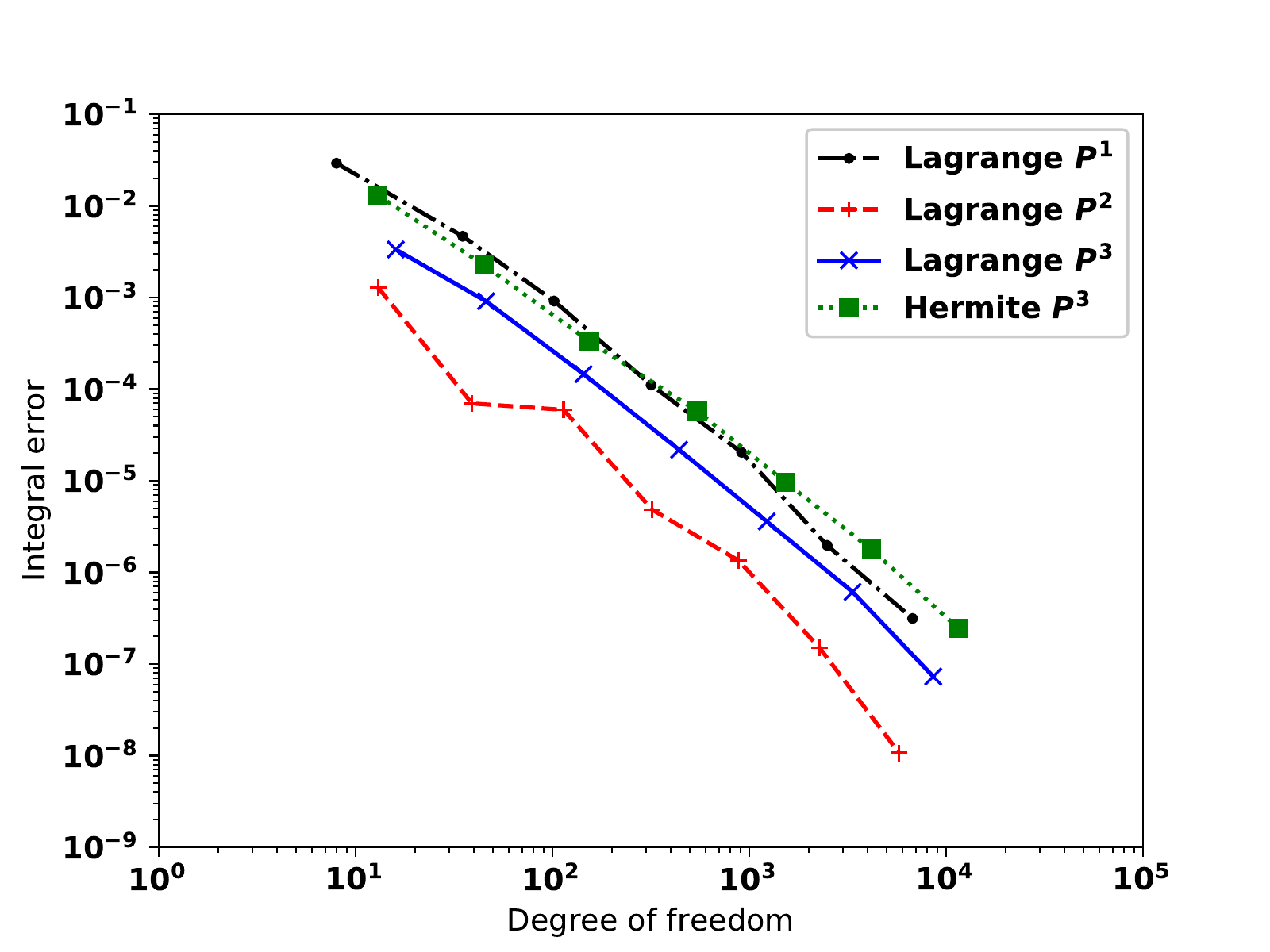}}\\
		\subfigure[$f_4({\bf x})$] {\includegraphics[width=.49\textwidth]{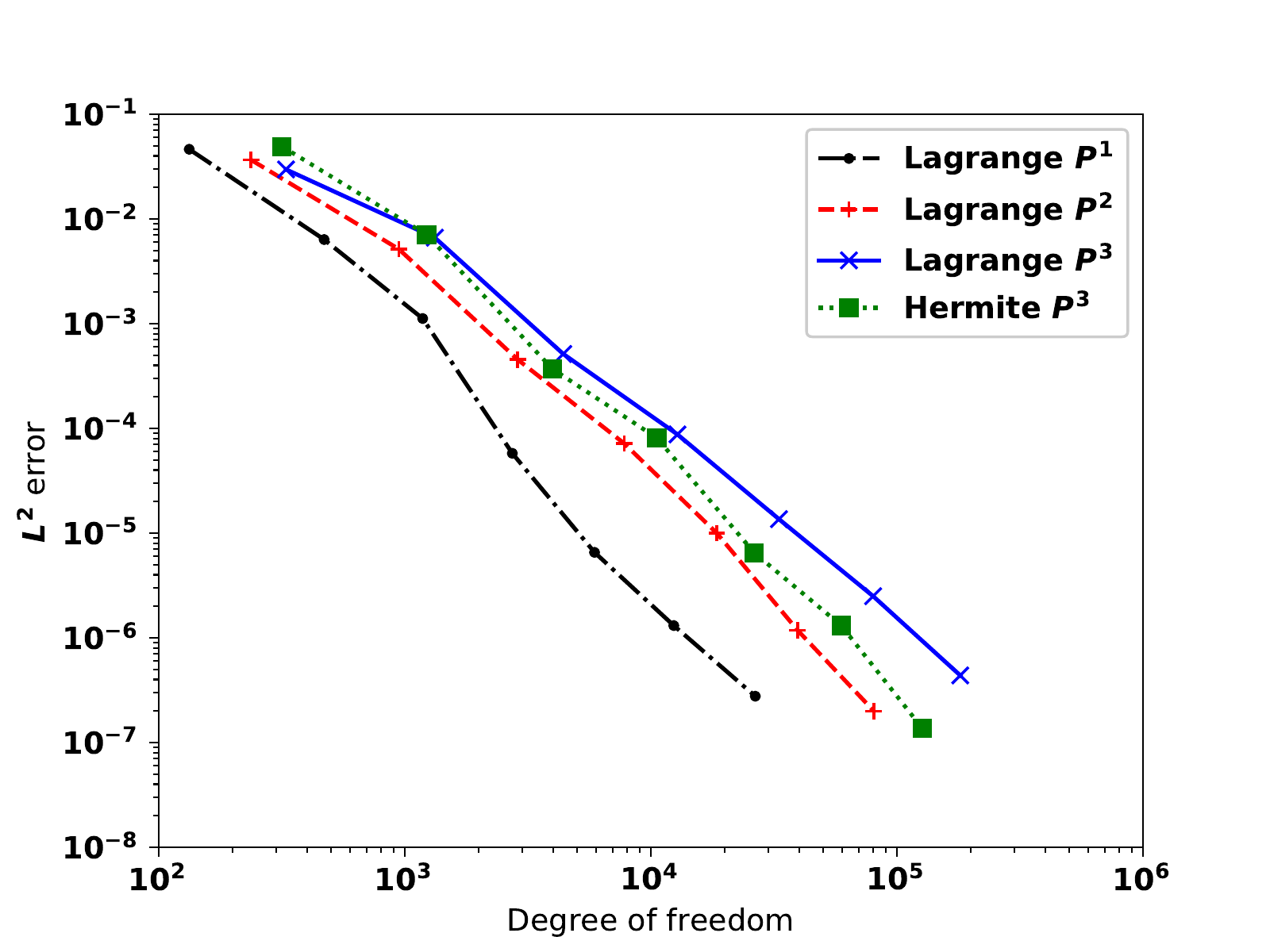}}
		\subfigure[$f_4({\bf x})$] {\includegraphics[width=.49\textwidth]{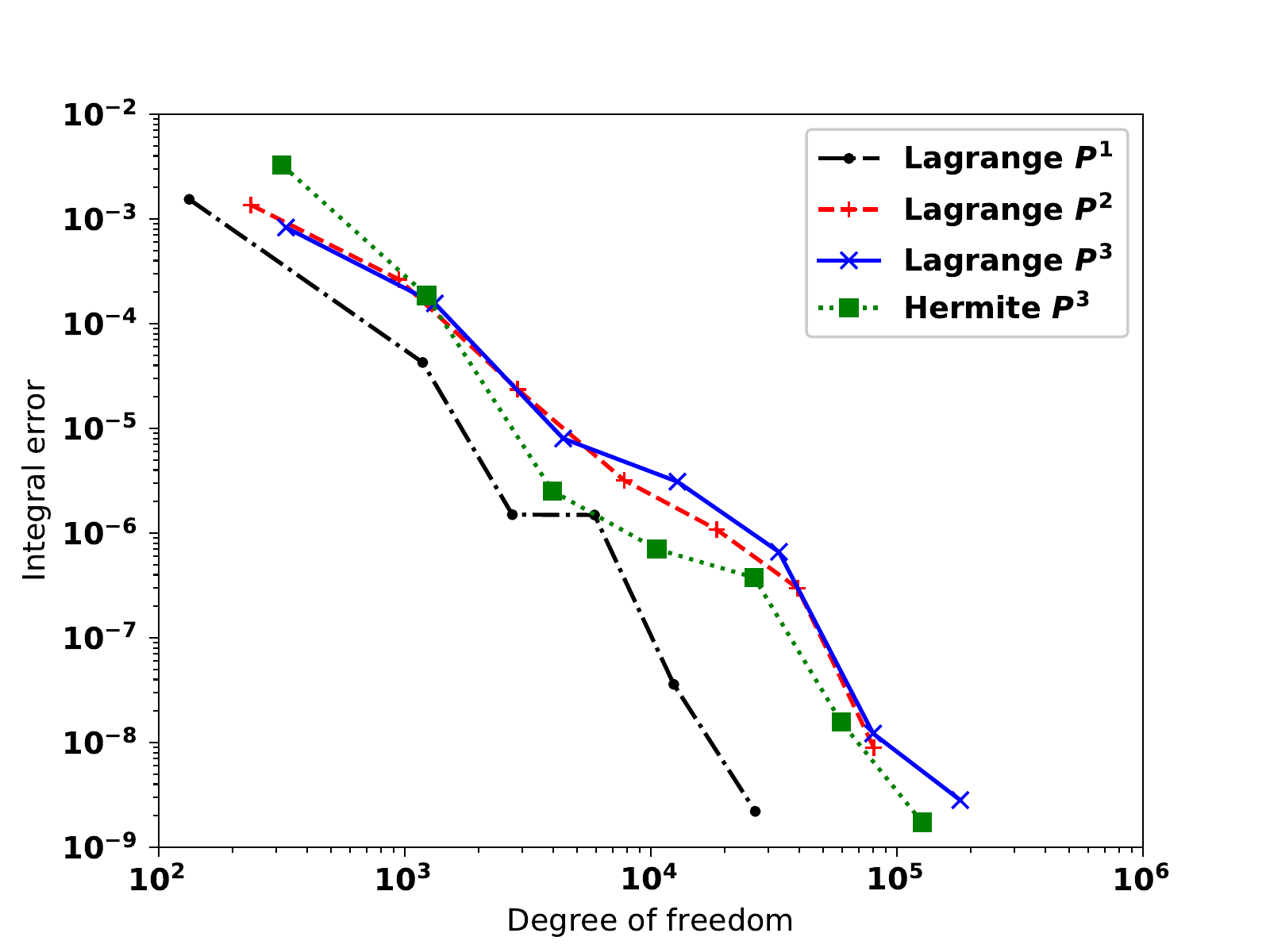}}
	\end{center}
	\caption{Errors vs DoFs. $d=10.$  Left: $L^2$ errors vs DoFs; right: Quadrature errors vs DoFs.	}
	\label{fig:integral_error_10D_ele_pt_coa}
\end{figure}

\subsection{Applications}
Now we consider several examples in UQ. Note that another application area is in the design of adaptive multiresolution DG methods, which has been considered in  \cite{huang2019adaptive}.

 \subsubsection{Stochastic elliptic equations}

We consider the following problem in one spatial dimension and $d>1$ random dimensions \cite{xiu2005high}:
\begin{equation}
\label{elliptic}
\frac{d}{dx}[a(y,x)\frac{du}{dx}(y,x)] = 0, \qquad (y,x) \in \varGamma \times (0,1),
\end{equation}
with boundary conditions 
$$ u(y,0) = 0, \qquad u(y,1)=1, \qquad y \in \varGamma, $$
where $(0,1)$ is the one-dimensional physical space and $\varGamma$ is the random space.
We assume that the random diffusivity has the form 
\begin{equation}
a(y,x) = 1 + \sigma \sum_{k=1}^{d} \frac{1}{k^2\pi^2} \cos(2 \pi kx)Y^k(\omega), 
\label{diff}
\end{equation}
where $Y^k(\omega) \in [-1,1], k=1, \cdots, d,$ are the independent uniformly distributed random variables.
The series in \eqref{diff} is convergent and strictly positive as $d \rightarrow \infty$. We have
$$E(a(y,x)) = 1,  \qquad 1-\frac{\sigma}{6} < a(y,x) < 1+\frac{\sigma}{6}.$$
A  spectral method based on Chebyshev polynomial is used for the spatial discretization.
We use 31 Chebyshev points  such that the error in random space is dominant.
Then the  sparse grid collocation method is used to approximate the random space.

In Figures \ref{fig:ellip_error_2D} and  \ref{fig:ellip_error_6D}, we present the errors in mean and variance with respect to maximum mesh levels with Lagrange $P^K,K=1,2,3$  bases for $d=2, 6.$ To compute the errors in mean and variance, 
we  use the numerical solution with maximum mesh level 8 ($d=2$) and 5 ($d=6$) as the reference solution.
We observe that the errors with $P^2$ and $P^3$ bases are much better than those with $P^1$ bases.
It is noted that higher order interpolations can achieve round-off errors even when the mesh is very coarse. 

\begin{figure}[htp]
	\begin{center}
		\subfigure[mean]{\includegraphics[width=.49\textwidth]{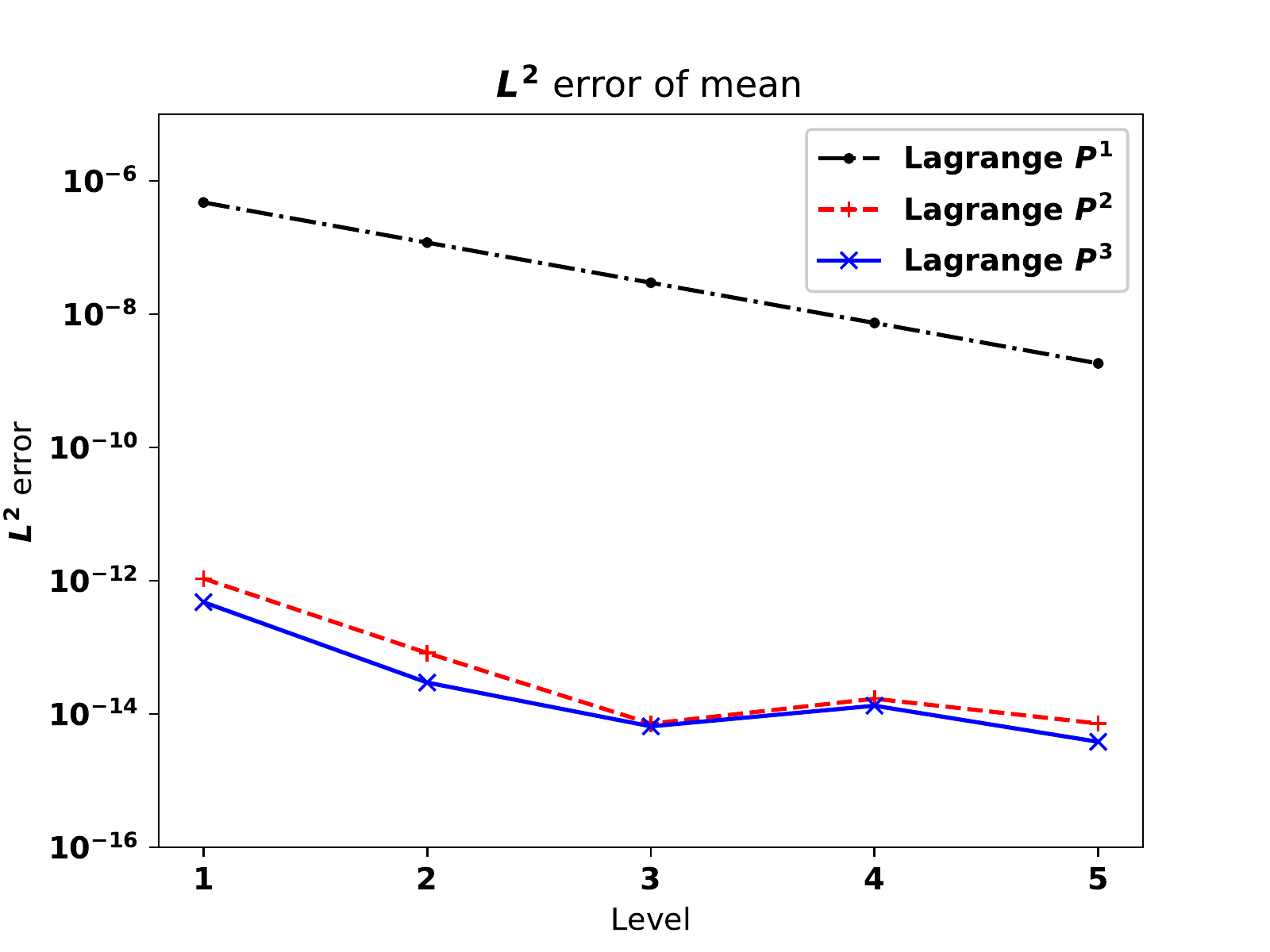}}
		\subfigure[variance]{\includegraphics[width=.49\textwidth]{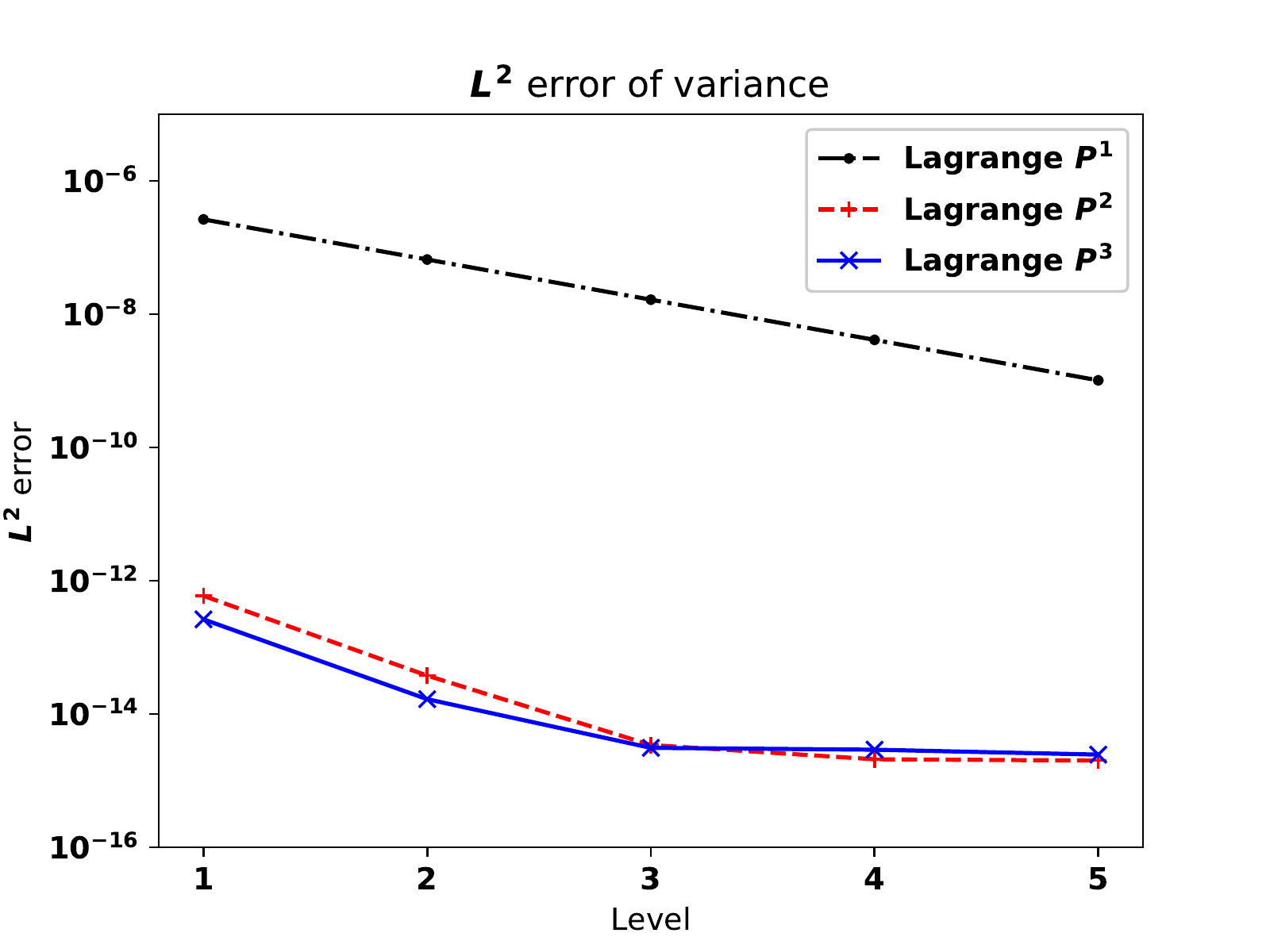}}
	\end{center}
	\caption{Errors in mean and variance with respect to maximum mesh levels in low dimensional random inputs ($d=2$).  (a) mean; (b) variance.  
	}
	\label{fig:ellip_error_2D}
\end{figure}

In Figure  \ref{fig:ellip_m_v_6D}, we show the mean and variance solutions of \eqref{elliptic} with maximum mesh level $N=3$ for $d=6$ random inputs. We observe that the mean and variance solutions are almost the same for different order bases. 

\begin{figure}[htp]
	\begin{center}
		\subfigure[mean]{\includegraphics[width=.49\textwidth]{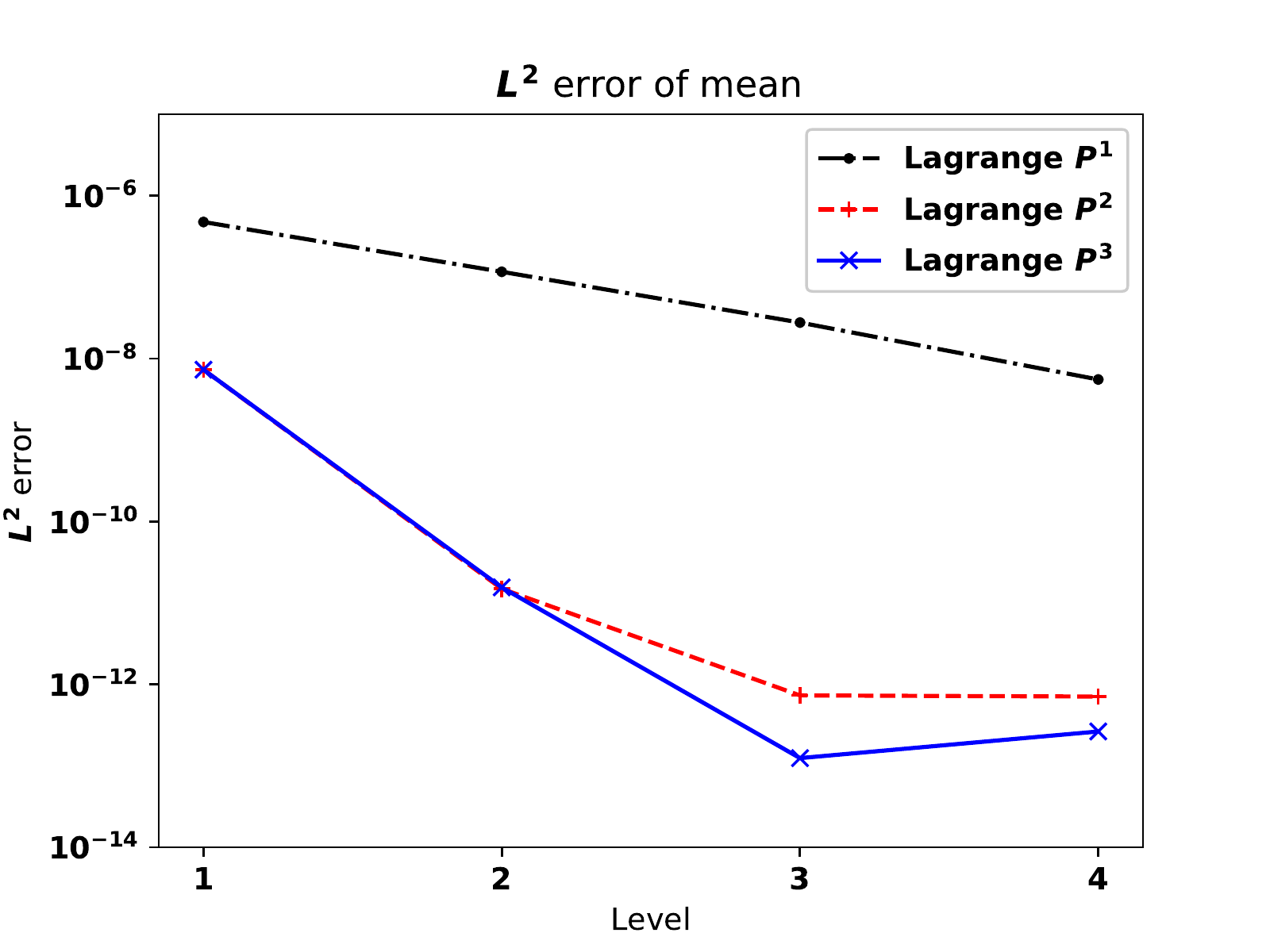}}
		\subfigure[variance]{\includegraphics[width=.49\textwidth]{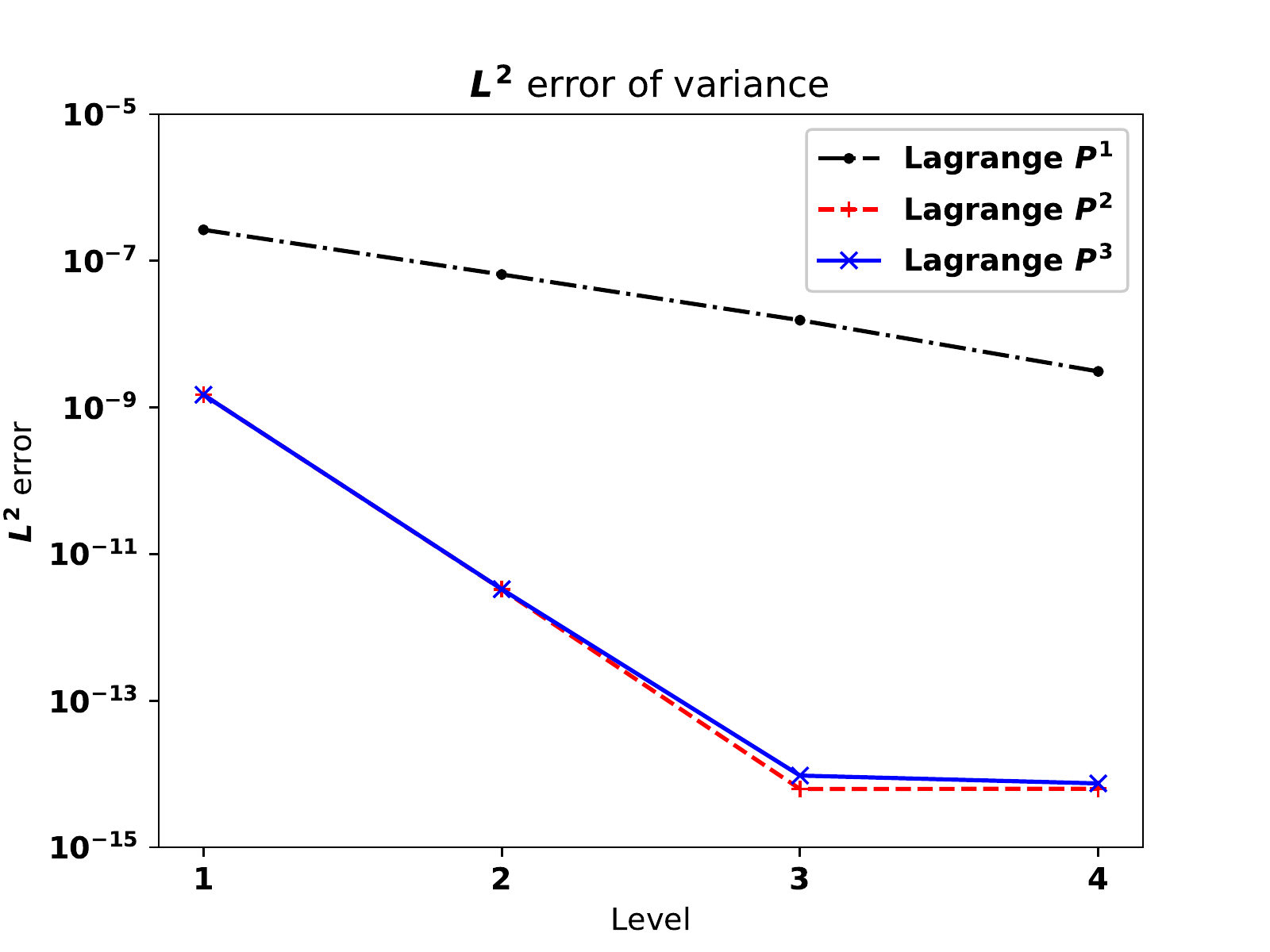}}
	\end{center}
	\caption{Errors in mean and variance with respect to maximum mesh levels in moderate dimensional random inputs ($d=6$).  (a) mean; (b) variance.  
	}
	\label{fig:ellip_error_6D}
\end{figure}

\begin{figure}[htp]
	\begin{center}
		\subfigure[Lagrange $P^1$]{\includegraphics[width=.49\textwidth]{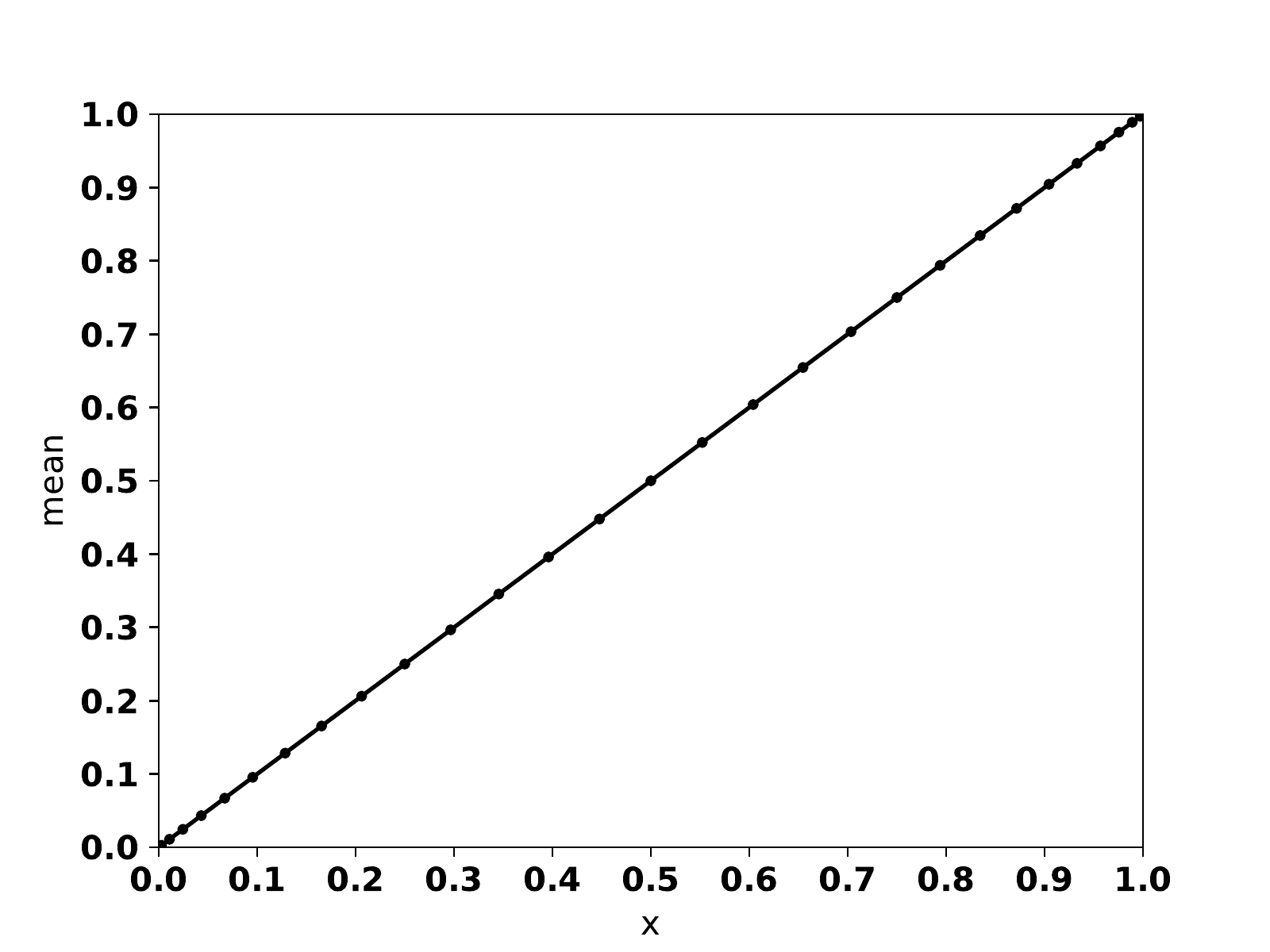}}
		\subfigure[Lagrange $P^1$]{\includegraphics[width=.49\textwidth]{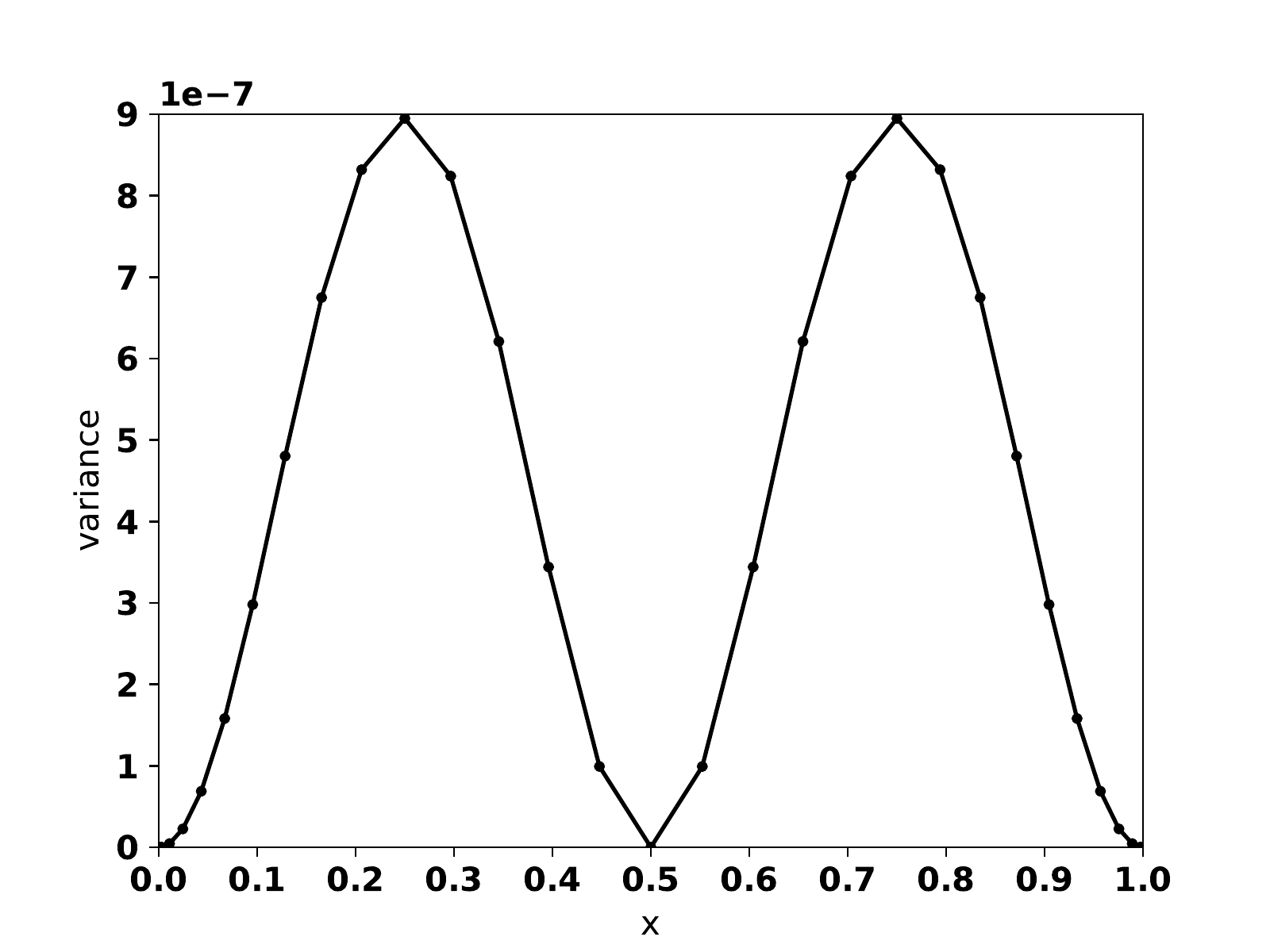}}
		\subfigure[Lagrange $P^2$]{\includegraphics[width=.49\textwidth]{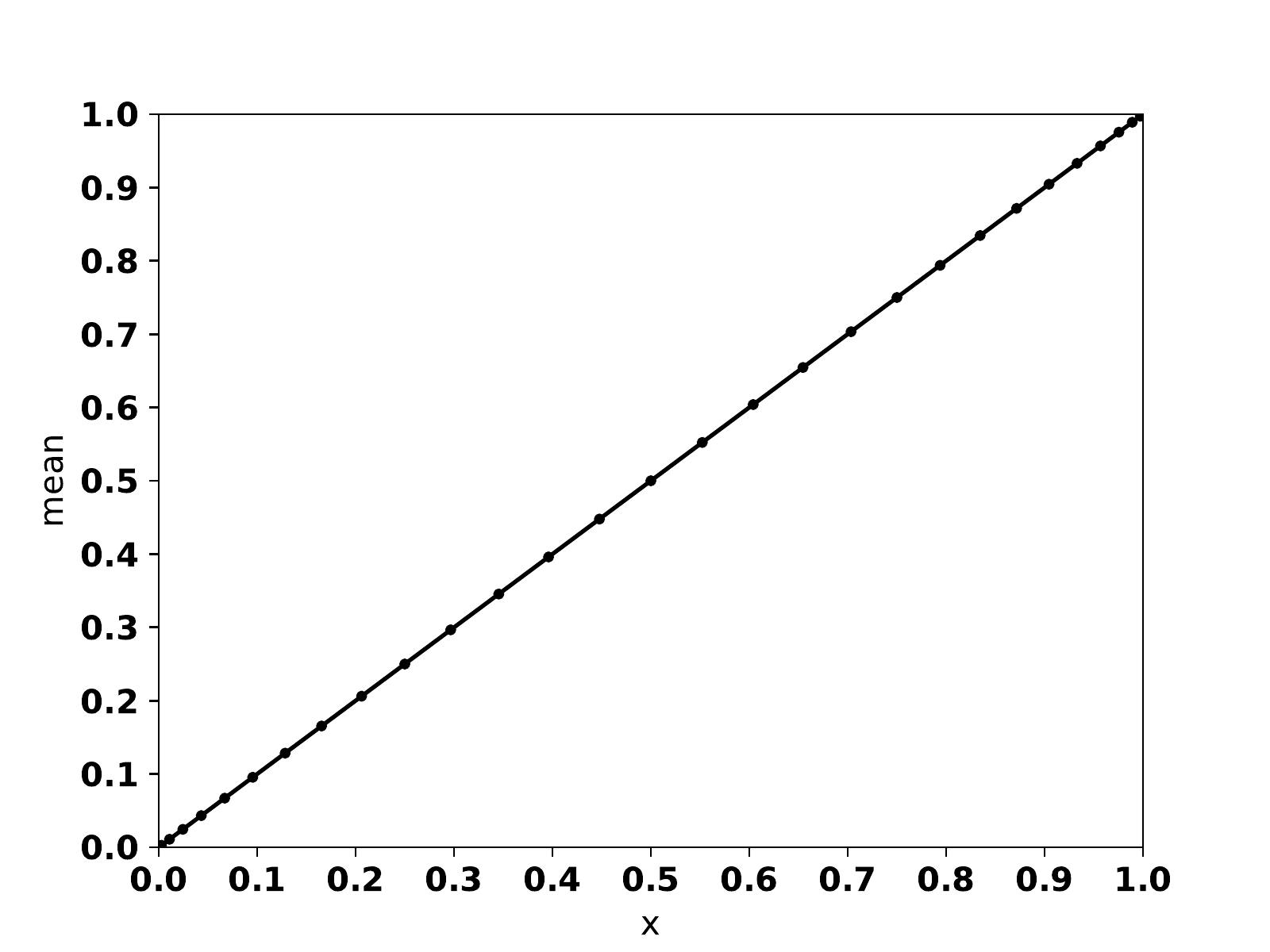}}
		\subfigure[Lagrange $P^2$]{\includegraphics[width=.49\textwidth]{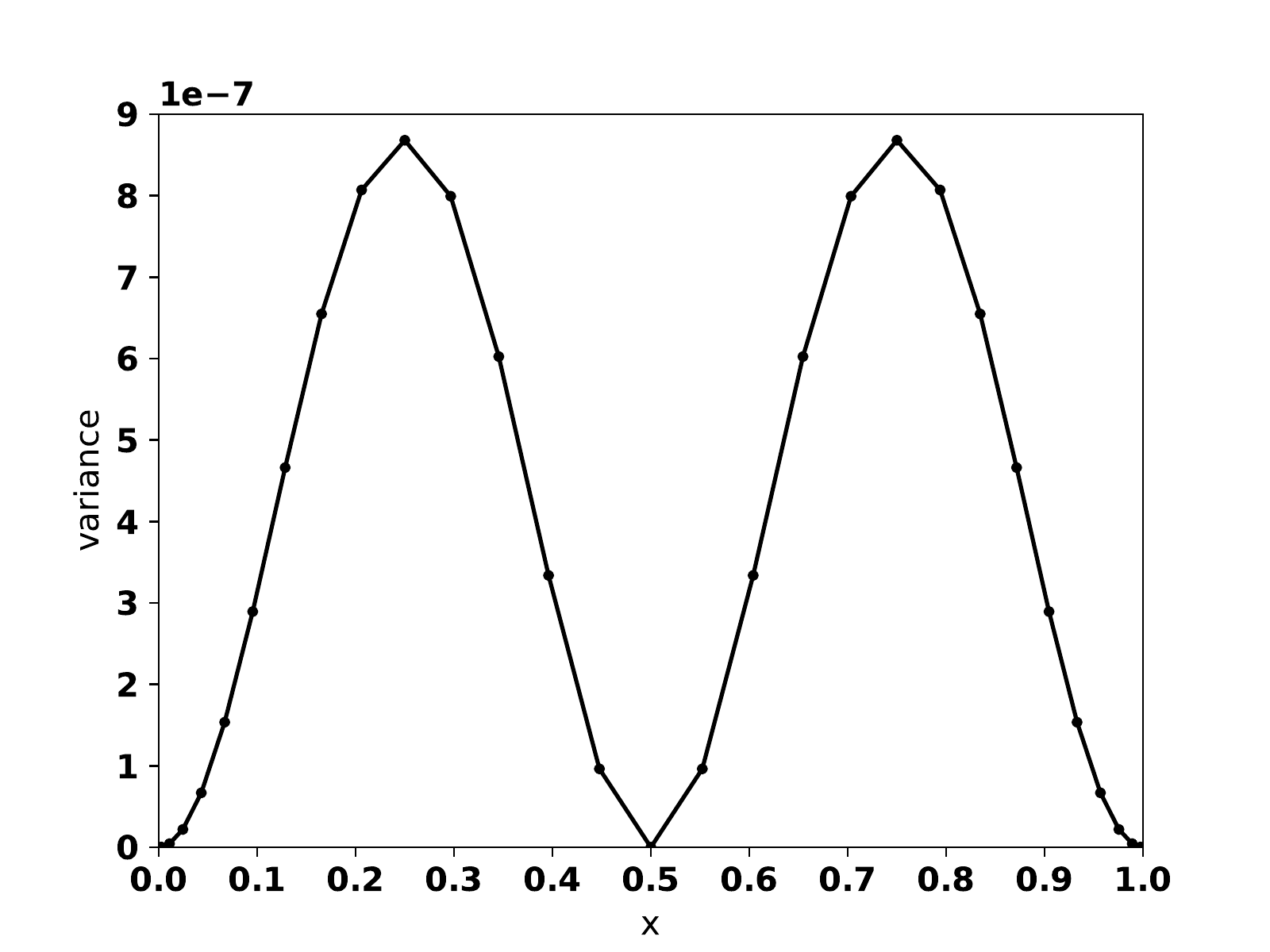}}
		\subfigure[Lagrange $P^3$]{\includegraphics[width=.49\textwidth]{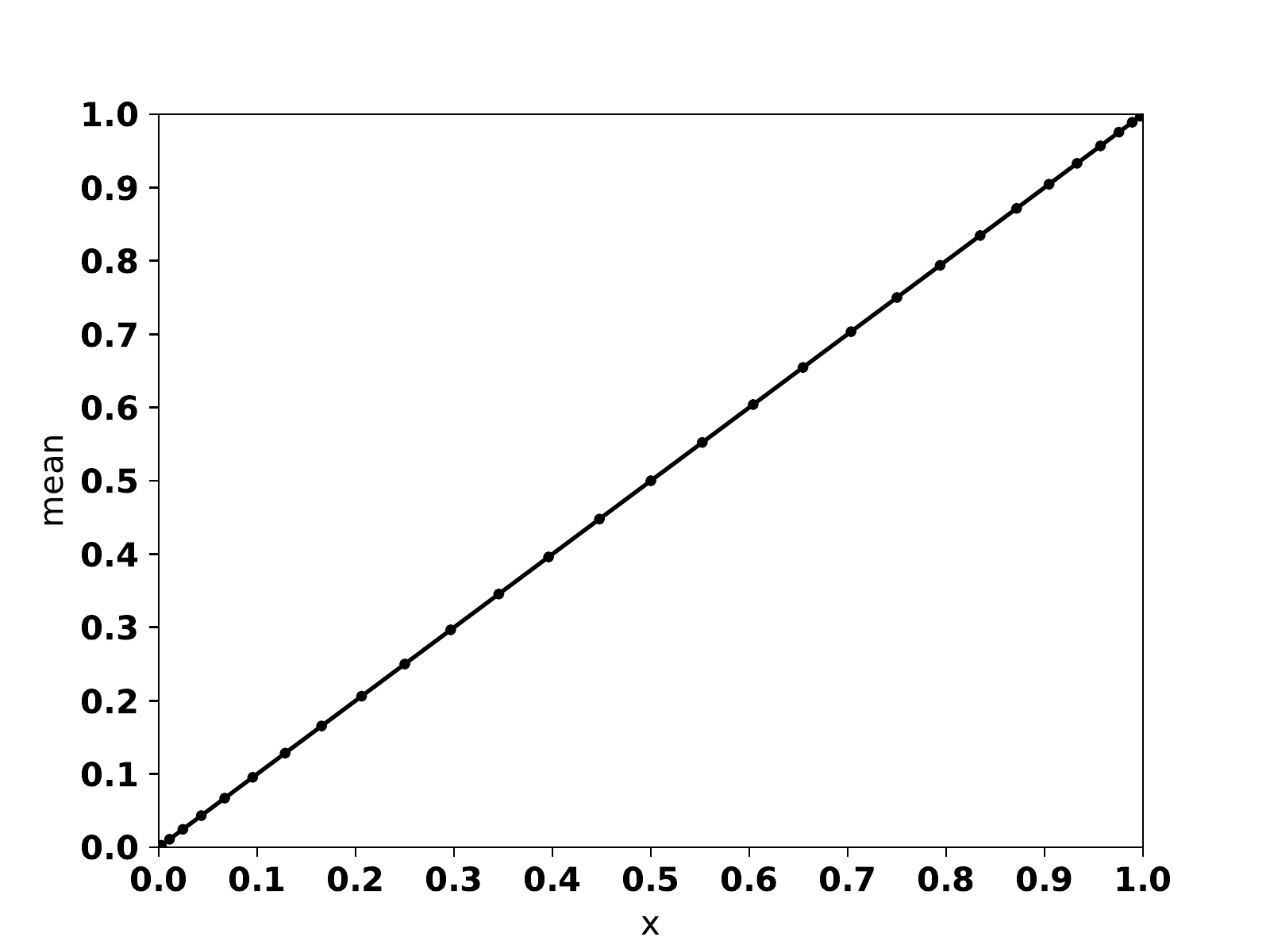}}
		\subfigure[Lagrange $P^3$]{\includegraphics[width=.49\textwidth]{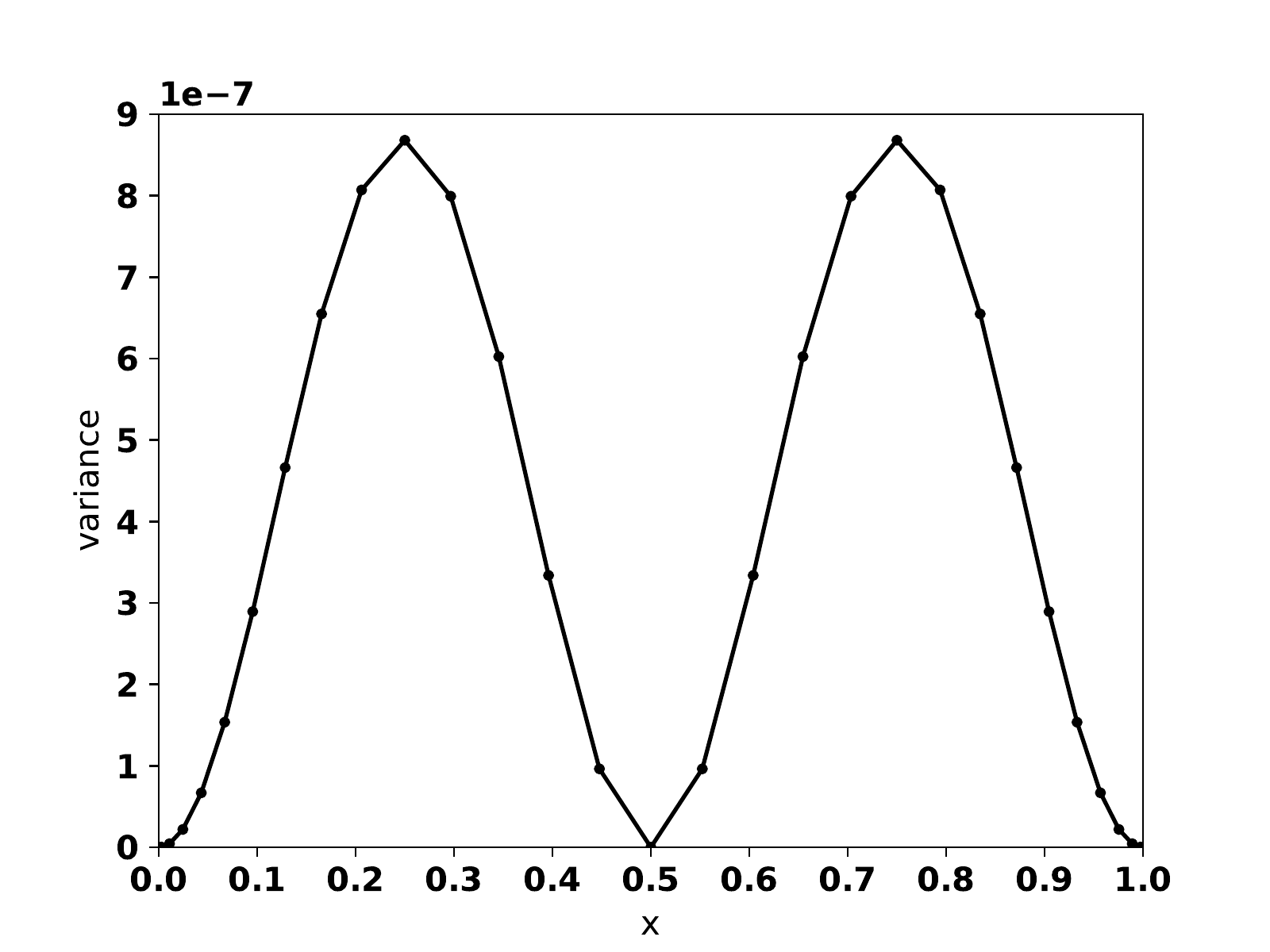}}
	\end{center}
	\caption{Stochastic solution of \eqref{elliptic} with $d=6$ random inputs. Left: mean of solution; right: variance of solution. (a)-(b) Lagrange $P^1$; (c)-(d) Lagrange $P^2$; (e)-(f) Lagrange $P^3$.  
	}
	\label{fig:ellip_m_v_6D}
\end{figure}

\subsubsection{Kraichnan-Orszag (K-O) problem}

We consider the transformed Kraichnan-Orszag three-mode problem
\begin{align}
\label{K-O}
\frac{dy_1}{dt} & = y_1  y_3,  \notag \\
\frac{dy_2}{dt} & = -y_2  y_3, \\
\frac{dy_3}{dt} & = -y_1^2  + y_2^2, \notag
\end{align}
with initial condition
\begin{align*}
y_1(0) = Y_1(0;\omega), \quad y_2(0) = Y_2(0;\omega), \quad y_3(0) = Y_3(0;\omega).
\end{align*}
 This problem presents a bifurcation on the parameter $y_1(0)$ and $y_2(0)$. The deterministic solutions of the problem are periodic and the period goes to infinity if the initial conditions are located at the planes   $y_1=0$ and $y_2=0,$ which means that discontinuity occurs when the initial conditions cross these two planes. The random initial conditions are chosen as the uniform distribution $Y\sim U(-1,1)$. In this setting, the initial conditions cross the discontinuity plane and therefore adaptive sparse grid method is necessary for this problem. 
 This problem has been studied by many researchers \cite{ma2009adaptive, wan2005adaptive, foo2008multi}. In our experiments, a third order Runge-Kutta method with time step $\Delta t = 0.01$ is used for the time integration.  Since the location of discontinuity is fixed in this example, we just apply the refinement step and skip the coarsening step.


First, we consider the simplified case of one-dimensional random input as follows
\begin{align*}
y_1(0) = 1.0, \quad y_2(0) = 0.1Y(0;\omega), \quad y_3(0) = 0.
\end{align*}
In the Hermite bases, we also need the evolving equations for the derivatives
\begin{align*}
\frac{d}{dt} (y_1)_Y & = (y_1)_Y  y_3 + y_1 (y_3)_Y,  \\
\frac{d}{dt} (y_2)_Y& = -(y_2)_Y  y_3 - y_2 (y_3)_Y, \\
\frac{d}{dt} (y_3)_Y & = -2 y_1 (y_1)_Y  + 2 y_2 (y_2)_Y.
\end{align*}
In the left part of Figure \ref{fig:KO_1D_2D}, we present the time evolution of the variance of the solution $(y_1, y_2, y_3)$ during the time interval [0,30] (short time behavior) with Lagrange and Hermite bases. $\varepsilon$ is taken as $ 10^{-4}$ and the maximum mesh level is set to be 10. We observe that all the solutions by our adaptive methods are convergent and they are consistent with those in \cite{ma2009adaptive}. In this 1D case, the results are almost the same. 
The realization of solutions $(y_1, y_2, y_3)$ at $t=1$ and $t=60$ are shown in  Figure \ref{fig:KO_1D_realization} with Lagrange $P^2$ bases.  At earlier time, no discontinuity has been developed. When time increases, the discontinuity becomes stronger and the oscillations are generated in the solutions.

\begin{figure}[htp]
	\begin{center}
		\subfigure[$y_1$]{\includegraphics[width=.49\textwidth]{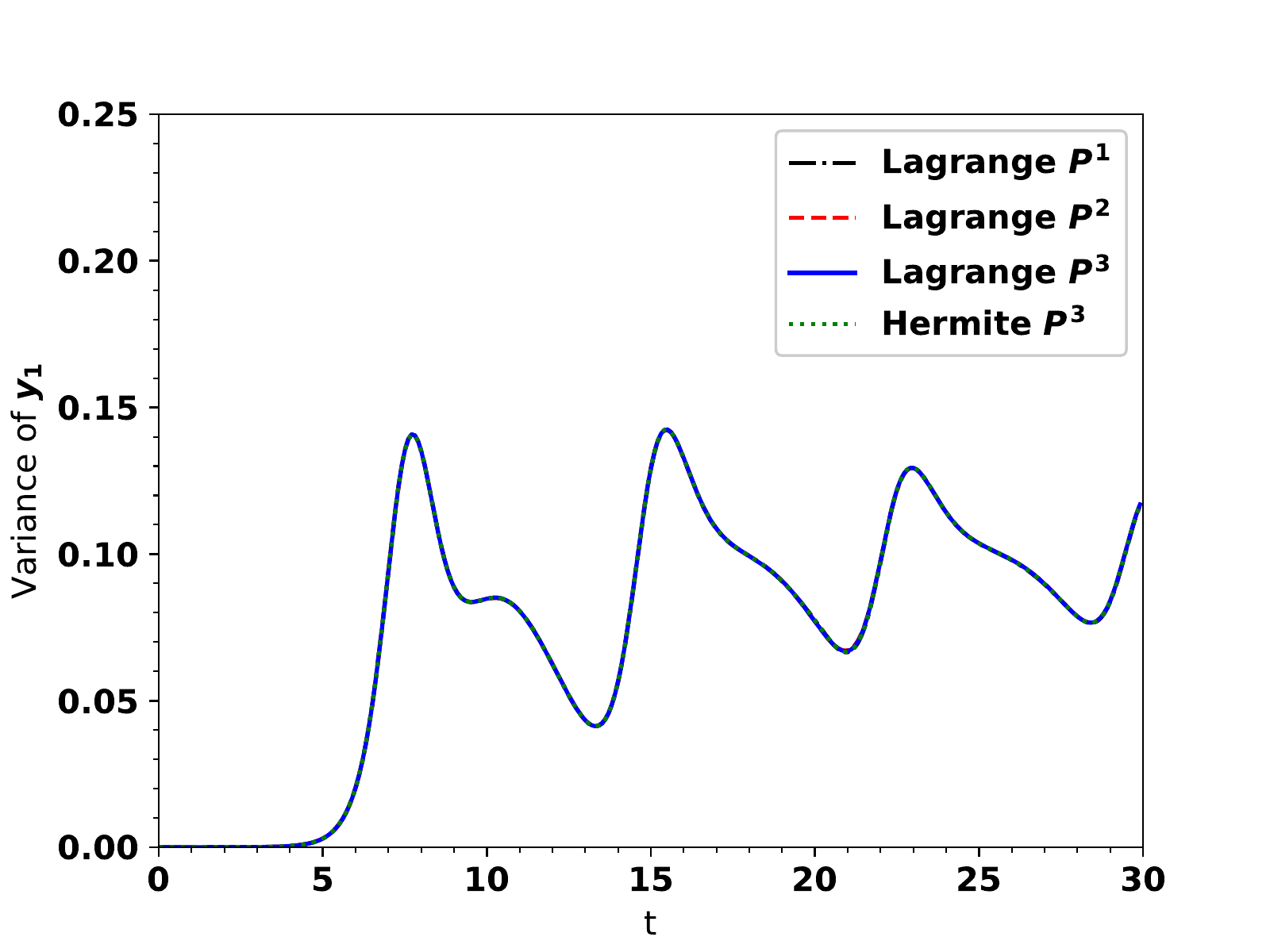}}
		\subfigure[$y_1$]{\includegraphics[width=.49\textwidth]{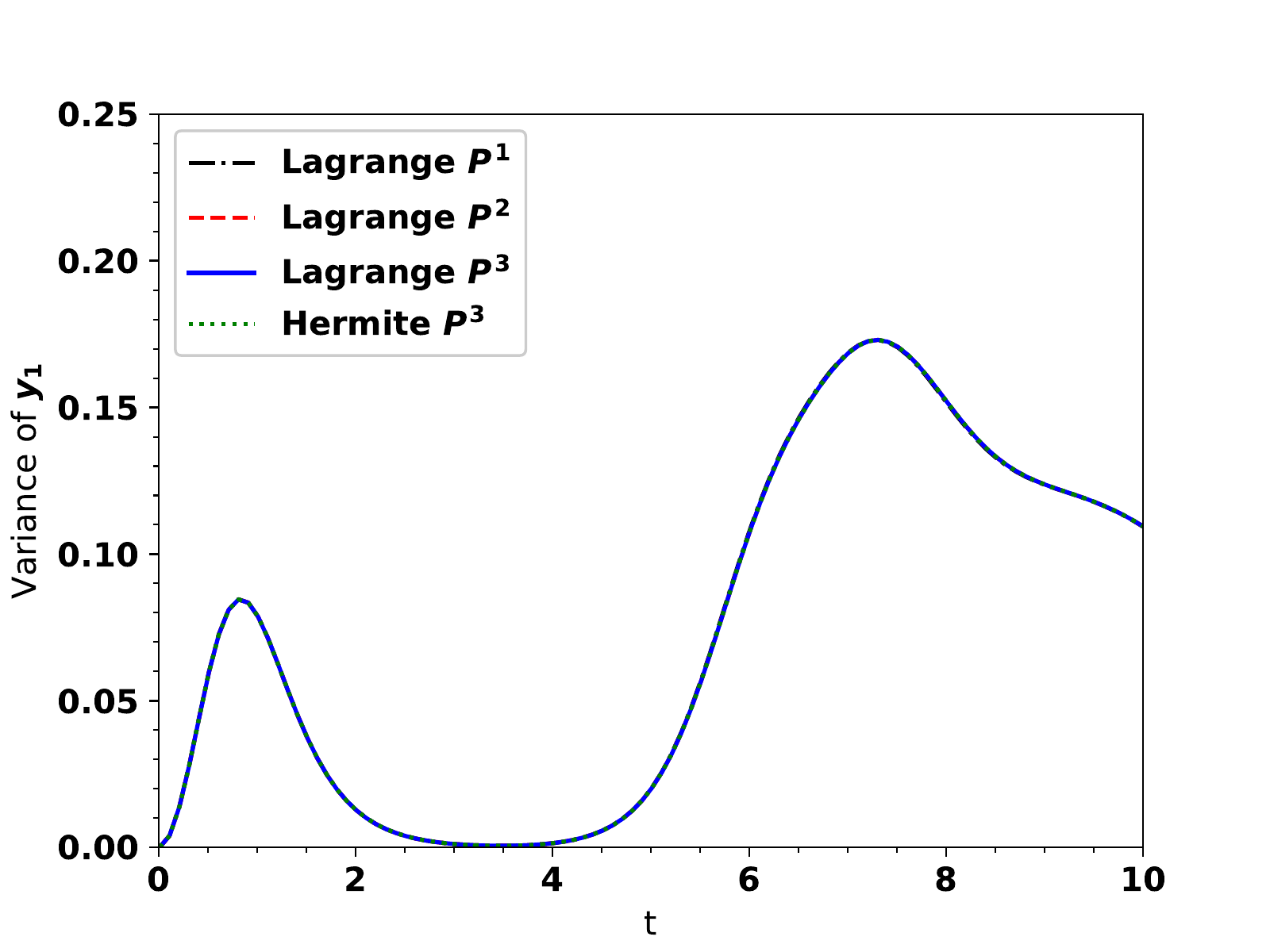}}\\
		\subfigure[$y_2$]{\includegraphics[width=.49\textwidth]{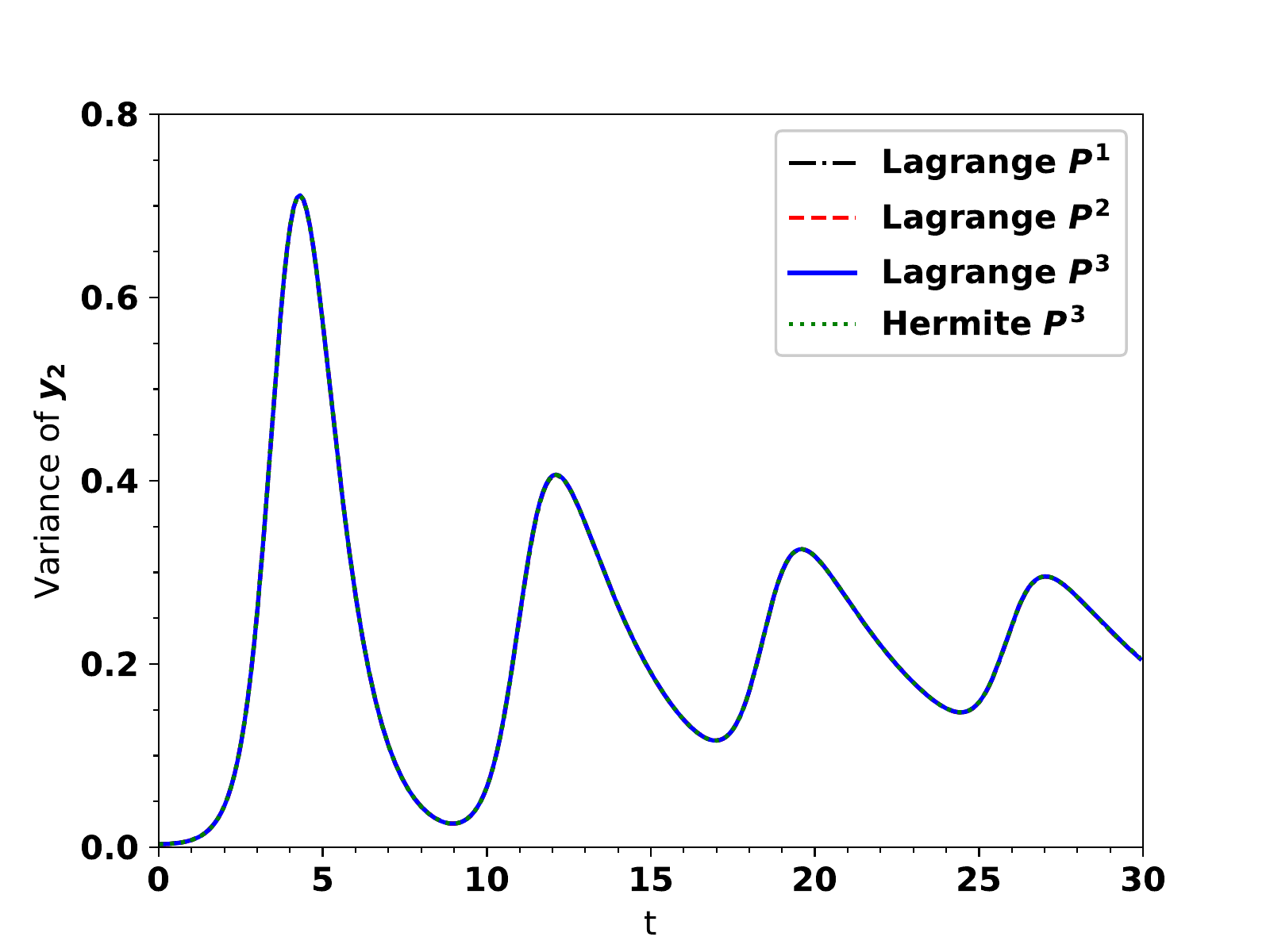}}
		\subfigure[$y_2$]{\includegraphics[width=.49\textwidth]{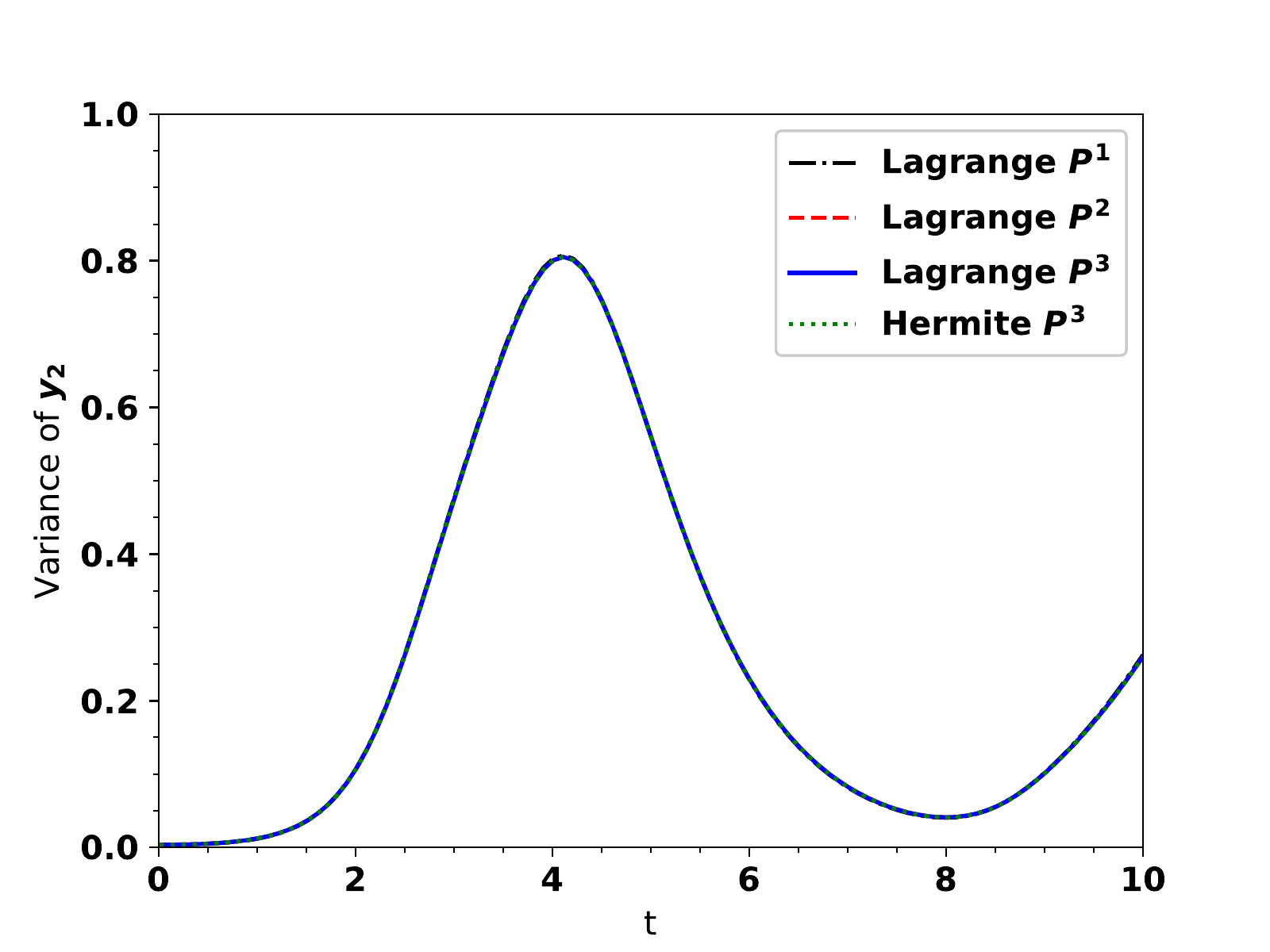}}\\
		\subfigure[$y_3$]{\includegraphics[width=.49\textwidth]{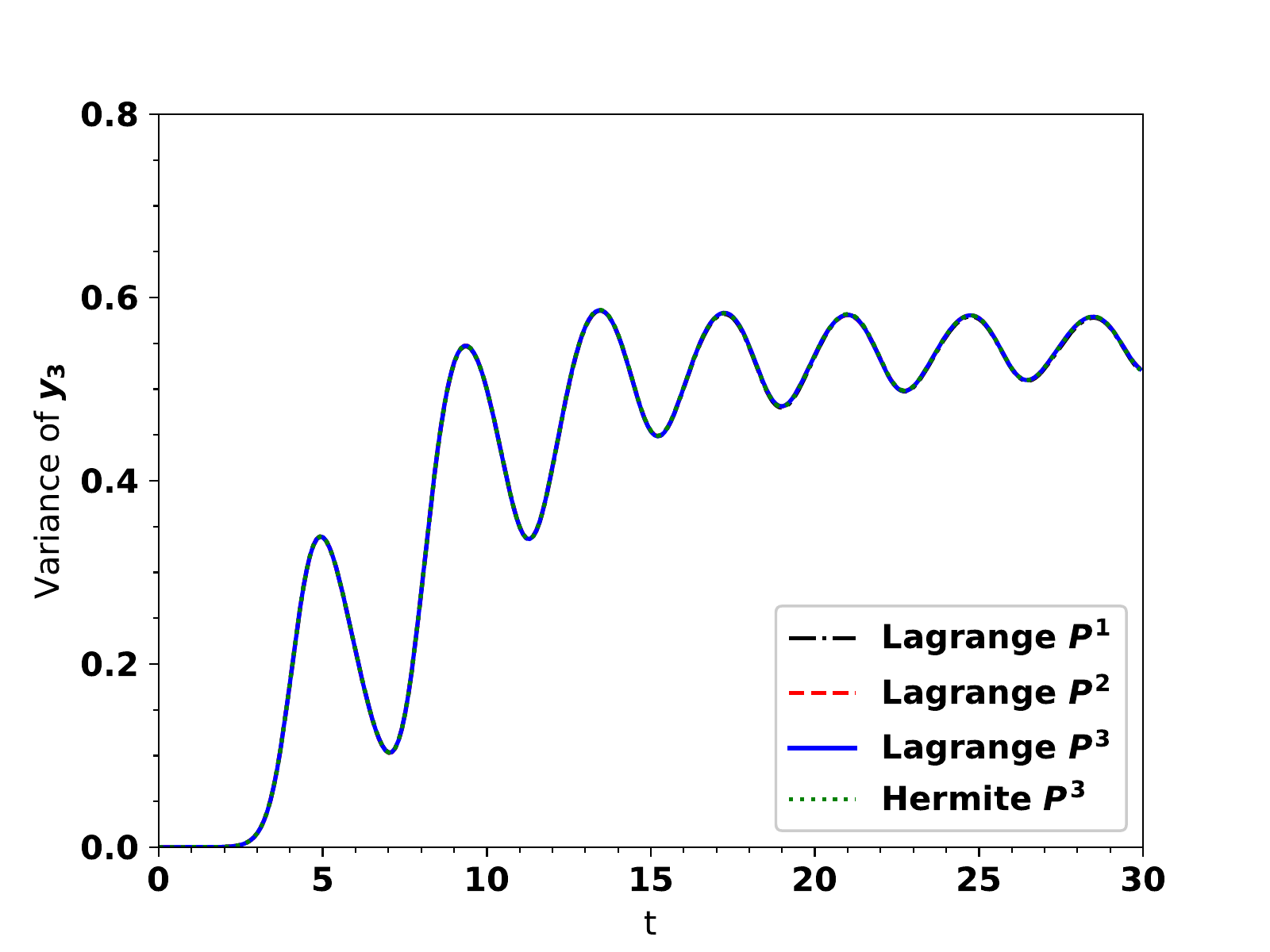}}
		\subfigure[$y_3$]{\includegraphics[width=.49\textwidth]{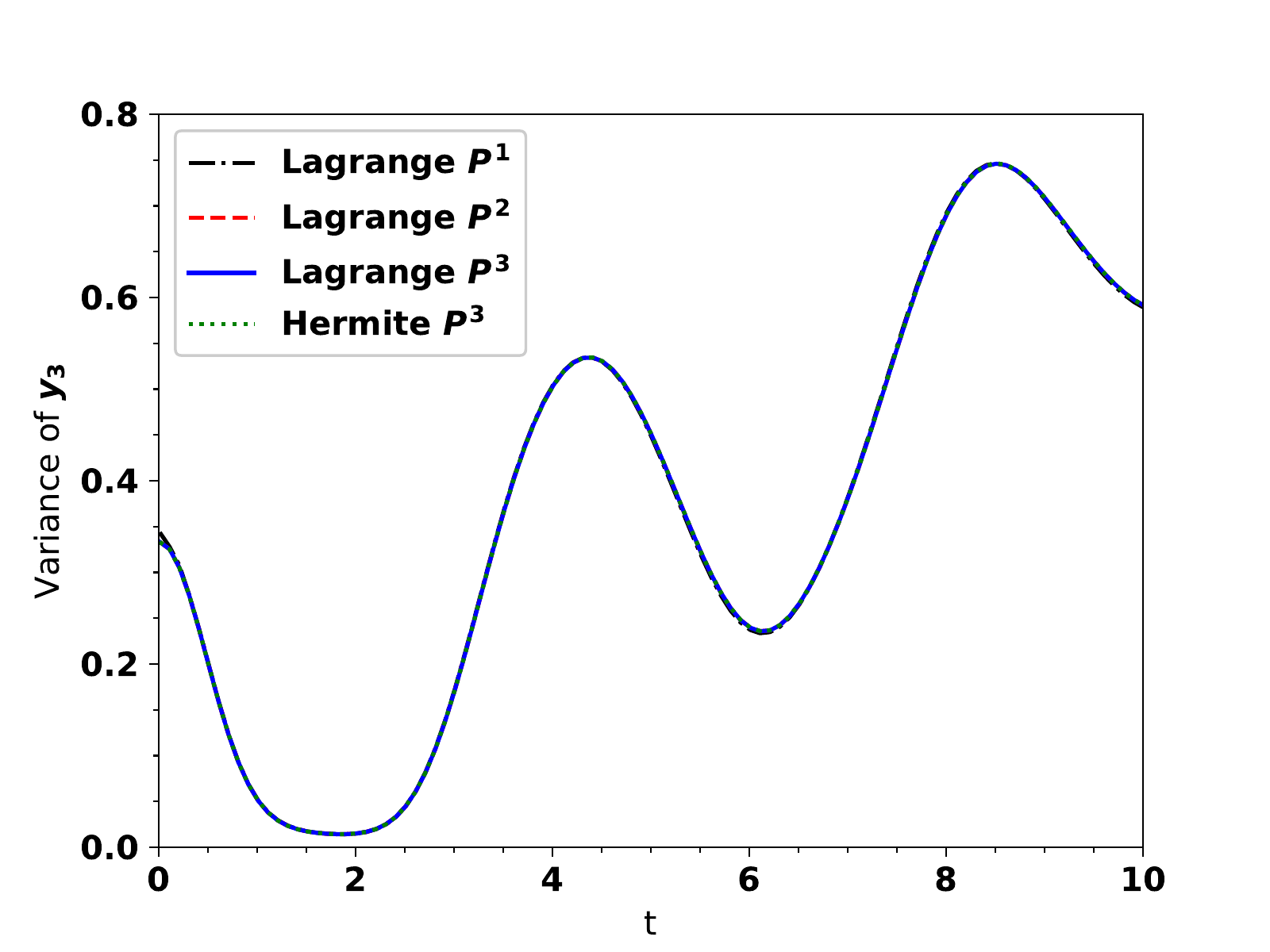}}
	\end{center}
	\caption{Time evolution of the variance of the solution in one and two-dimensional random inputs for K-O problem with adaptive sparse grid method.   $\varepsilon = 10^{-4}$. Left: one-dimensional case; right: two-dimensional case; top: $y_1$; middle: $y_2$; bottom: $y_3$.
	}
	\label{fig:KO_1D_2D}
\end{figure}


\begin{figure}[htp]
	\begin{center}
		\subfigure[$y_1, t=1$]{\includegraphics[width=.49\textwidth]{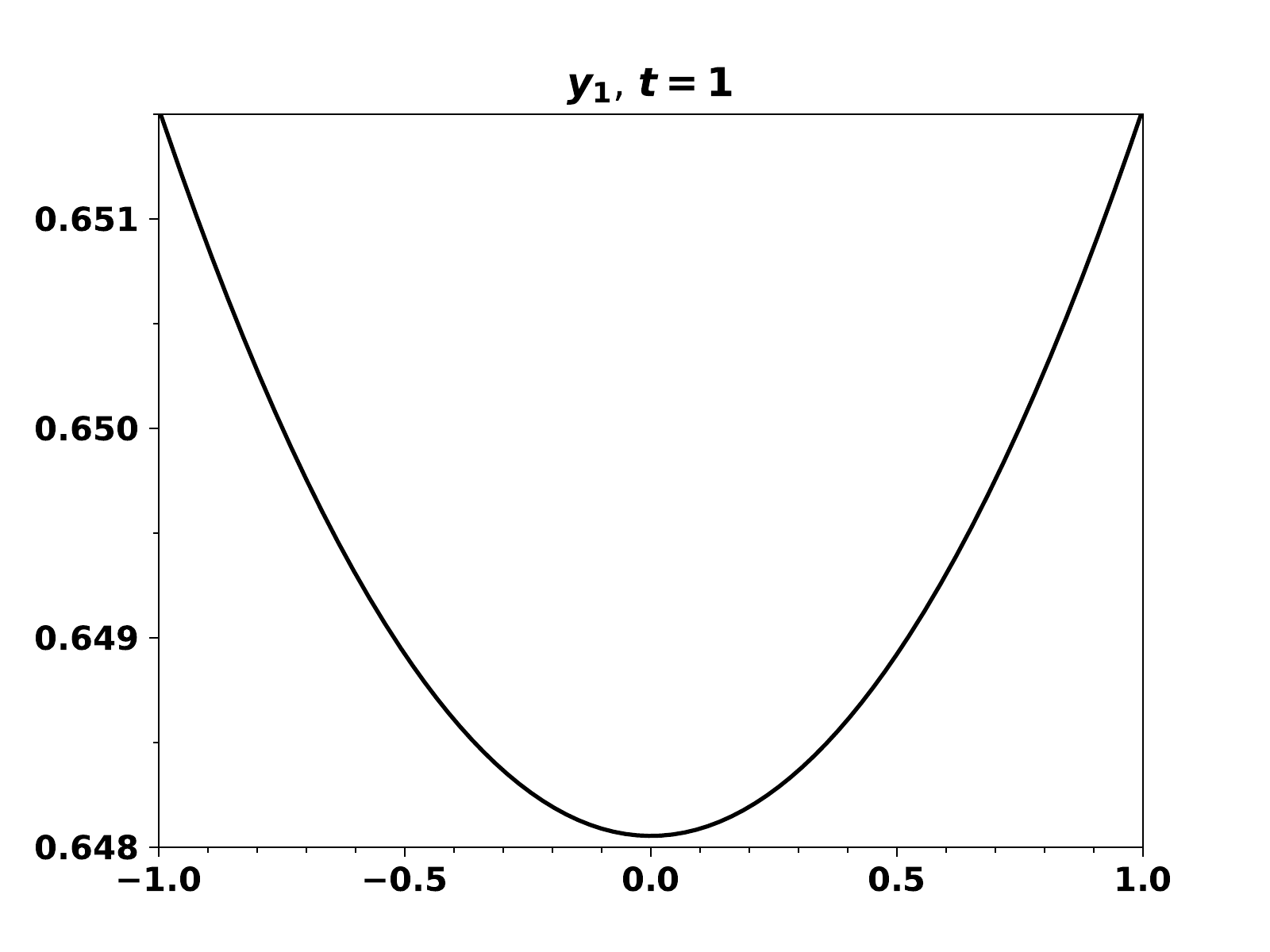}}
		\subfigure[$y_1,t=60$]{\includegraphics[width=.49\textwidth]{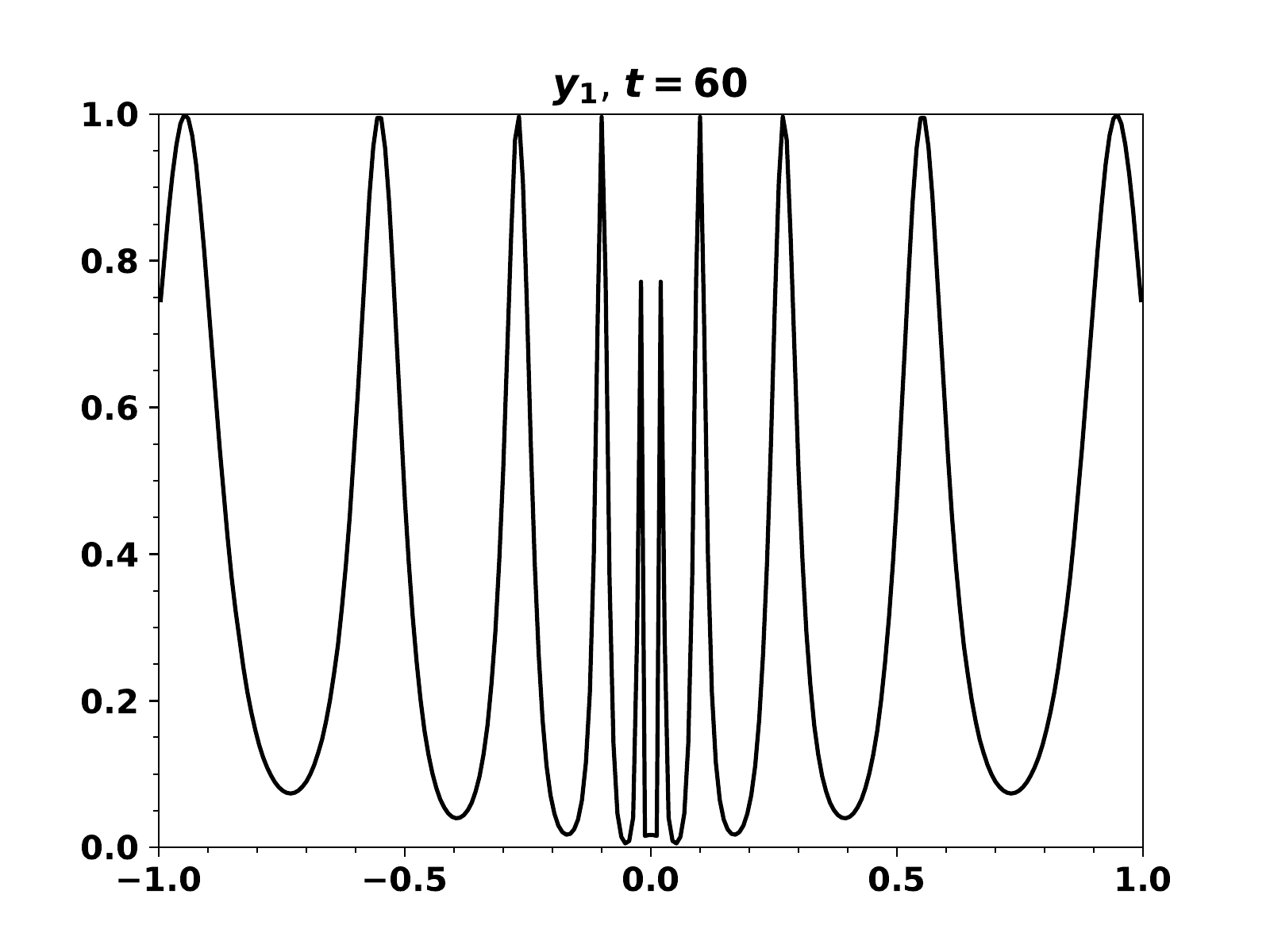}}\\
		\subfigure[$y_2, t=1$]{\includegraphics[width=.49\textwidth]{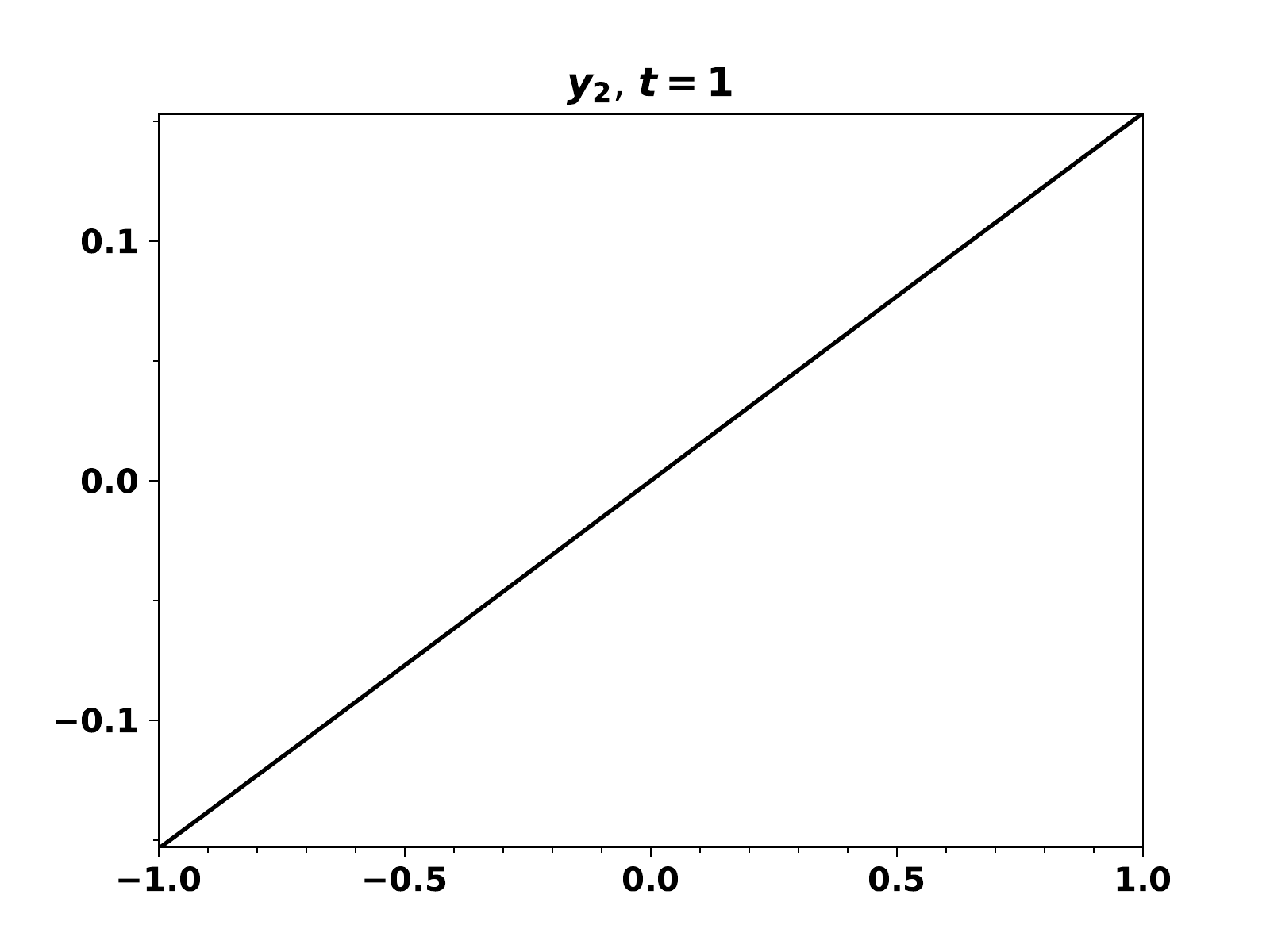}}
		\subfigure[$y_2,t=60$]{\includegraphics[width=.49\textwidth]{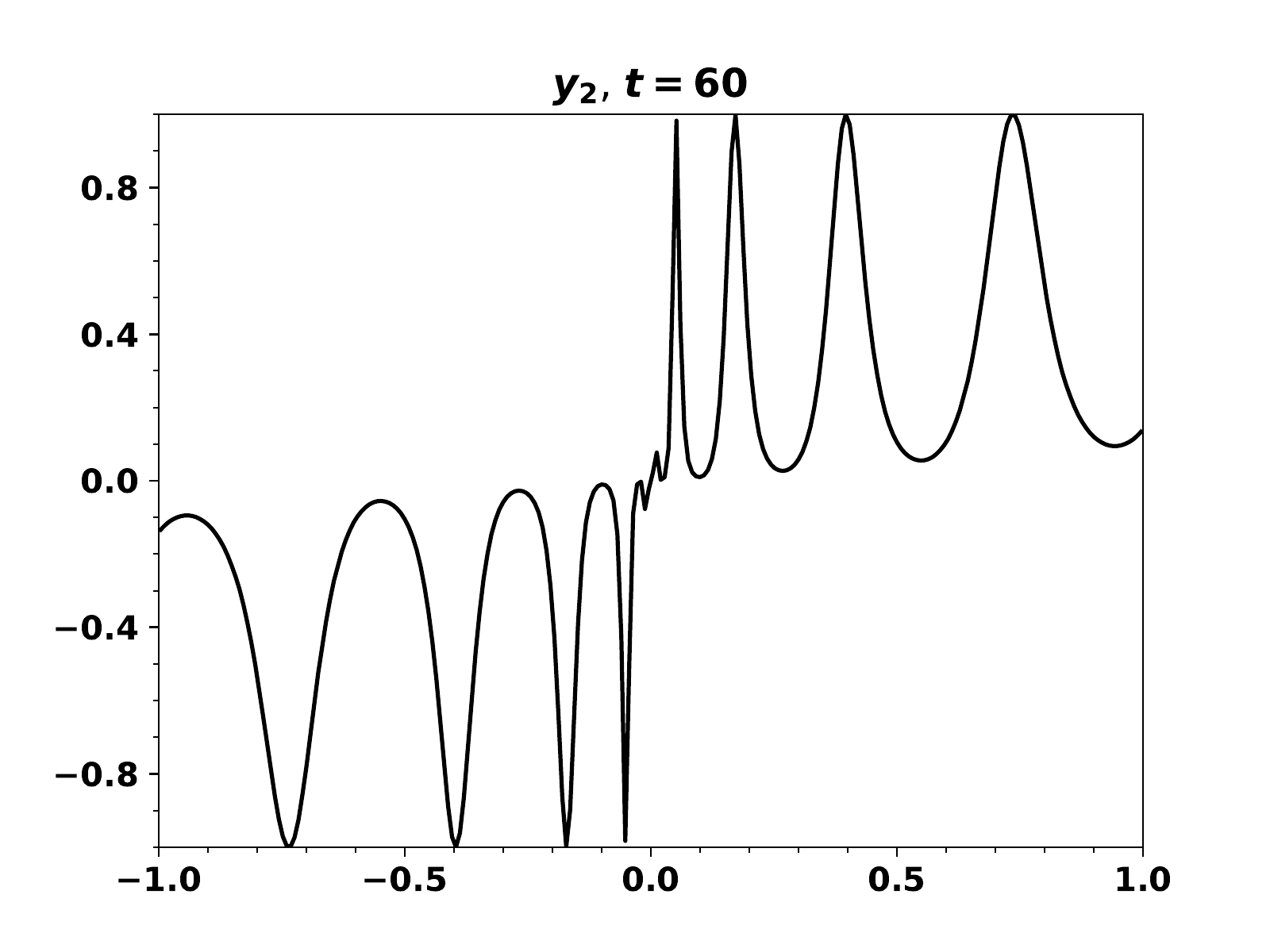}}\\
		\subfigure[$y_3, t=1$]{\includegraphics[width=.49\textwidth]{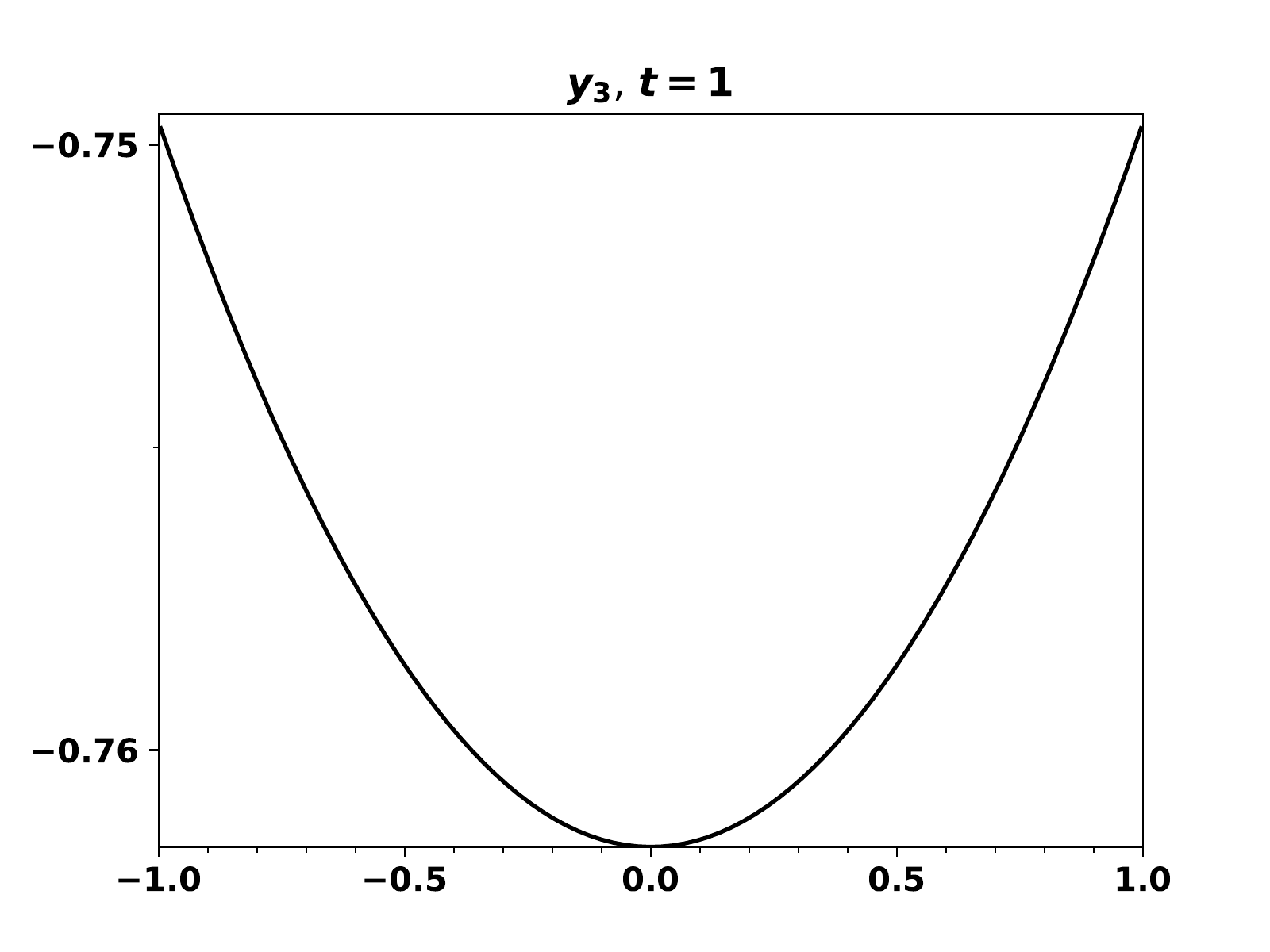}}
		\subfigure[$y_3,t=60$]{\includegraphics[width=.49\textwidth]{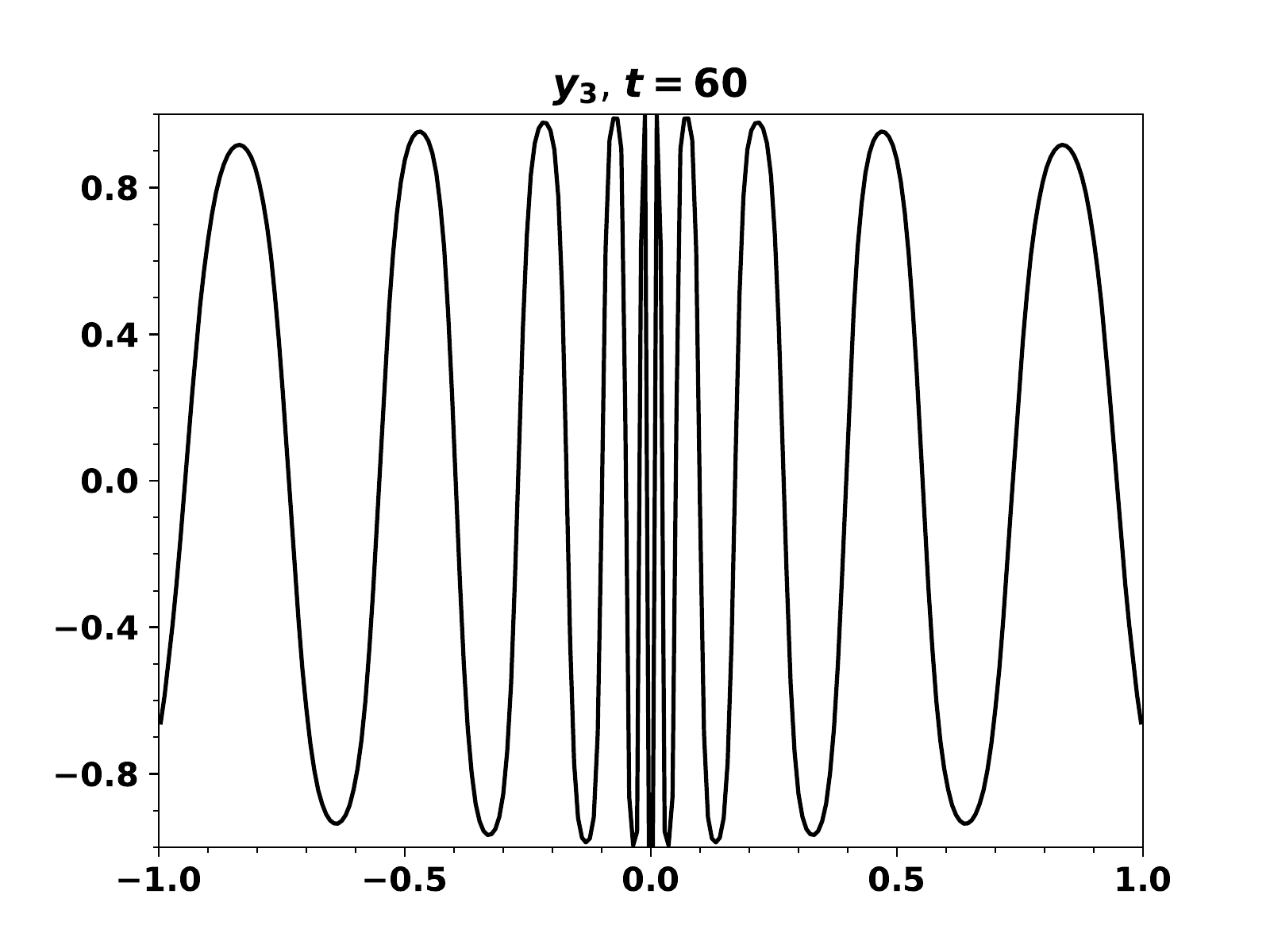}}
%
	\end{center}
	\caption{Realization of the solution ($y_1, y_2, y_3$) in one-dimensional random inputs for K-O problem with adaptive sparse grid method. Top: $y_1$; middle: $y_2$; bottom: $y_3$.   Lagrange $P^2,$ $\varepsilon = 10^{-4}$.
	}
	\label{fig:KO_1D_realization}
\end{figure}

For two-dimensional case, the random initial conditions are chosen as
\begin{align*}
y_1(0) = 1.0, \quad y_2(0) = 0.1Y_1(0;\omega), \quad y_3(0) = Y_2(0;\omega).
\end{align*}
In the Hermite bases, we also need the evolving equations for the derivatives
\begin{align*}
&\frac{d}{dt} (y_1)_{Y_1}  = (y_1)_{Y_1}  y_3 + y_1 (y_3)_{Y_1}, \quad \frac{d}{dt} (y_1)_{Y_2} = (y_1)_{Y_2}  y_3 + y_1 (y_3)_{Y_2}, \\
&\frac{d}{dt} (y_1)_{Y_1 Y_2}  = (y_1)_{Y_1 Y_2}  y_3 + y_1 (y_3)_{Y_1 Y_2} + (y_1)_{Y_1}  (y_3)_{Y_2} + (y_1)_{Y_2} (y_3)_{Y_1}, \\
&\frac{d}{dt} (y_2)_{Y_1} = -(y_2)_{Y_1}  y_3 - y_2 (y_3)_{Y_1}, \quad \frac{d}{dt} (y_2)_{Y_2} = -(y_2)_{Y_2}  y_3 - y_2 (y_3)_{Y_2}, \\
&\frac{d}{dt} (y_2)_{Y_1 Y_2}  = -(y_2)_{Y_1 Y_2}  y_3 - y_2 (y_3)_{Y_1 Y_2} - (y_2)_{Y_1}  (y_3)_{Y_2} - (y_2)_{Y_2} (y_3)_{Y_1}, \\
&\frac{d}{dt} (y_3)_{Y_1}  = -2 y_1 (y_1)_{Y_1}  + 2 y_2 (y_2)_{Y_1}, \quad \frac{d}{dt} (y_3)_{Y_2} = -2 y_1 (y_1)_{Y_2}  + 2 y_2 (y_2)_{Y_2},\\
&\frac{d}{dt} (y_3)_{Y_1 Y_2}  = -2 \left[ (y_1)_{Y_1} (y_1)_{Y_2}  + y_1 (y_1)_{Y_1 Y_2} \right] + 2 \left[ (y_2)_{Y_1} (y_2)_{Y_2}  + y_2 (y_2)_{Y_1 Y_2} \right].
\end{align*}
In the right part of  Figure \ref{fig:KO_1D_2D},  we present the evolution of the variance of the solutions $(y_1, y_2, y_3)$ over time interval [0,10] for Lagrange and Hermite bases.  The adaptive grids at $t=10$ with various bases are shown in Figure \ref{fig:KO_2D_grid}.  $\varepsilon$ is taken as $ 10^{-4}$ and the maximum mesh level is set to be 10.  In this 2D case, the results of variance are almost the same. The variances by our adaptive method converge and are consistent with those in \cite{ma2009adaptive}.
In this case, the discontinuity region is a line. It is noted that more points are placed around the line $Y_1=0$ which the discontinuity crosses. From Figure \ref{fig:KO_2D_grid}, it is concluded that the Hermite method has the DOF most concentrated towards singularity, and the Lagrange $P^3$ method is the most efficient. This is consistent with the function interpolation result.


\begin{figure}[htp]
	\begin{center}
		\subfigure [Lagrange $P^1$ (12907 DoF)] {\includegraphics[width=.49\textwidth]{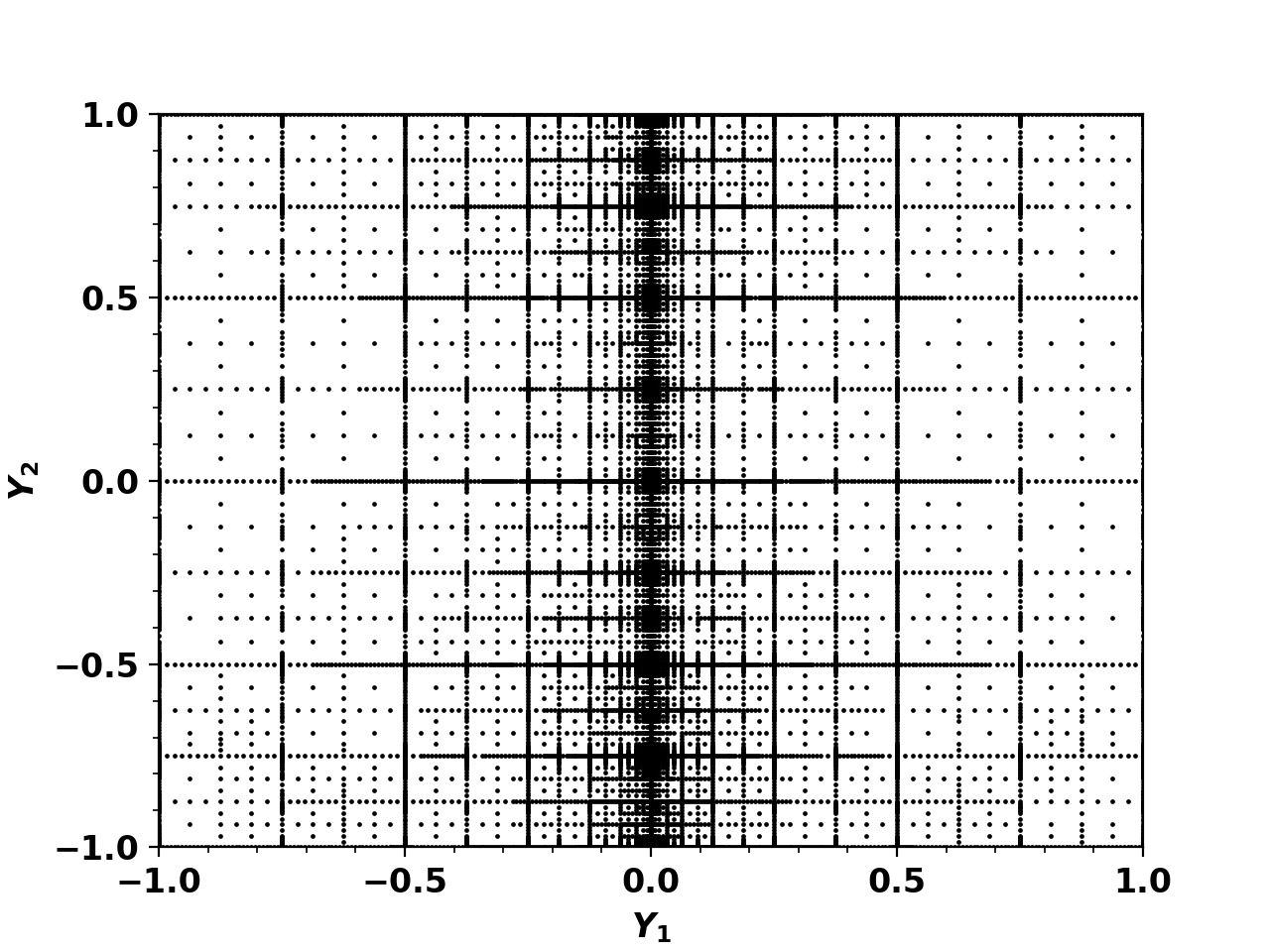}}
		\subfigure [Lagrange $P^2$ (12925 DoF)] {\includegraphics[width=.49\textwidth]{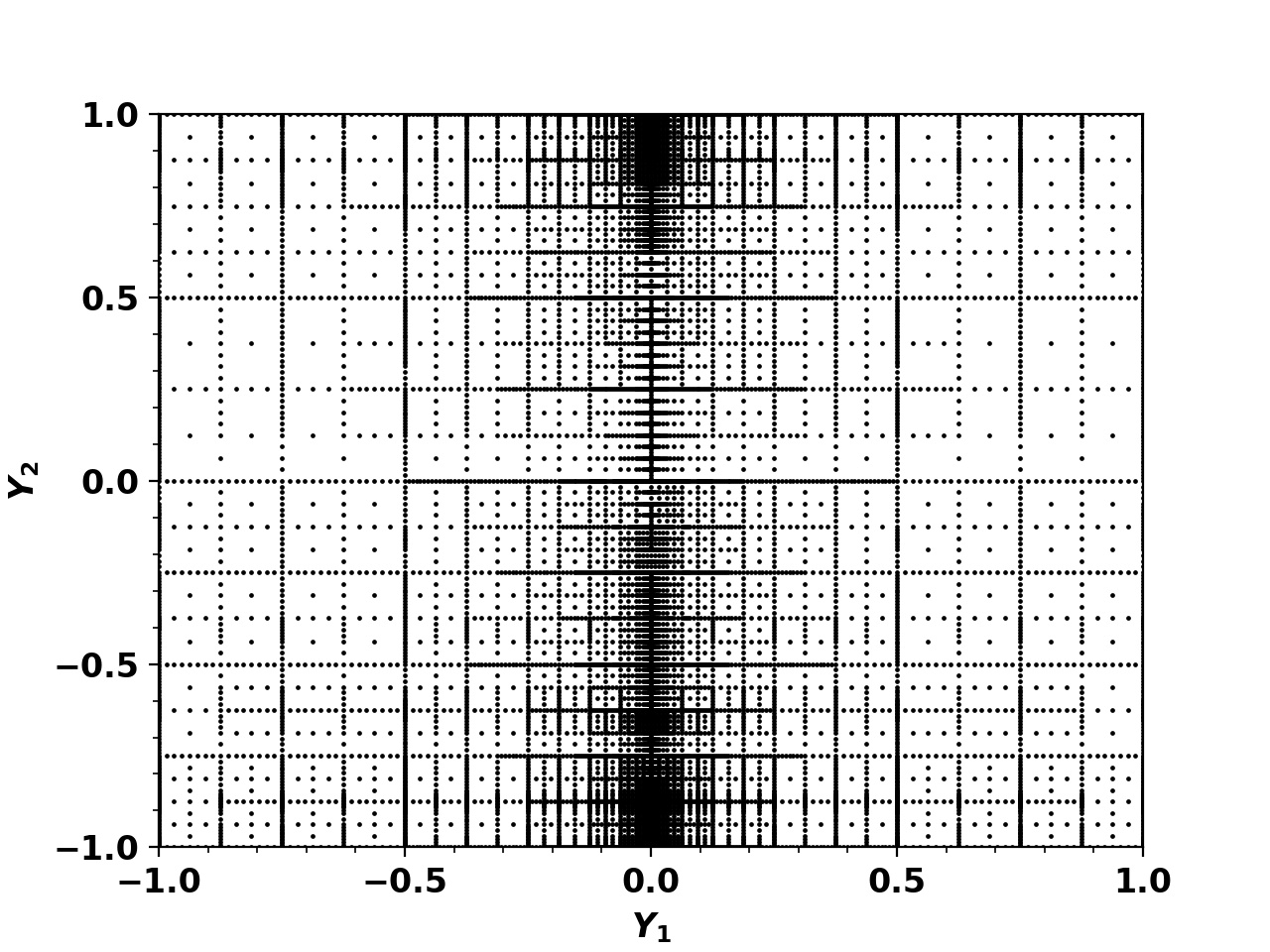}}\\
		\subfigure [Lagrange $P^3$ (8032 DoF)] {\includegraphics[width=.49\textwidth]{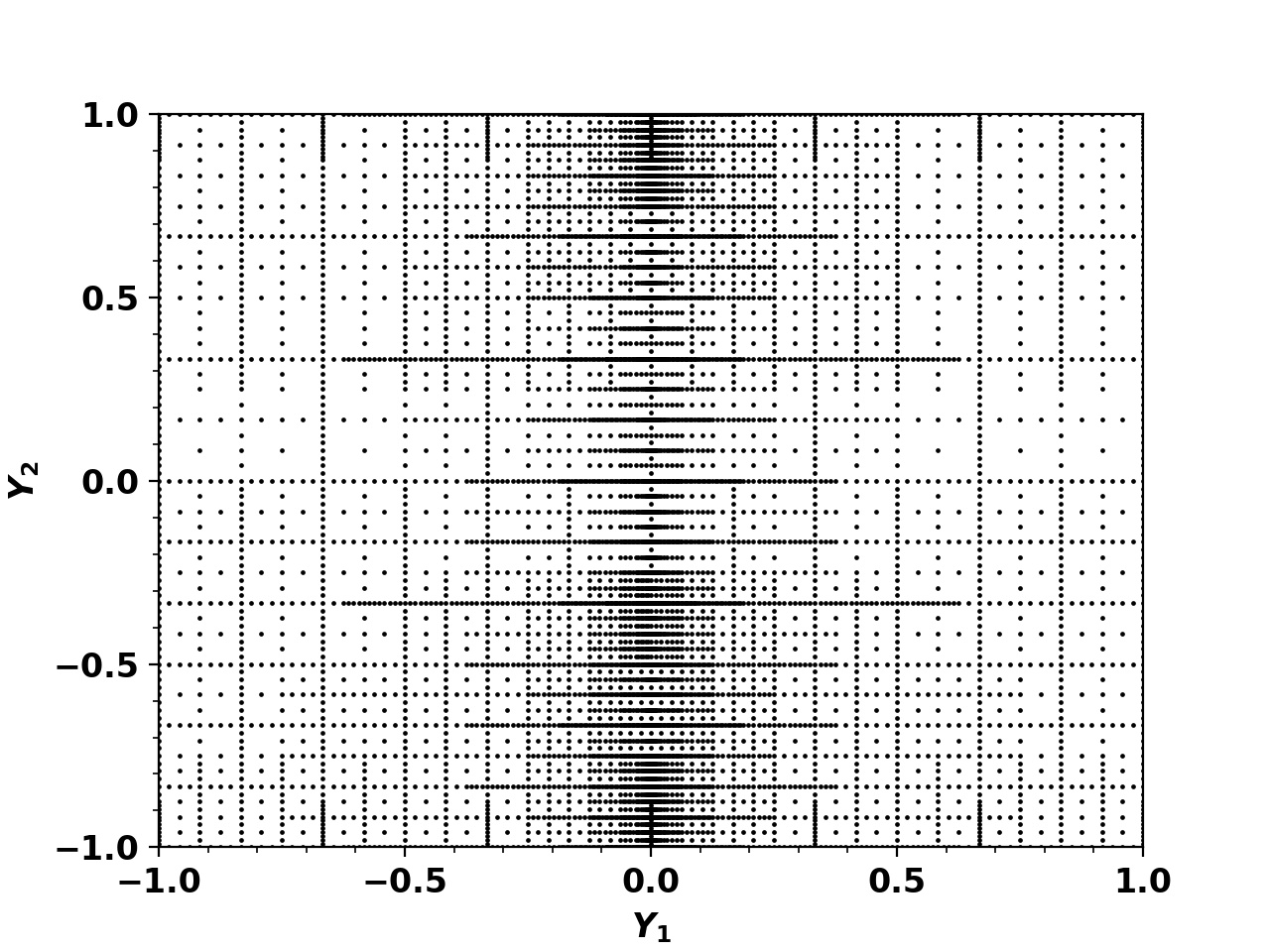}}
		\subfigure [Hermite $P^3$ (22928 DoF)] {\includegraphics[width=.49\textwidth]{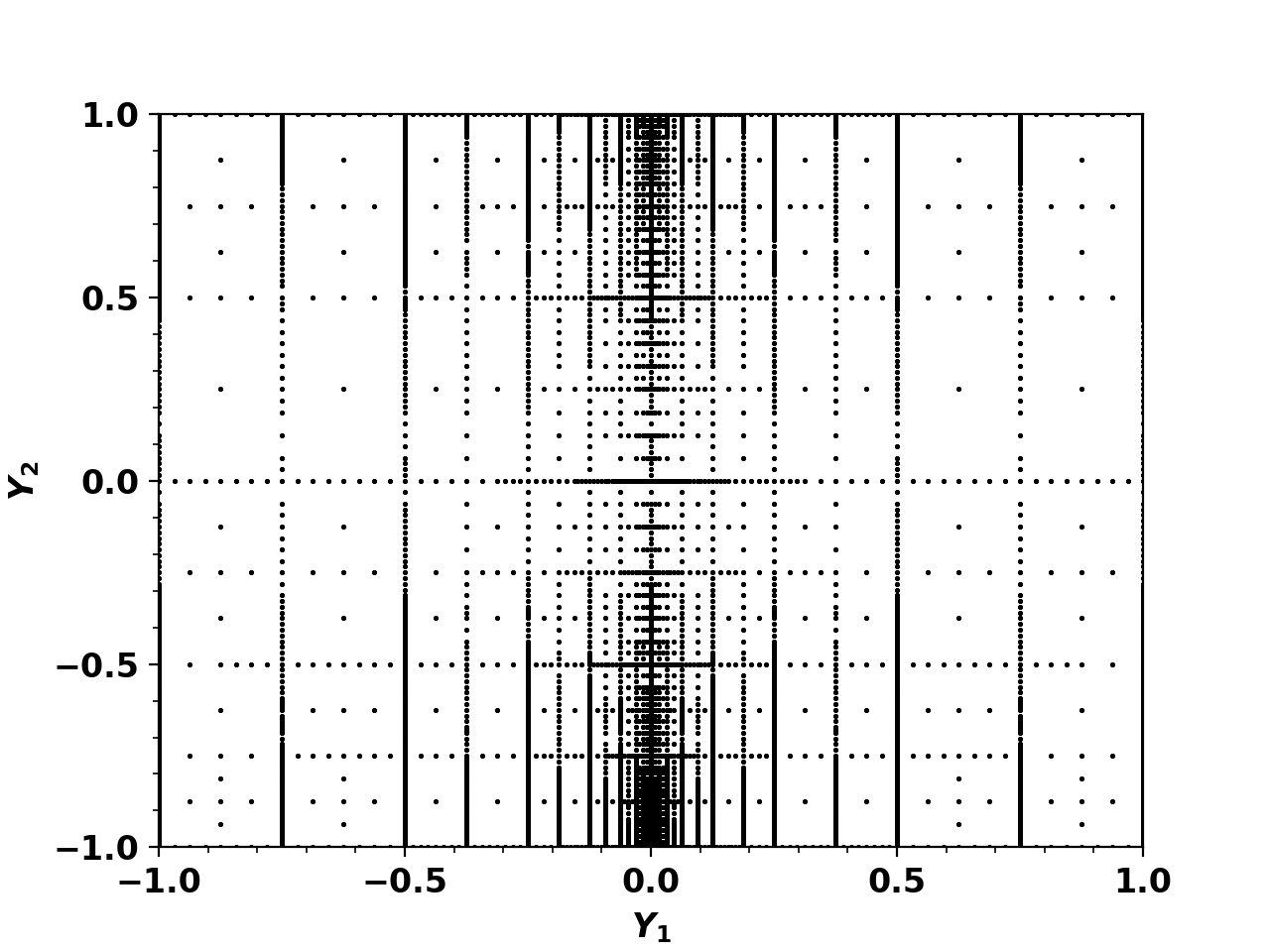}}\\
	\end{center}
	\caption{The adaptive grids at $t=10$ in two-dimensional random inputs for K-O problem.  $\varepsilon = 10^{-4}$. (a) Lagrange $P^1$; (b) Lagrange $P^2$; (c) Lagrange $P^3$; (d) Hermite $P^3$.  
	}
	\label{fig:KO_2D_grid}
\end{figure}

For three-dimensional case, the following random initial conditions are used
\begin{align*}
y_1(0) = Y_1(0;\omega), \quad y_2(0) = Y_2(0;\omega), \quad y_3(0) = Y_3(0;\omega).
\end{align*}

Because of strong discontinuity (which consists of the planes $Y_1=0$ and $Y_2=0$) and higher dimension, this case is more difficult than previous low dimensional case. In Figure \ref{fig:KO_3D}, we show the  evolution of the variance of the solution $(y_1, y_3)$ because of the symmetry of $y_1$  and $y_2$ (the figures of them are the same). When the dimension increases, more and more derivatives are involved which makes the Hermite schemes much more complicated than the Lagrange counterparts, so we only show the results obtained by the  Lagrange bases.   $\varepsilon$ is taken as $10^{-5}$ for $P^1$ and $10^{-4}$ for $P^2$ and $P^3$.  The maximum mesh level is set to be 7. From results of this 3D case, we can observe that the variance with $P^2$ or $P^3$ coincide, while  $P^1$ results demonstrate some discrepancy when $t\ge4.$ This demonstrates that high order bases outperform lower order ones in this 3D case.
The results we have are comparable with \cite{ma2009adaptive}.  

\begin{figure}[htp]
	\begin{center}
		\subfigure[$y_1$]{\includegraphics[width=.49\textwidth]{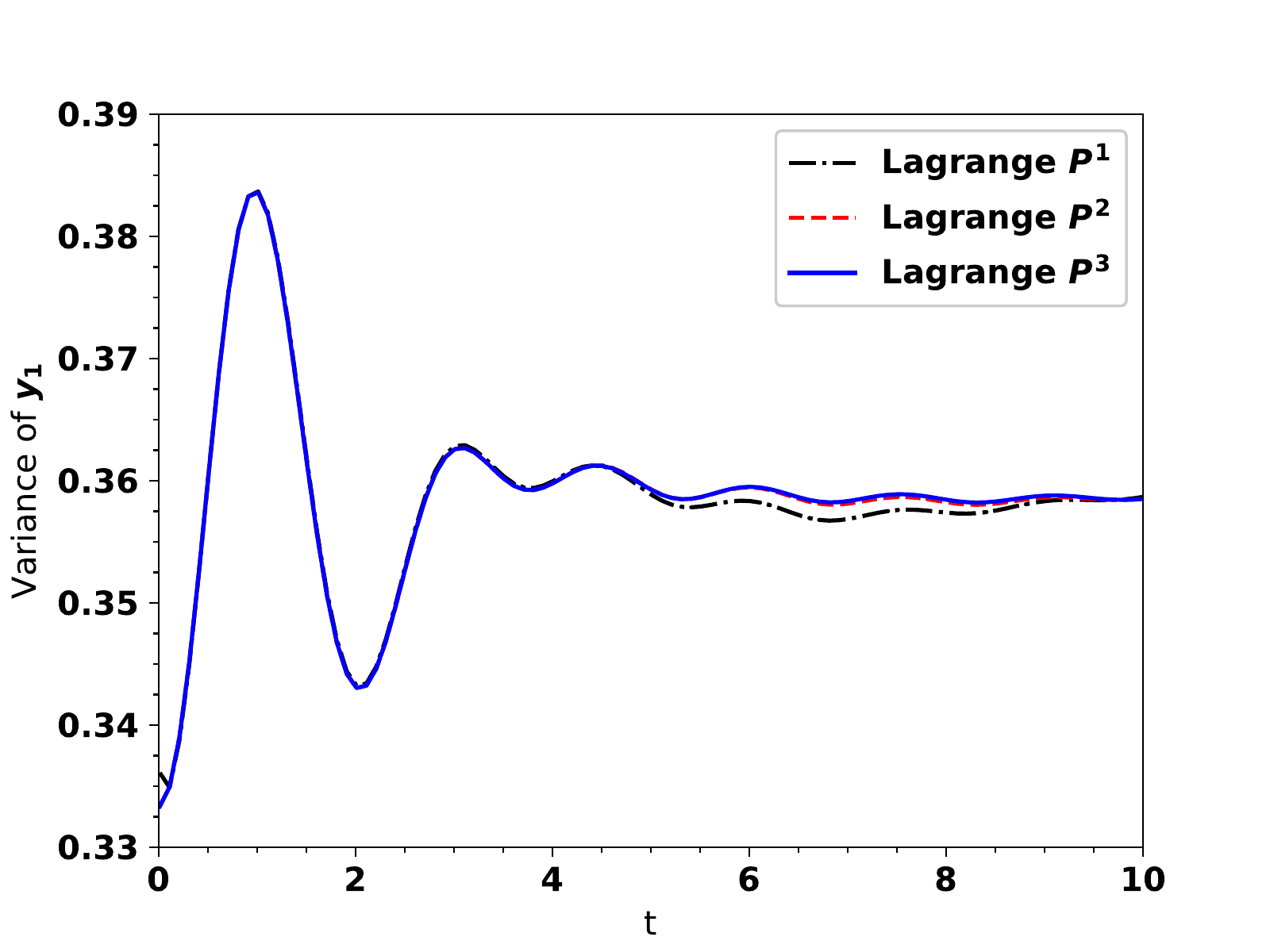}}
		\subfigure[$y_3$]{\includegraphics[width=.49\textwidth]{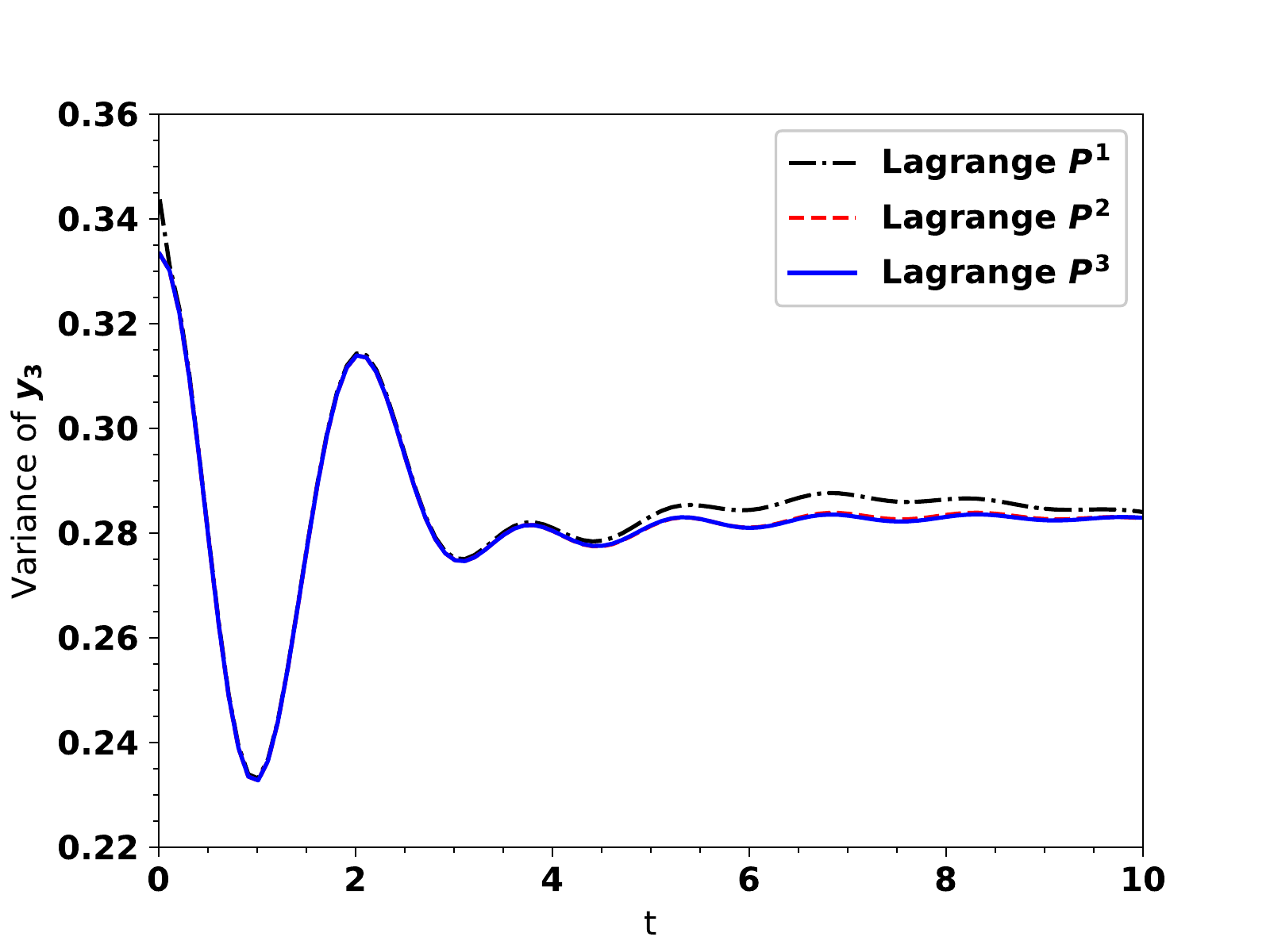}}\\
	\end{center}
	\caption{Time evolution of the variance of the solution in three-dimensional random inputs for K-O problem.   (a) $y_1$; (b) $y_3$. 
	}
	\label{fig:KO_3D}
\end{figure}

%
%
%
%
%


\section{Conclusions and future work}
\label{sec:conclusion}

This work introduces a systematic framework of (adaptive) sparse grid collocation schemes for high-order piecewise polynomial space. We consider both Lagrange and Hermite interpolation methods on nested collocation points.   For function interpolation, it was verified that higher order methods perform better for smooth functions, and the Hermite interpolation methods provide a  solution representation more concentrated towards singularities.  In a separate work  \cite{huang2019adaptive}, we apply the collocation scheme to facilitate the computation of adaptive multiresolution DG scheme for nonlinear hyperbolic equations. It was found in  \cite{huang2019adaptive} that the Hermite interpolation provides more stable numerical solution than Lagrange interpolation. Another possible application of this work is to construct  adaptive semi-Lagrangian schemes, which will be explored in the future.

\appendix
\section{Proof of Lemma \ref{lem1} }
\label{sec:append2}


	Using the definition \eqref{eq:point}, \eqref{eq:nested_relation1} can be represented as
	\begin{align}
	\label{eq:nested_relation2}
	2 x^{0}_{i,0} = x^{0}_{r,0}, \quad \text{or} \quad 
	2 x^{0}_{i,0} = 1+ x^{0}_{r,0}.
	\end{align}
	Therefore, we can find  $\begin{pmatrix} 2P+2 \\ P+1\\ \end{pmatrix}$ different choices of $\{\alpha_i\}$ for general $P$, and the values of $\{\alpha_i\}$ can be obtained from solving \eqref{eq:nested_relation2}. 
	 For example, when $P=0$, \eqref{eq:nested_relation2} gives
	\begin{align*}
	2x^{0}_{0,0} =  x^{0}_{0,0}  \Rightarrow x^{0}_{0,0}=0, \quad
	\text{or} \quad   
	2x^{0}_{0,0} = 1+x^{0}_{0,0} \Rightarrow  x^{0}_{0,0}=1.
	\end{align*}
	This implies $X^{0}_{n}=\{2^{-n}j\}_{j=0}^{2^n-1}$ or $X^{0}_{n}=\{2^{-n}j\}_{j=1}^{2^n}$.
	
	However, for $P\geq1$, there will be some redundant counts causing points to overlap. This includes 
	\begin{align*}
	\left\{ \begin{array}{ll}
	x^{j}_{0,n-1}=x^{2j}_{0,n}\\
	x^{j}_{1,n-1}=x^{2j}_{1,n}\\
	\end{array} \right.
	\quad \text{or} \quad \left\{ \begin{array}{ll}
	2 x^{0}_{0,0} =  x^{0}_{0,0} \\
	2 x^{0}_{1,0} =  x^{0}_{1,0} \\
	\end{array} \right.  \quad \Rightarrow \quad x^{0}_{0,0}=x^{0}_{1,0}=0
	\end{align*}
	and
	\begin{align*}
	\left\{ \begin{array}{ll}
	x^{j}_{P-1,n-1}=x^{2j+1}_{P-1,n}\\
	x^{j}_{P,n-1}=x^{2j+1}_{P,n}\\
	\end{array} \right.
	\quad \text{or}\quad
	\left\{ \begin{array}{ll}
	2 x^{0}_{P-1,0} = 1+x^{0}_{P-1,0}, \\
	2 x^{0}_{P,0}= 1+x^{0}_{P,0}
	\end{array}
	\right. \quad \Rightarrow \quad x^{0}_{P-1,0}=x^{0}_{P,0}=1.
	\end{align*}
	Therefore, we need to exclude those 
	$2\begin{pmatrix} 2P \\ P-1\\\end{pmatrix}$ 
	cases. The intersect of those two cases consists of $\begin{pmatrix} 2P-2 \\ P-3\\ \end{pmatrix}$ choices. Consequently, there exists 
	$\begin{pmatrix} 2P+2 \\ P+1\\	\end{pmatrix} - 2  \begin{pmatrix} 2P \\ P-1\\	\end{pmatrix} + \begin{pmatrix} 2P-2 \\ P-3\\	\end{pmatrix}$ types of the nested points for $P\geq0$, and the lemma is proved.

\section{Interpolation basis functions in 1D }
\label{sec:append1}

Here, we list some choices of interpolation points and corresponding basis functions $\{\phi_{i,l}\}$ and $\{\varphi_{i,l}\}.$ Note that the basis functions in $W^{K}_1$ are piecewise polynomials, and they are all supported on either interval $I_l:=(0,\frac{1}{2})$ or $I_r:=(\frac{1}{2},1)$ and vanish on the other half. Therefore, for simplicity of notation, we only declare the function on its support. 

\subsection{All interpolation basis functions with $K=0$}\label{sec:append1-0}

When $K=0$, we can only take $P=M=0$. In this case, there are 2 types of nested points.
\begin{itemize}
	\item type 1: The interpolation points are 
	$$X^{0}_{0}=\{0^+\}, \quad  \widetilde{X}^{0}_{1}=\{ \left( \frac{1}{2}\right)^+ \}.$$ 
	The basis functions are
	\begin{align*}
	\begin{array}{lll}
	\phi_{0,0}(x)=1,  && 
	\varphi_{0,0}(x)|_{I_r}=1. \\
	\end{array}
	\end{align*} 
	
	\item type 2: The interpolation points are 
	$$X^{0}_{0}=\{1^{-}\}, \quad \widetilde{X}^{0}_{1}=\{\left(\frac{1}{2}\right)^-\}.$$ 
	The basis functions are
	\begin{align*}
	\begin{array}{lll}
	\phi_{0,0}(x)=1,  &&  
	\varphi_{0,0}(x)|_{I_l}=1. \\
	\end{array}
	\end{align*} 
	This is mirror-symmetric to type 1 with respect to the point $1/2$.
\end{itemize}

\subsection{All interpolation basis functions with $K=1$}\label{sec:append1-1}
We list all possible Lagrange/Hermite interpolation basis functions $\{\phi_{i,l}\}$ and $\{\varphi_{i,l}\}$ with $K=(P+1)(M+1)-1=1$. This means we can choose $P=0, M=1$ or  $P=1, M=0.$ 

%
%

\subsubsection{Hermite interpolation $P=0$ and $M=1$}\label{sec:append1.1}
There are 2 types of nested points.
\begin{itemize}
	\item type 1: The interpolation points are 
		$$X^{0}_{0}=\{0^+\}, \quad  \widetilde{X}^{0}_{1}=\{ \left( \frac{1}{2}\right)^+ \}.$$ 
		The basis functions are
	\begin{align*}
	\begin{array}{lll}
	\phi_{0,0}(x)=1,  &&  \phi_{0,1}(x)=x,\\
	\varphi_{0,0}(x)|_{I_r}=1, && \varphi_{0,1}(x)|_{I_r}=x-\frac{1}{2}. \\
	\end{array}
	\end{align*} 
	
	\item type 2: The interpolation points are 
		$$X^{0}_{0}=\{1^{-}\}, \quad \widetilde{X}^{0}_{1}=\{\left(\frac{1}{2}\right)^-\}.$$ 
	The basis functions are
	\begin{align*}
	\begin{array}{lll}
	\phi_{0,0}(x)=1,  &&  \phi_{0,1}(x)=x-1,\\
	\varphi_{0,0}(x)|_{I_l}=1, && \varphi_{0,1}(x)|_{I_l}=x-\frac{1}{2}. \\
	\end{array}
	\end{align*} 
	Actually, type 2 is mirror-symmetric to  type 1 with respect to the point $1/2$.
\end{itemize}
	
\subsubsection{Lagrange interpolation $P=1$ and $M=0$}\label{sec:append1.2}
There are 4 types of nested points.
	\begin{itemize}
		\item type 1: The interpolation points are 
		$$X^{1}_{0}=\{ 0^+, 1^-\},\quad \widetilde{X}^{1}_{1}=\{ \left(\frac{1}{2}\right)^{-}, \left(\frac{1}{2}\right)^{+}\}.$$
		The basis functions are 
		\begin{align*} 
		\begin{array}{lll}
		\phi_{0,0}(x)=-x+1,  && \phi_{1,0}(x)=x,\\
		\varphi_{0,0}(x)|_{I_l}=2x, && \varphi_{1,0}(x)|_{I_r}=-2x+2.
		\end{array}
		\end{align*}
		
		\item type 2: The interpolation points are 
		$$X^{1}_0=\{ \frac{1}{3}, \frac{2}{3} \}, \quad \widetilde{X}^{1}_1=\{ \frac{1}{6}, \frac{5}{6} \}.$$
		The basis functions are 
		\begin{align*} 
		\begin{array}{lll}
		\phi_{0,0}(x)=-3x+2,  && \phi_{1,0}(x)=3x-1\\
		\varphi_{0,0}(x)|_{I_l}= -6x+2,  && 
		\varphi_{1,0}(x)|_{I_r}=6x-4.
		\end{array} 
		\end{align*}
		
		\item type 3: The interpolation points are 
		$$X^{1}_0=\{ 0^{+}, \left(\frac{1}{2}\right)^{+}\}, \quad \widetilde{X}^{1}_1=\{ \left(\frac{1}{4}\right)^+, \left(\frac{3}{4}\right)^+ \}.$$
		The basis functions are 
		\begin{align*} 
		\begin{array}{lll}
		\phi_{0,0}(x)=-2x+1,  && \phi_{1,0}(x)=2x\\
		\varphi_{0,0}(x)|_{I_l}=4x,  &&
		\varphi_{1,0}(x)|_{I_r}=4x-2.
		\end{array}
		\end{align*}
		
		\item type 4: The interpolation points are 
		$$X^{1}_0=\{ \left(\frac{1}{2}\right)^{-}, 1^{-} \}, \quad \widetilde{X}^{1}_1=\{ \left(\frac{1}{4}\right)^-, \left(\frac{3}{4}\right)^- \}.$$ 
		The basis functions are 
		\begin{align*} 
		\begin{array}{lll}
		\phi_{0,0}(x)=-2x+2,  && \phi_{1,0}(x)=2x-1\\
		\varphi_{0,0}(x)|_{I_l}=-4x+2,  && 
		\varphi_{1,0}(x)|_{I_r}=-4x+4.
		\end{array}
		\end{align*}
		This is mirror-symmetric to type 3 with respect to the point $1/2$.
	\end{itemize}

\subsection{Special interpolation basis functions with $K=2, 3$}\label{sec:append1-2}
For higher order polynomials, there are many types of choices one can make as illustrated in Lemma \ref{lem1}. Therefore, we only list some choices which are used in the numerical experiments.

\subsubsection{Lagrange interpolation $P=2$ and $M=0$}
\label{sec:append2.2}
The interpolation points are 
\begin{align*}
 	X^{2}_0=\{0^+, \left(\frac{1}{2}\right)^-, 1^- \},\quad 
 	\widetilde{X}^{2}_1=\{ \left(\frac{1}{4}\right)^-, \left(\frac{1}{2}\right)^+, \left(\frac{3}{4}\right)^- \}.
\end{align*}
The basis functions are 
\begin{align*} 
	\begin{array}{lll}
	\phi_{0,0}(x)=2(x-\frac{1}{2})(x-1),  & 	
	\phi_{1,0}(x)= -4x(x-1),	& 
	\phi_{2,0}(x)=2x(x-\frac{1}{2}),\\
	\varphi_{0,0}(x)|_{I_l}=-16x(x-\frac{1}{2}),  &
	\varphi_{1,0}(x)|_{I_r}=8(x-\frac{3}{4})(x-1),& 
	\varphi_{2,0}(x)|_{I_r}=-16(x-\frac{1}{2})(x-1).
	\end{array}
\end{align*}

\subsubsection{Lagrange interpolation $P=3$ and $M=0$} 
\label{sec:append2.4}
The interpolation points are 
\begin{align*}
	X^{3}_0=\{ 0^+, \frac{1}{3},  \frac{2}{3} , 1^-\} , \quad  
	\widetilde{X}^{3}_{1} = \{ \frac{1}{6}, \left(\frac{1}{2}\right)^-, \left(\frac{1}{2}\right)^+, \frac{5}{6} \}. 
\end{align*}
The basis functions are 
\begin{align*}
	\begin{array}{lll}
	\phi_{0,0}(x)= -\frac{9}{2}(x-\frac{1}{3})(x-\frac{2}{3})(x-1), &&
	\phi_{1,0}(x)= \frac{27}{2}x(x-\frac{2}{3})(x-1),\\
	\phi_{2,0}(x)= -\frac{27}{2}x(x-\frac{1}{3})(x-1), &&
	\phi_{3,0}(x)= \frac{9}{2}x(x-\frac{1}{3})(x-\frac{2}{3}), \\
	\varphi_{0,0}(x)|_{I_l}=108 x(x-\frac{1}{3})(x-\frac{1}{2}), &&
	\varphi_{1,0}(x)|_{I_l}=36x(x-\frac{1}{6})(x-\frac{1}{3}), \\
	\varphi_{2,0}(x)|_{I_r}=-36(x-\frac{2}{3})(x-\frac{5}{6})(x-1), &&
	\varphi_{3,0}(x)|_{I_r}=-108(x-\frac{1}{2})(x-\frac{2}{3})(x-1).
	\end{array}
\end{align*}

\subsubsection{Hermite interpolation $P=1$ and $M=1$}
\label{sec:append2.3}
The interpolation points are 
\begin{align*}
	X^{1}_0=\{ 0^+, 1^-\} , \quad  \widetilde{X}^{1}_{1} = \{ \left(\frac{1}{2}\right)^-, \left(\frac{1}{2}\right)^+ \}. 
\end{align*}
The basis functions are 
\begin{align*}
	\begin{array}{lll}
	\phi_{0,0}(x)= 2(x+\frac{1}{2})(x-1)^2, &&
	\phi_{1,0}(x)= -2x^2(x-\frac{3}{2}),  \\
	\phi_{0,1}(x)= x(x-1)^2, &&
	\phi_{1,1}(x)= x^2(x-1), \\
	\varphi_{0,0}(x)|_{I_l}=-16x^2(x-\frac{3}{4}), &&
	\varphi_{1,0}(x)|_{I_r}=16(x-1)^2(x-\frac{1}{4}), \\
	\varphi_{0,1}(x)|_{I_l}=4x^2(x-\frac{1}{2}),  &&
	\varphi_{1,1}(x)|_{I_r}=4(x-1)^2(x-\frac{1}{2}).
	\end{array}
\end{align*}

\section{Proof of Theorem \ref{thm:1}} \label{sec:proof}

We prove \eqref{eqn:relation} and \eqref{eqn:relation3} first.
	We split the error into two parts
	\begin{align*}
	f - \widehat{\cI}^{P,M}_{N}[f] = f - \cI^{P,M}_{N}[f] + \cI^{P,M}_{N}[f] - \widehat{\cI}^{P,M}_{N}[f].
	\end{align*}
	By the property of multi-dimensional interpolation, we have that
	there is a constant $\tilde{C}$ independent of $N$, such that
	\begin{align*}
	 	\|  f - \cI^{P,M}_{N} [f] \|_{W^{s,p}(\Omega_{N})} 
	 	\leq \tilde{C}  \, (h_{N})^{q+1-s} | f |_{W^{q+1,p}(\Omega)} 
	 	\leq \tilde{C} \, 2^{-N(q+1-s)} | f |_{W^{q+1,M,p}(\Omega)}.
	 \end{align*} 
	Therefore, we only need to bound
	\begin{align*}
	\cI^{P,M}_{N} [f] - \widehat{\cI}^{P,M}_{N}[f] 
	=& \sum_{|\bn|_\infty \leq N, \bn\in\mathbb{N}_{0}^{d}} \widetilde{\mathcal{I}}^{P,M}_{n_1,x_1} \circ \cdots \circ \widetilde{\mathcal{I}}^{P,M}_{n_d,x_d} [f]  
	- \sum_{|\bn|_1 \leq N, \bn\in\mathbb{N}_{0}^{d}} \widetilde{\mathcal{I}}^{P,M}_{n_1,x_1} \circ \cdots \circ \widetilde{\mathcal{I}}^{P,M}_{n_d,x_d} [f]  \\
	=& \sum_{\substack{|\bn|_{\infty}\leq N, |\bn|_1 \geq N+1 \\ \bn\in\mathbb{N}_{0}^{d}}} \widetilde{\mathcal{I}}^{P,M}_{n_1,x_1} \circ \cdots \circ \widetilde{\mathcal{I}}^{P,M}_{n_d,x_d} [f] .
	\end{align*}
	
	\noindent
	In what follows, we will estimate the term $\widetilde{\cI}^{P,M}_{n_1,x_1} \circ \cdots \circ \widetilde{\mathcal{I}}^{P,M}_{n_d,x_d}[f]$. 
	For a multi-index $\bn$, let $L= supp(\bn) := \{ \beta_1, \cdots, \beta_\gamma\} \subset \{1,\ldots,d\}$, i.e., $n_{\beta}\geq1$ if $\beta\in L$.  And correspondingly, we define the multi-indexes $\mathbf{j}_{L}$ and $\mathbf{n}_{L} \in \mathbb{N}^d_0$ as 
	\begin{align*}
	(\mathbf{j}_{L})_{\beta} = \left\{ \begin{array}{ll} 
		\lfloor \, j_{\beta}/2 \, \rfloor, & \text{if} \, \beta\in L, \\
		j_{\beta}, & \text{if} \, \beta\notin L, \\
		\end{array}\right.  \quad \text{and} \quad
		(\mathbf{n}_{L})_{\beta} = \left\{ \begin{array}{ll} 
		n_{\beta}-1, & \text{if} \, \beta\in L, \\
		0, & \text{if} \, \beta\notin L. \\
		\end{array}\right. 
	\end{align*} 
	Then
	\begin{align*}
	\| \widetilde{\mathcal{I}}^{P,M}_{n_1,x_1} \circ \cdots \circ \widetilde{\mathcal{I}}^{P,M}_{n_d,x_d} [f] \| _{L^{p}(I^{\bj}_{\bn})} 
	\leq & (C_{1})^\gamma (C_2)^{d-\gamma} (1+2^{q+1})^{\gamma} h_{\beta_1}^{q+1} \cdots h_{\beta_\gamma}^{q+1} |f|_{W^{q+1,M,p,L}(I_{\bn_L}^{\bj_L})} \notag \\
	\leq & \bar{C}^d (1+2^{q+1})^{\gamma} 2^{-n_{\beta_1}(q+1)} \cdots 2^{-n_{\beta_\gamma} (q+1)} |f|_{W^{q+1,M,p,L}(I_{\bn_L}^{\bj_L})} \notag \\
	= & \bar{C}^d (1+2^{q+1})^{\gamma} 2^{-(n_{\beta_1}+\cdots+n_{\beta_\gamma})(q+1)}  |f|_{W^{q+1,M,p,L}(I_{\bn_L}^{\bj_L})} \notag \\
	=& \bar{C}^d (1+2^{q+1})^{\gamma} 2^{-|\bn|_{1}(q+1)}  |f|_{W^{q+1,M,p,L}(I_{\bn_L}^{\bj_L})}
	\end{align*}
	The last equality holds because when $i\notin\{\beta_1,\cdots,\beta_{\gamma}\}$, we have $n_i=0$. On the other hand,
	\begin{align*}
	& | \widetilde{\mathcal{I}}^{P,M}_{n_1,x_1} \circ \cdots \circ \widetilde{\mathcal{I}}^{P,M}_{n_d,x_d} [f] |^2 _{H^{1}(I^{\bj}_{\bn})}  \\
	=& \sum_{m=1}^{d} | \partial_{x_m} \widetilde{\mathcal{I}}^{P,M}_{n_1,x_1} \circ \cdots \circ \widetilde{\mathcal{I}}^{P,M}_{n_d,x_d} [f] |^2 _{L^{2}(I^{\bj}_{\bn})} \\
	=& \sum_{x_m\in L} | \partial_{x_m} \widetilde{\mathcal{I}}^{P,M}_{n_1,x_1} \circ \cdots \circ \widetilde{\mathcal{I}}^{P,M}_{n_d,x_d} [f] |^2 _{L^{2}(I^{\bj}_{\bn})} 
	+ \sum_{x_m \notin L} | \partial_{x_m} \widetilde{\mathcal{I}}^{P,M}_{n_1,x_1} \circ \cdots \circ \widetilde{\mathcal{I}}^{P,M}_{n_d,x_d} [f] |^2 _{L^{2}(I^{\bj}_{\bn})} \\
	\leq&  \sum_{x_m\in L}  (C_1)^{2\gamma} (C_2)^{2(d-\gamma)} (1+2^{q+1})^{2(\gamma-1)} (1+2^{q+1-1})^2 h_{\beta_1}^{2(q+1)} \cdots h_{\beta_\gamma}^{2(q+1)} h_{m}^{-2} |f|^2_{W^{q+1,M,2,L}(I_{\bn_L}^{\bj_L})} \\
	+&  \sum_{x_m\notin L}  (C_1)^{2\gamma} (C_2)^{2(d-\gamma)} (1+2^{q+1})^{2\gamma}  h_{\beta_1}^{2(q+1)} \cdots h_{\beta_\gamma}^{2(q+1)} |f|^2_{W^{q+1,M,2,L}(I_{\bn_L}^{\bj_L})} \\
	\leq&  \sum_{x_m\in L}   (\bar{C})^{2d} (1+2^{q+1})^{2\gamma} \, 2^{-2|\bn|_1 (q+1) +2|\bn|_{\infty} } |f|^2_{W^{q+1,M,2,L}(I_{\bn_L}^{\bj_L})} \\
	+&  \sum_{x_m\notin L}  (\bar{C})^{2d} (1+2^{q+1})^{2\gamma}  \, 2^{-2|\bn|_1 (q+1)} |f|^2_{W^{q+1,M,2,L}(I_{\bn_L}^{\bj_L})} \\
	 \leq& d\,\bar{C}^{2d} (1+2^{q+1})^{2\gamma} \, 2^{-2|\bn|_1 (q+1) +2|\bn|_{\infty} } |f|^2_{W^{q+1,M,2,L}(I_{\bn_L}^{\bj_L})} .
	\end{align*}
	Hence, summing up on all elements, 
	\begin{align*}
	\| \widetilde{\mathcal{I}}^{P,M}_{n_1,x_1} \circ \cdots \circ \widetilde{\mathcal{I}}^{P,M}_{n_d,x_d} [f] \| _{L^{p}(\Omega_{N})} 
	= & \| \widetilde{\mathcal{I}}^{P,M}_{n_1,x_1} \circ \cdots \circ \widetilde{\mathcal{I}}^{P,M}_{n_d,x_d} [f] \| _{L^{p}(\Omega_{\bn})} \notag\\
	\leq& \sum_{\mathbf{0}\leq \mathbf{j} \leq 2^\bn-\mathbf{1}} \| \widetilde{\mathcal{I}}^{P,M}_{n_1,x_1} \circ \cdots \circ \widetilde{\mathcal{I}}^{P,M}_{n_d,x_d} [f] \| _{L^{p}(I^{\bj}_{\bn})} \notag\\
	\leq &  \sum_{\mathbf{0}\leq \mathbf{j} \leq 2^\bn-\mathbf{1}}  \bar{C}^d (1+2^{q+1})^{\gamma} 2^{-|\bn|_{1}(q+1)}  |f|_{W^{q+1,M,p,L}(I_{\bn_L}^{\bj_L})} \notag \\
	\leq &  \sum_{\mathbf{0}\leq \mathbf{j} \leq 2^{(\bn_{L})}-\mathbf{1}}  \bar{C}^d (1+2^{q+1})^{\gamma} 2^{-|\bn|_{1}(q+1)}  \, 2^{\gamma} \, |f|_{W^{q+1,M,p,L}(I_{\bn_L}^{\bj})} \notag \\
	= &  \bar{C}^d (2+2^{q+2})^{\gamma} 2^{-|\bn|_{1}(q+1)}   \, |f|_{W^{q+1,M,p,L}(\Omega)}, \\
	| \widetilde{\mathcal{I}}^{P,M}_{n_1,x_1} \circ \cdots \circ \widetilde{\mathcal{I}}^{P,M}_{n_d,x_d} [f] |^2_{H^{1}(\Omega_{N})} 
	=& \sum_{\mathbf{0}\leq \mathbf{j} \leq 2^\bn-\mathbf{1}} | \widetilde{\mathcal{I}}^{P,M}_{n_1,x_1} \circ \cdots \circ \widetilde{\mathcal{I}}^{P,M}_{n_d,x_d} [f] |^2 _{H^{1}(\Omega_{\bn})} \notag\\
	\leq &  \sum_{\mathbf{0}\leq \mathbf{j} \leq 2^\bn-\mathbf{1}}  d \, \bar{C}^{2d} (1+2^{q+1})^{2\gamma} 2^{-2|\bn|_{1}(q+1) + 2|\bn|_{\infty}}  |f|^2_{W^{q+1,M,2,L}(I_{\bn_L}^{\bj_L})} \notag \\
	\leq &  \sum_{\mathbf{0}\leq \mathbf{j} \leq 2^{(\bn_{L})}-\mathbf{1}}  d \, \bar{C}^{2d} (1+2^{q+1})^{2\gamma} 2^{-2|\bn|_{1}(q+1)+2|\bn|_{\infty}}  \, 2^{2\gamma} \, |f|^2_{W^{q+1,M,2,L}(I_{\bn_L}^{\bj})} \notag \\
	= &  d\, \bar{C}^{2d} (2+2^{q+2})^{2\gamma} 2^{-2|\bn|_{1}(q+1) + 2|\bn|_{\infty}}   \, |f|^2_{W^{q+1,M,2,L}(\Omega)},
	\end{align*}
	where we have taken into account the overlap of the cells on the fourth line.
	As a consequence,
	\begin{align*}
	& \|\cI^{P,M}_{N} [f] - \widehat{\cI}^{P,M}_{N} [f] \| _{L^{p}(\Omega_{N})}  \\
	\leq &  \sum_{\substack{|\bn|_{\infty}\leq N, |\bn|_1 \geq N+1\\ \bn\in\mathbb{N}_0^d} } \|\widetilde{\mathcal{I}}^{k}_{n_1,x_1} \circ \cdots \circ \widetilde{\mathcal{I}}^{k}_{n_d,x_d} [f] \| _{L^{p}(\Omega_{N})}  \\
	\leq &  \sum_{\substack{|\bn|_{\infty}\leq N, |\bn|_1 \geq N+1\\ \bn\in\mathbb{N}_0^d} } \bar{C}^d (2+2^{q+2})^{\gamma} 2^{-|\bn|_{1}(q+1)}   \, |f|_{W^{q+1,M,p,L}(\Omega)}   \quad 
	\text{with}\, L= supp(\bn), \gamma=|L|\\
	\leq & \sum_{\substack{|\bn|_{\infty}\leq N, |\bn|_1 \geq N+1\\ \bn\in\mathbb{N}_0^d} } \bar{C}^d (2+2^{q+2})^{d} 2^{-(N+1)(q+1)}   \, |f|_{W^{q+1,M,p}(\Omega)}  \\
	\leq &  \bar{C}^d (2+2^{q+2})^{d} 2^{-(N+1)(q+1)}   \, |f|_{W^{q+1,M,p}(\Omega)}  \sum_{\substack{|\bn|_{\infty}\leq N, |\bn|_1 \geq N+1\\ \bn\in\mathbb{N}_0^d} } 1 \\
	\leq &  \bar{C}^d (N+1)^d \left( 1-\frac{1}{d!} \right)  \, \left(2+2^{q+2}\right)^{d} \,  2^{-(N+1) (q+1)} |f|_{W^{q+1,M,p}(\Omega)} \\
	\leq &  \bar{C}^d (N+1)^d \, \left(2+2^{q+2}\right)^{d} \,  2^{-(N+1) (q+1)} |f|_{W^{q+1,M,p}(\Omega)}, 
	\end{align*}
	where the last two inequalities follow from (A.11) in \cite{guo2016sparse}. Therefore, \eqref{eqn:relation} is proved. For the $H^1$ broken semi-norm, similarly we have
	\begin{align*}
		& |\cI^{P,M}_{N} [f] - \widehat{\cI}^{P,M}_{N} [f] | _{H^{1}(\Omega_{N})}  \\
		\leq &  \sum_{\substack{|\bn|_{\infty}\leq N, |\bn|_1 \geq N+1\\ \bn\in\mathbb{N}_0^d} } |\widetilde{\mathcal{I}}^{k}_{n_1,x_1} \circ \cdots \circ \widetilde{\mathcal{I}}^{k}_{n_d,x_d} [f] | _{H^{1}(\Omega_{N})}  \\
		\leq &  \sum_{\substack{|\bn|_{\infty}\leq N, |\bn|_1 \geq N+1\\ \bn\in\mathbb{N}_0^d} } \sqrt{d} \, \bar{C}^d (2+2^{q+2})^{\gamma} 2^{-|\bn|_{1}(q+1) + |\bn|_{\infty}}   \, |f|_{W^{q+1,M,2,L}(\Omega)}   \quad 
		\text{with}\, L= supp(\bn), \gamma=|L|\\
		\leq & \sqrt{d} \, \bar{C}^d (2+2^{q+2})^{d}  |f|_{W^{q+1,M,2,L}(\Omega)} \sum_{\substack{|\bn|_{\infty}\leq N, |\bn|_1 \geq N+1\\ \bn\in\mathbb{N}_0^d} } 2^{-|\bn|_{1}(q+1) + |\bn|_{\infty}}   \\
		\leq & \sqrt{d} \, \bar{C}^d (2+2^{q+2})^{d}  |f|_{W^{q+1,M,2,L}(\Omega)} d 2^{d-1}2^{-qN} , 
	\end{align*}
	where the last inequality is proved in \cite{guo2016sparse} as well.  \eqref{eqn:relation3} can be obtained.
	
	Next, we will prove \eqref{eqn:relation2} and \eqref{eqn:relation4}. Similarly, the error is splitted into two parts
	\begin{align*}
	f - \widehat{\cI}^{P,M}_{c,N}[f] = f - \cI^{P,M}_{N}[f] + \cI^{P,M}_{N}[f] - \widehat{\cI}^{P,M}_{c,N}[f],
	\end{align*}
	and
	\begin{align*}
	\cI^{P,M}_{N} [f] - \widehat{\cI}^{P,M}_{c,N}[f] 
	=& \sum_{\substack{|\bn|_1 \geq N-d+2 \\ \mathbf{-1} \leq \bn \leq \mathbf{N} }} \widetilde{\mathcal{I}}^{P,M}_{c,n_1,x_1} \circ \cdots \circ \widetilde{\mathcal{I}}^{P,M}_{c,n_d,x_d} [f] .
	\end{align*}
	Here, we denote $L_2=supp(\bn)=\{\beta_{1}, \cdots, \beta_{\gamma}\}$, and $L_1=\{ \alpha_{1}, \cdots, \alpha_{\theta}\} \subset \{1,\cdots,d\}$ that $n_{\alpha}=-1$ if $\alpha\in L_1$. Hence, we can prove that 
	\begin{align*}
	 \| \widetilde{\mathcal{I}}^{P,M}_{c,n_1,x_1} \circ \cdots \circ \widetilde{\mathcal{I}}^{P,M}_{c,n_d,x_d} [f] \| _{L^{p}(I^{\bj}_{\bn})}  
	\leq & \bar{C}^d (1+2^{q+1})^{\gamma} 2^{-n_{\beta_1}(q+1)} \cdots 2^{-n_{\beta_\gamma} (q+1)} |f|_{W^{q+1,M,p,L}(I_{\bn_L}^{\bj_L})} \notag \\
	=& \bar{C}^d (1+2^{q+1})^{\gamma} 2^{-(|\bn|_{1}+\theta) (q+1)}  |f|_{W^{q+1,M,p,L}(I_{\bn_L}^{\bj_L})}\\
	| \widetilde{\mathcal{I}}^{P,M}_{c,n_1,x_1} \circ \cdots \circ \widetilde{\mathcal{I}}^{P,M}_{c,n_d,x_d} [f] |^2 _{H^{1}(I^{\bj}_{\bn})}  
	=& \sum_{x_m\in L_1} | \partial_{x_m} \widetilde{\mathcal{I}}^{P,M}_{c,n_1,x_1} \circ \cdots \circ \widetilde{\mathcal{I}}^{P,M}_{c,n_d,x_d} [f] |^2 _{L^{2}(I^{\bj}_{\bn})} \\
	+ & \sum_{x_m\in L_2} | \partial_{x_m} \widetilde{\mathcal{I}}^{P,M}_{c,n_1,x_1} \circ \cdots \circ \widetilde{\mathcal{I}}^{P,M}_{c,n_d,x_d} [f] |^2 _{L^{2}(I^{\bj}_{\bn})} \\
	+ & \sum_{x_m \notin (L_1\cup L_2)} | \partial_{x_m} \widetilde{\mathcal{I}}^{P,M}_{c,n_1,x_1} \circ \cdots \circ \widetilde{\mathcal{I}}^{P,M}_{c,n_d,x_d} [f] |^2 _{L^{2}(I^{\bj}_{\bn})} \\
	\leq&  \sum_{x_m\in L_2}   (\bar{C})^{2d} (1+2^{q+1})^{2\gamma} \, 2^{-2(|\bn|_1+\theta) (q+1) +2|\bn|_{\infty} } |f|^2_{W^{q+1,M,2,L}(I_{\bn_L}^{\bj_L})} \\
	+&  \sum_{x_m\notin (L_1\cup L_2)}  (\bar{C})^{2d} (1+2^{q+1})^{2\gamma}  \, 2^{-2(|\bn|_1+\theta) (q+1)} |f|^2_{W^{q+1,M,2,L}(I_{\bn_L}^{\bj_L})} \\
	\leq& (d-\theta)\,\bar{C}^{2d} (1+2^{q+1})^{2\gamma} \, 2^{-2(|\bn|_1+\theta) (q+1) +2|\bn|_{\infty} } |f|^2_{W^{q+1,M,2,L}(I_{\bn_L}^{\bj_L})} ,
	\end{align*}
	and 
	\begin{align*}
	& \| \widetilde{\mathcal{I}}^{P,M}_{c,n_1,x_1} \circ \cdots \circ \widetilde{\mathcal{I}}^{P,M}_{c,n_d,x_d} [f] \| _{L^{p}(\Omega_{N})} 
	\leq   \bar{C}^d (2+2^{q+2})^{\gamma} 2^{-(|\bn|_{1}+\theta) (q+1)}   \, |f|_{W^{q+1,M,p,L}(\Omega)}, \\
	& | \widetilde{\mathcal{I}}^{P,M}_{c,n_1,x_1} \circ \cdots \circ \widetilde{\mathcal{I}}^{P,M}_{c,n_d,x_d} [f] |_{H^{1}(\Omega_{N})} 
	\leq   \sqrt{d-\theta}\, \bar{C}^{d} (2+2^{q+2})^{\gamma} 2^{-(|\bn|_{1}+\theta)(q+1) + |\bn|_{\infty}}   \, |f|_{W^{q+1,M,2,L}(\Omega)}.
	\end{align*}
	Consequently, 
	\begin{align*}
	 \|\cI^{P,M}_{N} [f] - \widehat{\cI}^{P,M}_{c,N}[f] \| _{L^{p}(\Omega_{N})}  
	\leq &  \sum_{\substack{|\bn|_1 \geq N-d+2\\ \mathbf{-1}\leq \bn\leq \mathbf{N}} } \|\widetilde{\mathcal{I}}^{k}_{c,n_1,x_1} \circ \cdots \circ \widetilde{\mathcal{I}}^{k}_{c,n_d,x_d} [f] \| _{L^{p}(\Omega_{N})}  \\
	\leq &  \sum_{\substack{|\bn|_1 \geq N-d+2\\ \mathbf{-1}\leq \bn\leq \mathbf{N}} } \bar{C}^d (2+2^{q+2})^{\gamma} 2^{-(|\bn|_{1}+\theta)(q+1)}   \, |f|_{W^{q+1,M,p,L}(\Omega)}   \\
	\leq &  \bar{C}^d (2+2^{q+2})^{d} 2^{-(N-d+2)(q+1)}   \, |f|_{W^{q+1,M,p}(\Omega)}  \sum_{\substack{|\bn|_1 \geq N-d+2\\ \mathbf{-1}\leq \bn\leq \mathbf{N}} } 2^{-\theta(q+1)} \\
 	|\cI^{P,M}_{N} [f] - \widehat{\cI}^{P,M}_{c,N}[f] | _{H^{1}(\Omega_{N})}  
 	\leq & \bar{C}^{d} (2+2^{q+2})^{d} \, |f|_{W^{q+1,M,2}(\Omega)}  \sum_{\substack{|\bn|_1 \geq N-d+2\\ \mathbf{-1}\leq \bn\leq \mathbf{N}} }  \sqrt{d-\theta}\, 2^{-(|\bn|_{1}+\theta)(q+1) + |\bn|_{\infty}}   .
	\end{align*}

	Note that 
	\begin{align*}
	\sum_{\substack{|\bn|_1 \geq N-d+2\\ \mathbf{-1}\leq \bn\leq \mathbf{N}} } 2^{-\theta(q+1)}
	=& \sum_{\theta=0}^{d} 2^{-\theta(q+1)} \left\{ \sum_{\substack{|\bn|_1 \geq N-d+2\\ \mathbf{-1}\leq \bn\leq \mathbf{N} \\ |L_1|=\theta} } 1  \right\}
	= \sum_{\theta=0}^{d} 2^{-\theta(q+1)} \left\{ \sum_{\substack{ \mathbf{-1}\leq \bn\leq \mathbf{N} \\ |L_1|=\theta } } 1 - \sum_{\substack{|\bn|_1 \leq N-d+1\\ \mathbf{-1}\leq \bn\leq \mathbf{N} \\ |L_1|=\theta} } 1 \right\} \\
	=& \sum_{\theta=0}^{d} 2^{-\theta(q+1)} \left\{ \sum_{\substack{ \mathbf{-1}\leq \bn\leq \mathbf{N} \\ |L_1|=\theta } } 1 - \sum_{s=-\theta}^{N-d+1} \sum_{\substack{|\bn|_1 =s \\ \mathbf{-1}\leq \bn\leq \mathbf{N} \\ |L_1|=\theta} } 1 \right\} \\
	\leq& \sum_{\theta=0}^{d} 2^{-\theta(q+1)} \begin{pmatrix}d \\ \theta\\ \end{pmatrix} (N+1)^{d-\theta} \\
	=&  \left(  N+1 + 2^{-(q+1)}  \right)^d.
	\end{align*}
	Therefore, 
	\begin{align*}
	\|\cI^{P,M}_{N} [f] - \widehat{\cI}^{P,M}_{c,N}[f] \| _{L^{p}(\Omega_{N})}  
	\leq &  \bar{C}^d (2+2^{q+2})^{d} \left( N+1+2^{-(q+1)} \right)^d 2^{-(N-d+2)(q+1)}   \, |f|_{W^{q+1,M,p}(\Omega)},  
	\end{align*}
	and \eqref{eqn:relation2} can be obtained.

	On the other hand, for any $\bn$, we define $\bn'=\{n_{m_1}, \ldots, n_{m_{d-\theta}} \}$, with $m_i\notin L_1$. Then $|\bn'|_{\infty}\geq |\bn|_{\infty}$ and $|\bn'|_{1} - \theta = |\bn|_{1}$. Furthermore,
	\begin{align*}
	\sum_{\substack{|\bn|_1 \geq N-d+2\\ \mathbf{-1}\leq \bn\leq \mathbf{N}} }  \sqrt{d-\theta}\, 2^{-(|\bn|_{1}+\theta)(q+1) + |\bn|_{\infty}} 
	\leq&  \sum_{\theta=0}^{d} \sqrt{d-\theta} \sum_{\substack{|\bn'|_1 \geq N-d+\theta+2\\ \mathbf{0}\leq \bn'\leq \mathbf{N} \\ |L_1|=\theta} }  2^{-|\bn'|_{1}(q+1) }  2^{|\bn'|_{\infty} }  \\
	=& \sum_{\theta=0}^{d} \sqrt{d-\theta}  \begin{pmatrix}d \\ \theta\\ \end{pmatrix} 
	\sum_{\substack{ s=N-d+\theta+2 \\ |\bn'|_1 =s \\ \mathbf{0}\leq \bn'\leq \mathbf{N} }}^{(d-\theta)N}  2^{-|\bn'|_1(q+1)}  2^{|\bn'|_{\infty}} \\
	=& \sum_{\theta=0}^{d} \sqrt{d-\theta}  \begin{pmatrix}d \\ \theta\\ \end{pmatrix} 
	\sum_{s=N-d+\theta+2}^{(d-\theta)N}  2^{-s(q+1)}
	\sum_{\substack{|\bn'|_1 =s \\ \bn'\in\mathbb{N}^{d-\theta}_{0} } }  2^{|\bn'|_{\infty}} \\
	\leq & \sum_{\theta=0}^{d} \sqrt{d-\theta}  \begin{pmatrix}d \\ \theta\\ \end{pmatrix} 
	\sum_{s=N-d+\theta+2}^{(d-\theta)N}  2^{-s(q+1)}  (d-\theta) 2^{d-\theta-1+s}\\
	\leq & \sum_{\theta=0}^{d} (d-\theta)^{3/2}	  \begin{pmatrix}d \\ \theta\\ \end{pmatrix} 
	\frac{2^{-q}}{1-2^{-q}}  2^{d-1}  2^{-q(N-d+1)}  2^{-\theta(q+1)} \\
	\leq& \sum_{\theta=0}^{d} (d-\theta)^{3/2}	  \begin{pmatrix}d \\ \theta\\ \end{pmatrix} 
	2^{d-1}  2^{-q(N-d+1)}   2^{-\theta(q+1)} \\
	\leq & d^{3/2} \, 2^{d-1} \, 2^{-q(N-d+1)}  \left( 1+2^{-(q+1)}\right)^d.
	\end{align*}
	The fourth line is based on the formula $\sum_{|\bn|_1=s, \bn\in \mathbb{N}_{0}^d} 2^{|\bn|_{\infty}}\leq d \, 2^{d-1+s}$ given in  \cite{schwab2008sparse}. This tells us that
	\begin{align*}
	|\cI^{P,M}_{N} [f] - \widehat{\cI}^{P,M}_{c,N}[f] | _{H^{1}(\Omega_{N})}  
	\leq & d^{3/2}  \, \bar{C}^{d} (2+2^{q+2})^{d} \, 2^{d-1} \left( 1+2^{-(q+1)}\right)^d \, 2^{-q(N-d+1)}  |f|_{W^{q+1,M,2}(\Omega)}  .
	\end{align*}


\bibliographystyle{abbrv}
\bibliography{ref_sparse_VM,refer,ref_cheng,ref_cheng_2,n}

\end{document}